\documentclass[11pt,a4paper]{amsart}
\usepackage[utf8]{inputenc}
\usepackage[english]{babel}
\usepackage{amsmath,amssymb,amsthm,amscd,graphicx,mathrsfs,url,dsfont,enumitem}
\usepackage[usenames,dvipsnames]{color}
\usepackage[colorlinks=true,linkcolor=Blue,citecolor=Green,urlcolor=Blue]{hyperref}
\usepackage{dsfont}
\usepackage[super]{nth}
\usepackage[open, openlevel=2, depth=3, atend]{bookmark}
\hypersetup{pdfstartview=XYZ}
\usepackage{epigraph}
\usepackage{mathtools}
\usepackage{enumitem}
\usepackage[mathscr]{eucal}
\usepackage{comment}

\let\oldtocsection=\tocsection
\let\oldtocsubsection=\tocsubsection
\let\oldtocsubsubsection=\tocsubsubsection
\renewcommand{\tocsection}[2]{\hspace{0em}\oldtocsection{#1}{#2}}
\renewcommand{\tocsubsection}[2]{\hspace{1em}\oldtocsubsection{#1}{#2}}
\renewcommand{\tocsubsubsection}[2]{\hspace{2em}\oldtocsubsubsection{#1}{#2}}

\def\?[#1]{\textbf{[#1]}\marginpar{\Large{\textbf{??}}}}

\DeclareMathSymbol{\eset}{\mathalpha}{AMSb}{"3F}     
\renewcommand{\d}{\mathrm{d}}             
\renewcommand{\emptyset}{\eset}

\numberwithin{equation}{section}

\newtheorem{theorem}{Theorem}[section]
\newtheorem{lemma}[theorem]{Lemma}
\newtheorem{proposition}[theorem]{Proposition}
\newtheorem{corollary}[theorem]{Corollary}

\theoremstyle{remark}
\newtheorem{remark}[theorem]{Remark}

\newtheorem{open}[theorem]{Open question}
\newtheorem{problem}[theorem]{Problem}
\theoremstyle{definition}
\newtheorem{definition}[theorem]{Definition}
\newtheorem{notation}[theorem]{Notation}

\newcommand{\C}{\mathbb{C}}
\newcommand{\D}{\mathbb{D}}
\newcommand{\R}{\mathbb{R}}
\newcommand{\Z}{\mathbb{Z}}
\renewcommand{\H}{\mathbb{H}}
\newcommand{\N}{\mathbb{N}}
\newcommand{\Q}{\mathbb{Q}}

\newcommand{\E}{\mathbb{E}}
\renewcommand{\P}{\mathbb{P}}
\renewcommand{\S}{\mathbb{S}}

\renewcommand{\Im}{\mathrm{Im}}
\renewcommand{\Re}{\mathrm{Re}}

\newcommand{\id}{\mathrm{Id}}

\newcommand{\cT}{\mathcal{T}}

\newcommand{\cE}{\mathcal{E}}
\newcommand{\cC}{\mathcal{C}}
\newcommand{\cH}{\mathcal{H}}

\newcommand{\cB}{\mathcal{B}}

\newcommand{\cA}{\mathcal{A}}
\newcommand{\cV}{\mathcal{V}}

\newcommand{\cP}{\mathcal{P}}
\newcommand{\ind}{\mathds{1}}

\newcommand{\cW}{\mathcal{W}}
\newcommand{\cD}{\mathcal{D}}

\newcommand{\cL}{\mathcal{L}}

\newcommand{\rR}{\mathrm{R}}

\newcommand{\cS}{\mathcal{S}}

\newcommand{\rL}{\mathrm{L}}

\renewcommand{\hat}{\widehat}

\newcommand{\del}{\partial}
\newcommand{\delbar}{\bar{\partial}}

\newcommand{\bc}{\mathbf{c}}

\newcommand{\norm}[1]{\left\Vert #1\right\Vert}

\newcommand{\cU}{\mathcal{U}}

\newcommand{\cQ}{\mathcal{Q}}

\newcommand{\bL}{\mathbf{L}}

\newcommand{\boldeta}{\boldsymbol{\eta}}
\newcommand{\ba}{\mathbf{a}}
\newcommand{\bH}{\mathbf{H}}

\newcommand{\A}{\mathbb{A}}

\newcommand{\cJ}{\mathcal{J}}

\newcommand{\rM}{\mathrm{m}}
\newcommand{\bN}{\mathbf{N}}

\renewcommand{\det}{\mathrm{det}}
\newcommand{\rv}{\mathrm{v}}
\newcommand{\rw}{\mathrm{w}}

\newcommand{\exc}{\mathrm{exc}}

\newcommand{\rmint}{\mathrm{int}}

\newcommand{\rmext}{\mathrm{ext}}

\newcommand{\bU}{\mathbf{U}}
\newcommand{\bk}{\mathbf{k}}
\renewcommand{\exp}{\mathrm{exp}}
\newcommand{\rB}{\mathrm{B}}

\newcommand{\MKS}{\mathrm{MKS}}

\newcommand{\inte}{\mathrm{int}}
\newcommand{\loopmeasure}{\mu^\mathrm{loop}}
\newcommand{\bubmeasure}{\mu^\mathrm{bub}}
\newcommand{\eps}{\varepsilon}
\newcommand{\indic}[1]{\mathbf{1}_{#1}}
\newcommand{\ann}{O}

\author{Guillaume Baverez}
\address{Beijing International Center for Mathematical Research, Peking University}
\email{guillaume.baverez@bicmr.pku.edu.cn}

\author{Antoine Jego}
\address{CNRS \& CEREMADE, Université Paris-Dauphine, PSL University, France}
\email{antoinejego@hotmail.fr}

\title[SLE, CFT, and MKS measures]{The CFT of SLE loop measures and\\the Kontsevich--Suhov conjecture}

\date{}


\begin{document}

\begin{abstract}
This paper initiates the study of the conformal field theory of the SLE$_\kappa$ loop measure $\nu$ for $\kappa\in(0,4]$, the range where the loop is almost surely simple. First, we construct two commuting representations $(\bL_n,\bar{\bL}_n)_{n\in\Z}$ of the Virasoro algebra with central charge $c_\rM=1-6(\frac{2}{\sqrt{\kappa}}-\frac{\sqrt{\kappa}}{2})^2\leq1$ as (unbounded) first order differential operators on $L^2(\nu)$. Second, we introduce highest-weight representations and characterise their structure: in particular, we prove the existence of vanishing singular vectors at arbitrary levels on the Kac table. Third, we prove an integration by parts formula for the SLE loop measure, and use it to define the Shapovalov form of the representation, a non degenerate (but \emph{not} positive definite) Hermitian form $\cQ$ on $L^2(\nu)$ with a remarkably simple geometric expression. The fact that $\cQ$ differs from the $L^2(\nu)$-inner product is a manifestation of non-unitarity. Finally, we write down a spectral resolution of $\cQ$ using the joint diagonalisation of $\bL_0$ and $\bar{\bL}_0$.
As an application of these results, we provide the first proof of the uniqueness of restriction measures, as conjectured by Kontsevich and Suhov.

Our results lay the groundwork for an in-depth study of the CFT of SLE: in forthcoming works, we will define correlation functions on Riemann surfaces, and prove conformal Ward identities, BPZ equations, and conformal bootstrap formulas.
\end{abstract}

\maketitle

\setcounter{tocdepth}{1}
\tableofcontents

\section{Introduction and main results}

    \subsection{Background and overview}
In 2000, Schramm \cite{Schramm2000} introduced a family of random fractal curves in the plane indexed by a positive real $\kappa$, now known as Schramm--Loewner Evolutions (SLE). These curves are conjectured (and sometimes proved) to describe the scaling limits of interfaces of spin clusters in critical 2D statistical mechanics models. This major breakthrough led to a tremendous increase in the rigorous understanding of the scaling limit of these systems. On the other hand, 2D critical systems have been studied through the lens of conformal field theory (CFT) since the groundbreaking work of Belavin, Polyakov and Zamolodchikov \cite{BPZ84}. The main assumption of this approach is that scaling limits of these models exhibit local conformal invariance, and the algebraic constraints implied by this infinite dimensional symmetry algebra drastically limit the possible theories existing in the limit. This idea was also present in Schramm's work: SLE was characterised from the start as the unique family of measures on curves satisfying conformal invariance and a certain domain Markov property. In \cite{LSW03_restriction}, Lawler, Schramm and Werner studied the ``conformal restriction" property of chordal SLE, relating SLE in a domain to SLE in its subdomains. Kontsevich and Suhov then asked whether this property gave another characterisation of the measure; more precisely, they conjectured the existence and uniqueness of a measure on Jordan loops satisfying the restriction property \cite[Conjecture 1]{KontsevichSuhov07}, in the range of parameters corresponding to $\kappa\in(0,4]$. While the existence was settled in previous works \cite{Werner08_loop,BenoistDubedat16,kemppainen2016nested,Zhan21_SLEloop}, Theorem \ref{thm:uniqueness} of this article gives a positive answer to the uniqueness part. 

In an attempt to reconcile probability and algebra, the CFT aspects of SLE have been studied since the inception of the theory, mostly focusing on the chordal case. The so-called partition functions of SLE can be interpreted as CFT correlation functions \cite{MR2000927, BauerBernard03, MR2187598, MR2337475, MR2358649, MR2518970, Dubedat15, Dubedat2015_fusion, MR3528421, MR3922531, feng2024multiple}, see \cite{Peltola19_review} for a thorough review,
but the mathematical framework of such a CFT remains only partially defined. As explained in \cite{Peltola19_review}, the most unsatisfactory aspect is the absence of mathematically well-defined ``fields" whose expectations would correspond to the SLE partition function. For instance, while SLE partition functions satisfy null-vector/BPZ equations, it is not known whether this corresponds to the insertion of a degenerate field in a correlation function. Our Theorem \ref{thm:module_structure} gives a first step toward a positive answer to this question (the second step --- the conformal Ward identities --- will be the topic of a follow-up work). 

In this paper, we consider the loop version of SLE \cite{kemppainen2016nested,Zhan21_SLEloop} (see also \cite{Werner08_loop,BenoistDubedat16} for the special cases $\kappa=\frac{8}{3},2$), and introduce the main ingredients that will turn the SLE loop measure into a full-fledged CFT. In a nutshell, we construct a structure $(L^2(\nu),(\bL_n,\bar{\bL}_n)_{n\in\Z},\cQ)$, where $\nu$ is the SLE$_\kappa$ loop measure ($\kappa\in(0,4]$), $(\bL_n,\bar{\bL}_n)_{n\in\Z}$ are two commuting representations of the Virasoro algebra with \emph{central charge} $c_\rM=1-6(\frac{2}{\sqrt{\kappa}}-\frac{\sqrt{\kappa}}{2})^2\leq1$ acting as (unbounded) operators on $L^2(\nu)$, and $\cQ$ is a Hermitian form on $L^2(\nu)$ such that $\bL_{-n}$ is the $\cQ$-adjoint of $\bL_n$ for all $n\in\Z$. In representation theoretic terms, $\cQ$ is called the \emph{Shapovalov form}. It turns out to have a very natural geometric interpretation, consisting in ``exchanging" the interior and the exterior of the curve. Our representations are constructed as follows: the Witt algebra defines invariant vector fields on the space of Jordan curves, and the operators are Lie derivatives in these directions. Representations of this flavour (but with substantial differences described in detail in Section \ref{SS:intro_Virasoro}) have been considered in \cite{KirillovYurev87,KirillovYurev88, BauerBernard03, MR2337475}, seeking an analogue of the Borel--Weil theorem for orbits of the group of diffeomorphisms of the circle. Much closer to our approach are the representations considered in \cite{ChavezPickrell14,GQW24}, inspired by the variational formulas of \cite{DurenSchiffer62}.

The Hermitian form $\cQ$ and the $\cQ$-adjoint relations are obtained via an integration by parts formula, which is an infinitesimal version of the conformal restriction property. Interestingly, the integration by parts is precisely the one (formally) satisfied by the path integral with action functional given by the universal Liouville action (a K\"ahler potential for the Weil-Petersson metric on the universal Teichm\"uller space \cite{TakhtajanTeo06}). This formula gives additional content to the idea -- initially motivated by semi-classical considerations \cite{ViklundWang19} (see also \cite{carfagnini2023onsager}) -- that the SLE loop is the correct interpretation of this path integral. The joint spectral theory of $\bL_0$ and $\bar{\bL}_0$ allows us to construct highest-weight representations indexed by a scalar $\lambda$ called the \emph{conformal weight}. We prove that these representations are irreducible (Theorem \ref{thm:module_structure}), giving important information when the conformal weight lies in the Kac table: in this case, there exist linear relations between certain states (e.g. Corollary \ref{C:level2}), a fact that will be used in a later work to prove the celebrated BPZ equations. Finally, we provide a spectral resolution of the Shapovalov form (Theorem \ref{T:spectral}), a crucial ingredient in our proof of the uniqueness part of Kontsevich--Suhov conjecture.

One of the most important aspects of our work is that the Shapovalov form is distinct from the $L^2(\nu)$-inner product, which is a symptom of the \emph{non-unitarity} of the theory. In fact, it is a standard result from the representation theory of the Virasoro algebra that the Shapovalov form is not positive definite for $c_\rM\leq1$ (except for a discrete set of values). In physics, non-unitary theories and their relations to loop models have been studied quite intensively over the past few years, mostly by algebraic methods and consistency conditions (crossing symmetry, conformal bootstrap, null-vector equations...) \cite{RibaultSantachiara15,PiccoRibault16_perco,HeJacobsenSaleur20_potts,JacobsenRibaultSaleur22_On}. However, contrary to the regime of central charge $\geq25$, there seems to be a zoo of possible theories when the central charge is less than 1. This underlines the importance of studying concrete representations like we do in this paper. In the probabilistic literature, a closely related model has been studied with a different point of view \cite{AngSun21_CLE,AngRemySun22_annuli}, namely using a coupling of SLE with Liouville CFT \cite{DKRV16,KRV_DOZZ,GKRV20_bootstrap} inspired by Sheffield's quantum zipper \cite{sheffield2016} (see also the comprehensive book \cite{berestycki2024gaussian} for a detailed and gentle introduction to this difficult topic). In fact, the generating function of the ``electrical thickness" computed in \cite{AngSun21_CLE} is the partition function of our Shapovalov form, in the sense that it is the norm of the top vector in a highest-weight representation. This type of coupling also appears in theoretical physics: \cite[(2.20)]{Eberhardt23_crossing} proposes an inner product on the space of Virasoro conformal blocks at central charge $c_\mathrm{L}\geq25$ by reweighting with the partition function of a ``matter CFT" with central charge $c_\rM=26-c_\mathrm{L}$, and the ``ghost partition function" of central charge $-26$. The putative matter CFT behind this partition function is not clearly identified there, but is conjectured to exist for all values of the central charge $c_\rM\leq 1$, with a continuous spectrum of representations \cite[Conjecture 1]{Eberhardt23_crossing}. These two properties are compatible with our results, and make SLE a natural candidate. In fact, our recent work \cite{BJ25} gives content to this idea by showing how the ghost field appears in a new approach to the conformal welding of quantum surfaces \cite{sheffield2016,AHS20}. Finally, we refer the reader to Section~\ref{sec:future} for additional research directions in relation to conformal field theory and SLE.

\smallskip

In the rest of this introduction, we will describe precisely our main results. Our CFT results require a precise setup and notations, introduced in Section \ref{SS:setup}. Section \ref{subsec:uniqueness} on uniqueness of MKS measures can be read independently. 
Future perspectives on the programme initiated in the current paper are described in Section \ref{sec:future}.

\smallskip

\noindent \textbf{Acknowledgments}
We are particularly grateful to E. Peltola for numerous discussions on SLE and a detailed feedback on this manuscript. We thank E. Peltola and B. Wu for bringing the question of uniqueness of MKS measures to our attention, and W. Werner for pointing out the relevance of our results to the reversibility and the duality of SLE. We also thank V. Vargas for discussions on SLE and the conformal bootstrap.
Finally, we are grateful to the referees for their careful readings and suggestions.

Some of this work has been achieved during the fantastic conference ``Probability in CFT" organised by J. Aru, C. Guillarmou and R. Rhodes, at the Bernoulli Center, EPFL, whose hospitality is greatly acknowledged.
AJ was supported by Eccellenza grant 194648 of the Swiss National Science Foundation and was a member of NCCR SwissMAP when this article was written up. GB acknowledges support from ANR-21-CE40-0003 CONFICA.

\smallskip

\textit{During the final stages of preparation of the current article, we learnt that M. Gordina, W. Qian and Y. Wang were independently working on a related article \cite{GQW24}. 
Specifically, they construct representations of the Witt and Virasoro algebras on $L^2(\nu)$ which would roughly correspond to our Theorems \ref{T:intro_witt} and \ref{T:intro_virasoro} below. They also derive an integration by parts formula and construct the Shapovalov form, which would correspond to Theorem \ref{T:intro_shapo} below. Their approach partly differs from ours, and we invite the reader to check it out in details. We thank them for bringing the references \cite{DurenSchiffer62,ChavezPickrell14} to our knowledge.}

    \subsection{Uniqueness of restriction measures}\label{subsec:uniqueness}
    

In this section, we state the uniqueness of restriction measures and outline the main steps of the proof. Before doing so, we need to introduce some relevant definitions.

\textbf{Space of Jordan curves.} 
Let $\cJ$  be the space of Jordan curves in $\hat\C$:
\begin{equation}\label{E:def_cJ}
\begin{aligned}
    & \cJ = \{ \eta = u(\S^1), u : \S^1 \to \hat\C \text{ continuous and injective} \}
    \qquad \text{and} \\
    & \hspace{40pt}\cJ_{0,\infty}  = \{ \eta \in \cJ: \eta \text{ disconnects 0 from } \infty \}.
\end{aligned}
\end{equation}
We endow $\cJ$ and $\cJ_{0,\infty}$ with the Hausdorff topology (induced by the Hausdorff metric on compact subsets of $\hat\C$). For each $\eta\in\cJ_{0,\infty}$, we denote by $\mathrm{int}(\eta)$ (resp. $\mathrm{ext}(\eta)$) the connected component of $\hat\C\setminus\eta$ containing $0$ (resp. $\infty$). We call these sets the interior and exterior of $\eta$ respectively. By the Jordan curve theorem, $\mathrm{int}(\eta)$ and $\mathrm{ext}(\eta)$ are homeomorphic to $\D$, and $\hat\C\setminus\eta=\mathrm{int}(\eta)\sqcup\mathrm{ext}(\eta)$.

\textbf{Brownian loop measure.}
As introduced by Lawler--Werner \cite{lawler2004brownian}, the whole-plane Brownian loop measure is a measure on Brownian-type trajectories defined by
\begin{equation}
    \label{E:def_BLM}
\loopmeasure_\C = \int_0^\infty \frac{\d t}{t} \int_\C |\d z|^2 \frac{1}{2\pi t} \P_\C^{t,z,z}
\end{equation}
where $\P_\C^{t,z,z}$ is the Brownian bridge measure from $z$ to $z$ with duration $t$ (in our convention, the associated generator is $\frac12 \Delta$). For a domain $D \subset \C$, one defines the Brownian loop measure in $D$ by restricting the whole plane measure to $D$:
\begin{equation}
    \label{E:BLM_restriction}
\loopmeasure_D(\d \cP) = \indic{\cP \subset D} \loopmeasure_\C(\d \cP).
\end{equation}
For any Borel sets $A,B \subset D$, we will denote by $\Lambda_D(A,B)$ the $\loopmeasure_D$-mass of loops touching both $A$ and $B$.
When $\partial D$ is non polar and $A$ and $B$ are disjoint closed sets, $\Lambda_D(A,B)$ is a finite quantity \cite[Lemma~2.4]{field2013reversed}. If $D = \hat{\C}$ and $A$ and $B$ are disjoint non polar subsets, this quantity is infinite but can be renormalised as follows (see \cite[Theorem 1.1]{field2013reversed}):
\begin{equation}
    \label{E:Lambda*}
    \Lambda^*(A,B) := \lim_{r \to 0} \Lambda_{O_r}(A,B) - \log |\log r|,
\end{equation}
where $O_r$ is the Riemann sphere minus a ball of radius $r$ (the resulting quantity does not depend on the precise centre of the ball).

\begin{remark}
Although this will not be used in the current paper, let us recall that, when $A$ and $B$ are two disjoint Jordan curves, which we assume in the rest of this remark, $\Lambda^*(A,B)$ can be expressed in terms of natural geometric quantities.
A proof of this identity can be found in \cite[Proposition 3.1]{LuoMaibach}, based on the observation of Dubédat \cite[Proposition 2.1]{Dubedat_SleFreeField}. The curves $A$ and $B$ cut the sphere $\hat\C$ into three regions: two simply connected domains, that we denote by $D_1$ and $D_2$, and an annular region, that we denote by $\mathrm{Ann}$. Then, there exists a universal constant $c_\diamond \in \R$ such that for any conformally flat metric $g$ on $\hat\C$,
\[ 
\Lambda^*(A,B) = \log \Big( \frac{\mathrm{vol}_g(\hat\C)}{\det'_\zeta(\Delta_{\hat\C,g})}\frac{\det_\zeta(\Delta_{D_1,g})\det_\zeta(\Delta_{D_2,g})}{\det_\zeta(\Delta_{\mathrm{Ann,g}})} \Big) + c_\diamond.
\]
In the above display, $\Delta_{D,g}$ denotes the Laplace operator in $D$ with respect to the background metric $g$, and $\det_\zeta(\Delta_{D,g})$ is its zeta-regularised determinant. In the case of the sphere, $\det'_\zeta(\Delta_{\hat\C,g})$ stands for this determinant where the zero eigenvalue has been removed. See \cite[Appendix~A.3]{LuoMaibach} for more on these notions.
\end{remark}

 \begin{definition}[Restriction measures]\label{D:restriction_measure}
 Fix $c_\rM \le 1$.
 Let $\nu$ be a Borel measure on the space $\cJ$ (resp. $\cJ_{0,\infty}$) of Jordan curves in the Riemann sphere (resp. separating $0$ from $\infty$), that is non-zero and locally finite. For any simply connected domain $D$ (with the disc topology), denote by
 \begin{equation}\label{E:conf_restriction}
 \d\nu_D(\eta):=\ind_{\eta \subset D} e^{\frac{c_\rM}{2}\Lambda^*(\eta,\del D)}\d\nu(\eta).
 \end{equation}
 We say that $\nu$ is a \textit{restriction measure} on $\cJ$ (resp. a \textit{weak restriction measure} on $\cJ_{0,\infty}$), with central charge $c_\rM$, if for any simply connected domains $D_1, D_2$ (resp. both containing $0$ and included in $\hat{\C}\setminus \{\infty\}$) and any biholomorphic map $\Phi:D_1 \to D_2$, $\nu_{D_2} = \Phi^*\nu_{D_1}$.
 \end{definition}

Note that, in the above definition of weak restriction measures on $\cJ_{0,\infty}$, there is no assumption of invariance under conformal maps which exchange the roles of $0$ and $\infty$. The word ``weak'' refers to this weak notion of conformal invariance. This will be used in Section \ref{S:reversibility} in order to give a new proof of reversibility of SLE.

Restriction measures on $\cJ$, also referred to as Malliavin--Kontsevich--Suhov measures (or simply MKS measures), have a long history. In \cite{AiraultMalliavin01}, Airault and Malliavin introduced the notion of \emph{unitarising measures} for the Virasoro algebra from an axiomatic point of view, leaving the construction of such measures open. Interestingly, this article was almost concomitant with Schramm's introduction of SLE \cite{Schramm2000}, and Lawler--Schramm--Werner soon showed the restriction property for chordal SLE \cite{LSW03_restriction}. Then, Kontsevich--Suhov related the restriction property to unitarising measures and conjectured that, for each central charge $c_\rM\leq1$, there exists a unique (up to multiplicative constant) measure on Jordan curves satisfying the properties of Definition~\ref{D:restriction_measure} \cite[Conjecture 1]{KontsevichSuhov07}. We stress that the link between restriction measures and unitarising measures is a rather subtle one, and refer to Section \ref{subsec:comparison} for a detailed comparison. 

The construction of the loop version of restriction measures was first done in the special cases $c_\rM=0$ \cite{Werner08_loop} and $c_\rM=-2$ \cite{BenoistDubedat16}.
A construction for the range $c_\rM \in (0,1]$ was later provided by Kemppainen and Werner \cite{kemppainen2016nested} based on the Conformal Loop Ensemble (counting measure of nested CLE, see also \cite[Theorem 1.1]{AngCaiSunWu}). Finally, Zhan's construction \cite{Zhan21_SLEloop} of the SLE loop measure covers the existence part of Kontsevich--Suhov's conjecture in the whole region $c_\rM\leq1$. In the range $c_\rM\in(0,1]$, Zhan's construction is closely related to the argument of Kemppainen--Werner showing reversibility \cite{kemppainen2016nested}, and indeed the two measures differ only by a multiplicative constant \cite[Theorem 2.18]{AngSun21_CLE}.
The fact that the measure is a Borel measure is verified in Proposition~\ref{prop:cSLE_is_borel} of the current paper.
We also mention that the articles \cite{Benoist16,MaibachPeltola24} study more general measures where the conformal restriction anomaly in \eqref{E:conf_restriction} is not necessarily given by the Brownian loop measure.
Here, we show the uniqueness part for loops in the sphere.

\begin{theorem}[Uniqueness of restriction measures]\label{thm:uniqueness}
Let $c_\rM\leq1$. Up to a positive multiplicative constant, there exists a unique restriction measure on $\cJ$ (resp. weak restriction measure on $\cJ_{0,\infty}$), with central charge $c_\rM$. This measure is the SLE$_\kappa$ loop measure (resp. the SLE$_\kappa$ loop measure restricted to the space $\cJ_{0,\infty}$), for the unique $\kappa\in(0,4]$ such that $c_\rM=1-6(\frac{2}{\sqrt{\kappa}}-\frac{\sqrt{\kappa}}{2})^2$.
\end{theorem}

As explained by Kontsevich and Suhov \cite[Section 2.5]{KontsevichSuhov07}, the problem of existence and uniqueness of MKS measures on general Riemann surfaces equipped with a projective structure (see \cite{BenoistDubedat16} for more details on this general setting of Riemann surfaces) actually reduces to the case of the Riemann sphere. In higher genus, \cite{KontsevichSuhov07} use the more complicated setting of measures with values in determinant line bundles, and we refrain from making a precise statement here. 

We mention that uniqueness was already known in the special case $c_\rM =0$ ($\kappa=8/3$); see \cite{Werner08_loop,ChavezPickrell14}. We also point out that \cite{BenoistDubedat16} shows the existence of a restriction measure with central charge $c_\rM=-2$ as the scaling limit of a discrete model, and uses an independent argument (relying on the convergence of the uniform spanning tree to SLE$_8$ and the duality $\kappa\leftrightarrow\frac{16}{\kappa}$) to prove that this measure is locally absolutely continuous with respect to SLE$_2$. Our theorem implies that this measure \emph{is} in fact the SLE$_2$ loop measure. In Section \ref{S:applications}, we discuss two additional consequences of this uniqueness result: the reversibility and the duality of SLE loop measures.

The main difficulty in proving this uniqueness statement is that there is no way \textit{a priori} to compare two restriction measures (they need not be absolutely continuous with respect to one another). As a matter of comparison, Lawler \cite{MR2518970, lawler2011defining} uses the restriction property to define SLE in multiply connected topologies as absolutely continuous measures with respect to the known SLE measure in simply connected domains. Uniqueness is then much clearer in this type of context.

\textit{Update: After the first version of this paper was released, G. Cai and Y. Gao \cite{arXiv:2502.05890} gave another proof of Theorem \ref{thm:uniqueness} in the case where $c_\rM \in (0,1]$.}

    \textbf{Proof strategy.}
We illustrate our broad strategy of proof with an elementary toy model; see \cite[Section 2]{Malliavin_Gaussian} for more perspectives on this classical toy model. The Gaussian distribution $\mu$, whose density with respect to Lebesgue measure on $\R$ is given by $x\in\R\mapsto (2\pi)^{-1/2} e^{-x^2/2}$, satisfies the following integration by parts formula: for any polynomials $f,g : \R \to \R$,
\begin{equation}
    \label{E:Gaussian_IBP} 
\int_\R f'(x) g(x) \mu(\d x) = \int_\R f(x) (xg(x)-g'(x)) \mu(\d x).
\end{equation}
This identity entirely characterises the probability measure $\mu$. A proof of this claim related to our proof of Theorem \ref{thm:uniqueness} goes as follows. Consider the Hermite polynomials defined inductively by:
\[
H_0 = 1, \qquad H_{n+1} = XH_n-H_n', \quad n \ge 0.
\]
Plugging $f=1$ and $g=H_n$ in \eqref{E:Gaussian_IBP} shows that for all $n \ge 0$, $\int_\R H_{n+1}(x) \mu(\d x) = 0$. Since the Hermite polynomials span all the polynomials,
one has in particular computed all the moments $\int_\R x^n \mu(\d x)$, $n \ge 0$, of $\mu$ and showed that they agree with the moments of the Gaussian distribution. This shows that the only probability measure satisfying \eqref{E:Gaussian_IBP} is the Gaussian distribution.

\medskip

The proof of Theorem \ref{thm:uniqueness} will follow a similar strategy, the rest of the introduction presenting each main step.
Let $c_\rM \le 1$ and fix a weak restriction measure $\nu$ on $\cJ_{0,\infty}$ with central charge $c_\rM$. To show that $\nu$ is a multiple of the SLE loop measure, we will:
\begin{itemize}
    \item Define a notion of derivative in the space $\cJ_{0,\infty}$ of Jordan curves;
    \item Establish an integration by parts formula with respect to $\nu$;
    \item Construct ``polynomials'' out of this formula and study their structure to guarantee that they span all polynomials;
    \item Compute the $L^2(\nu)$-inner product on these polynomials and conclude by a density argument.
\end{itemize}


\subsection{Setup}\label{SS:setup}

Recall from \eqref{E:def_cJ} that we denote by $\cJ_{0,\infty}$ the space of Jordan curves disconnecting 0 from $\infty$, endowed with the topology induced by the Hausdorff metric. 
It will be useful to parametrise the space $\cJ_{0,\infty}$ using conformal maps as follows.
Let 
\begin{align}\label{E:def_cE}
    \cE = \Big\{ \begin{array}{l} f : \D \to \C \text{ holomorphic with } f(0)=0, f'(0)=1 \text{ and} \\ \text{extending continuously and injectively to } \bar\D \text{ with } f(\bar\D) \subset \C
    \end{array} \Big\}.
\end{align}
Similarly and denoting $\D^* = \hat{\C} \setminus \overline{\D}$, let
\begin{align}\label{E:def_cE_tilde}
    \tilde\cE = \Big\{ \begin{array}{l} g : \D^* \to \hat\C\setminus\{0\} \text{ holomorphic with } g(\infty)=\infty, g'(\infty)=1 \text{ and} \\ \text{extending continuously and injectively to } \overline{\D^*} \text{ with } g(\overline{\D^*}) \subset \hat\C \setminus \{0\}
    \end{array} \Big\}.
\end{align}
Endow $\cE$ and $\tilde \cE$ with the local uniform topology on $\D$ and $\D^*$ respectively.
Every curve $\eta$ in $\cJ_{0,\infty}$ can be uniquely represented by a pair $(c_\eta,f_\eta)\in\R\times\cE$ (resp. a pair $(\tilde{c}_\eta,g_\eta)\in\R\times\tilde{\cE}$), where $e^{c_\eta}f_\eta$ uniformises the interior of $\eta$ (resp. $e^{-\tilde{c}_\eta}g_\eta$ uniformises the exterior of $\eta$). See Figure \ref{fig:setup}. This identifies $\cJ_{0,\infty}$ with $\R \times \cE$ and $\R \times \tilde \cE$.
These identifications are not homeomorphisms, but are nevertheless Borel isomorphisms, i.e. this defines a unique $\sigma$-algebra on $\cJ_{0,\infty}$. See Appendix \ref{app:SLE} for more details.

Every $f \in \cE$ and $g \in \tilde \cE$ have series expansions in the neighbourhoods of $0$ and $\infty$ respectively:
\begin{equation}\label{eq:coordinates}
f(z)=z\Big(1+\sum_{m=1}^\infty a_mz^m\Big)
\qquad \text{and} \qquad
g(z)=z\Big(1+\sum_{m=1}^\infty b_mz^{-m}\Big).
\end{equation}
We will think of $(a_m)_{m\geq1}$ as (complex analytic) coordinates on $\cE$. That is, we think of $c$, $f$, $(a_m)$, $\tilde c$, $g$, $(b_m)$ as functions defined on $\cJ_{0,\infty}$, e.g. $c: \eta \in \cJ_{0,\infty}\mapsto c_\eta \in \R$. When the context is clear, we will make the dependence on $\eta$ implicit.

\begin{figure}
\centering
\includegraphics[scale=1]{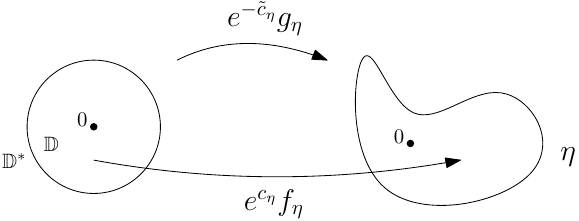}
\caption{\label{fig:setup}A Jordan curve $\eta$ having 0 in its interior and $\infty$ in its exterior. It is uniquely represented by a pair $(c_\eta,f_\eta)\in\R\times\cE$ such that $e^{c_\eta}f_\eta$ uniformises the interior (or a pair $(\tilde{c}_\eta,g_\eta)\in\R\times\tilde{\cE}$ such that $e^{-\tilde{c}_\eta}g_\eta$ uniformises the exterior). This defines homeomorphisms $\cJ_{0,\infty}\simeq\R\times\cE\simeq\R\times\tilde{\cE}$.}
\end{figure}

We denote by $\nu$ the SLE$_\kappa$ loop measure of \cite{Zhan21_SLEloop} restricted to $\cJ_{0,\infty}$, which is a scale-invariant Borel measure on $\cJ_{0,\infty}$ (Proposition \ref{prop:sle_is_borel}).
In this whole paper, $\kappa$ is fixed and belongs to $(0,4]$, the range corresponding to simple loops and the corresponding central charge is fixed by
\begin{equation}\label{E:central_charge}
    c_\rM = 1-6\Big(\frac{2}{\sqrt{\kappa}}-\frac{\sqrt{\kappa}}{2}\Big)^2.
\end{equation}
By scale invariance, the marginal law of $\nu$ on the $c$-coordinate is a multiple of Lebesgue measure on $\R$. The multiplicative constant is irrelevant for the current paper, and it will be convenient to fix it at $1/\sqrt{\pi}$.
Again by scale invariance, the marginal law of $\nu$ on the $f$-coordinate, conditionally on $c$, does not depend on the value of $c$. This gives a probability measure on $\cE$ that we will denote by $\nu^\#$. A variant of $\nu^\#$ is called the \emph{shape measure} in \cite{AngCaiSunWu}. Altogether, we have decomposed
\begin{equation}
    \label{E:SLE_shape}
    \nu=\frac{1}{\sqrt{\pi}}\d c\otimes\nu^\#.
\end{equation}

	\subsection{Virasoro representations on \texorpdfstring{$L^2(\nu)$}{L2(nu)}}\label{SS:intro_Virasoro}

We now present our results concerning Witt and Virasoro representations on the space of functions defined on $\cJ_{0,\infty}$, or more precisely on $L^2(\nu)$.
These results go beyond the SLE setting and are of general interest since they primarily concern notions of derivatives in the space of Jordan curves.
See below Theorem~\ref{T:intro_virasoro} for a comparison with earlier works.

 The \emph{Witt algebra} is the space of Laurent polynomial vector fields $\C(z)\del_z$. It is a Lie algebra with bracket
\begin{equation}
    \label{E:bracket_vector_field}
[v\del_z,w\del_z]=(vw'-v'w)\del_z.
\end{equation}
We will use the notation $\rv_n=-z^{n+1}\del_z$ for all $n\in\Z$. The commutation relations on these generators are $[\rv_n,\rv_m]=(n-m)\rv_{n+m}$. Let $\rv=v\del_z\in\C(z)\del_z$. The vector field $\rv$ generates a family of conformal transformations $\phi_t$ solution to the differential equation
\[\del_t\phi_t=v\circ\phi_t,\qquad\phi_0=\mathrm{id}.\]
More precisely, for all set $K$ compactly included in $\C\setminus\{0\}$, there exists a complex neighbourhood $O$ of the origin such that for all $t\in O$, $\phi_t$ is well defined on $K$.
Hence, for every $\eta\in\cJ_{0,\infty}$, we can consider the small motion
\begin{equation}\label{E:phi_cdot_eta}
\phi_t(\eta)=\{\phi_t(z):z \in \eta\}.
\end{equation}
The condition that $v$ is a Laurent polynomial can be weakened to any holomorphic $v$ defined in a neighbourhood of the trace of the curve (in which case the flow $(\phi_t)_t$ would still be well defined for small enough times, in a neighbourhood of the curve). 

In \eqref{E:cC} and \eqref{E:cC_compact}, we will introduce two dense subspaces $\cC$ and $\cC_\mathrm{comp}$ of $L^2(\nu)$. These spaces are essentially composed of polynomials in the variables $a_m$ and $\bar a_m$, $m\ge 1$, from the expansion \eqref{eq:coordinates} of $f$.
In  Proposition~\ref{P:construction_cL}, we will show that for any vector field $\rv \in \C(z)\del_z$, there exist endomorphisms $\cL_\rv, \bar\cL_\rv : \cC \to \cC$ such that for any $F \in \cC$ and $\eta \in \cJ_{0,\infty}$,
 \[F(\phi_t(\eta))=F(\eta)+t\cL_\rv F(\eta)+\bar{t}\bar{\cL}_\rv F(\eta)+o(t),
 \qquad \text{as} \quad t \to 0.\]
In other words, $\cL_\rv$ and $\bar{\cL}_\rv$ are the $(1,0)$- and $(0,1)$-parts of the Lie derivative in direction $\rv$.
We emphasise that the $o(t)$ in the above display is simply a complex number which vanishes as $t\to 0$.

Our first main result concerns the study of these differential operators as operators on $L^2(\nu)$.

\begin{theorem}[Witt representation]\label{T:intro_witt}
For all $\rv \in \C(z)\del_z$,
the densely defined operators $\cL_\rv,\bar{\cL}_\rv:\cC\to L^2(\nu)$ are closable and preserve $\cC$. They form two commuting representations of the Witt algebra: for all $\rv,\rw\in\C(z)\del_z$, we have
\[[\cL_\rv,\cL_\rw]=\cL_{[\rv,\rw]};\qquad[\bar{\cL}_\rv,\bar{\cL}_\rw]=\bar{\cL}_{[\rv,\rw]};\qquad[\cL_\rv,\bar{\cL}_\rw]=0.\]
In the sequel, we will write $\cL_n:=\cL_{\rv_n}$, $\bar{\cL}_n:=\bar{\cL}_{\rv_n}$, $n \in \Z$.
\end{theorem}

In our next result,
we modify the Witt representation to a Virasoro representation with central charge $c_\rM$ \eqref{E:central_charge}.
To do so,
let us introduce the ``differential one-forms" $\vartheta, \tilde \vartheta:\cJ_{0,\infty}\times\C(z)\del_z\to\C$ defined for $\rv = v(z) \partial_z \in\C(z)\del_z$ and $\eta \in \cJ_{0,\infty}$ by:
\begin{equation}\label{eq:def_vartheta}
\vartheta_\eta(\rv)=\frac{e^{-c_\eta}}{2i\pi}\oint\cS f_\eta^{-1}(z)v(e^{c_\eta}z)\d z;\qquad\tilde{\vartheta}_\eta(\rv)=\frac{e^{\tilde{c}_\eta}}{2i\pi}\oint\cS g_\eta^{-1}(z)v(e^{-\tilde{c}_\eta}z)\d z,
\end{equation}
where $\cS f=(\frac{f''}{f'})'-\frac{1}{2}(\frac{f''}{f'})^2$ is the \emph{Schwarzian derivative}. The function $\cS f_\eta^{-1}$ is viewed as a quadratic differential in the neighbourhood of 0, which is dual to meromorphic vector fields by the residue pairing (contour integral on a small loop around 0). Similarly, in the definition of $\tilde{\vartheta}$, the contour integral is on a small loop around $\infty$. By definition, we have the expansion $\cS (e^{c_\eta} f_\eta)^{-1}(z)=\sum_{n=2}^\infty\vartheta_\eta(\rv_{-n})z^{n-2}$ around 0, and the coefficients $\vartheta_\eta(\rv_n)$ are called the \emph{Neretin polynomials} of $e^{c_\eta}f_\eta$ \cite[Section 3.2]{GQW24}. For all $\rv \in \C(z)\del_z$, we define
\begin{equation}
    \label{E:bL}
\bL_\rv:=\cL_\rv-\frac{c_\rM}{12}\vartheta(\rv)
\qquad \text{and} \qquad
\bar{\bL}_\rv:=\bar{\cL}_\rv-\frac{c_\rM}{12}\overline{\vartheta(\rv)}
\end{equation}
as operators on $\cC$ \eqref{E:cC}.
Note that $\bL_\rv=\cL_\rv$ if $\rv\in\C[z]\del_z$ has no pole at $z=0$. In the sequel, we will use the notation $\bL_n=\bL_{\rv_n}$ for all $n\in\Z$. 


\begin{theorem}[Virasoro representation]\label{T:intro_virasoro}
For all $\rv \in \C(z)\del_z$,
the densely defined operators $\bL_\rv,\bar{\bL}_\rv : \cC \to L^2(\nu)$ are closable and preserve $\cC$. They form two commuting representations of the Virasoro algebra with central charge $c_\rM$: for all $\rv,\rw\in\C(z)\del_z$,
\begin{align*}
&[\bL_\rv,\bL_\rw]=\bL_{[\rv,\rw]}+\frac{c_\rM}{12}\omega(\rv,\rw);\qquad[\bar{\bL}_\rv,\bar{\bL}_\rw]=\bar{\bL}_{[\rv,\rw]}+\frac{c_\rM}{12}\overline{\omega(\rv,\rw)};\qquad[\bL_\rv,\bar{\bL}_\rw]=0,
\end{align*}
where $\omega$ is the Virasoro cocycle 
\begin{equation}\label{eq:def_omega}
\omega(v\del_z,w\del_z)=\frac{1}{2i\pi}\oint_{\mathbb{S}^1} v'''(z)w(z)\d z.
\end{equation}
\end{theorem}

Representations of the Virasoro algebra on the space of Jordan curves are not new, but there is a crucial difference between the one in \cite{Kirillov98,AiraultMalliavin01} and the one in \cite{KontsevichSuhov07} (which is also the one from \cite{GQW24} and the present paper). This is a rather subtle point, and we refer the reader to Section \ref{subsec:comparison} for a comprehensive discussion. As we described in Theorems \ref{T:intro_witt} and \ref{T:intro_virasoro}, our setup relies on the action of meromorphic vector fields on $\cJ_{0,\infty}$, via the associated flows. In particular, we obtain two commuting families of operators, instead of only one. Identifying the right framework was one of the main challenges of this work which was inspired by recent progress on Liouville CFT; see especially \cite{BGKRV22}.

		\subsection{Integration by parts and Shapovalov form}\label{SS:shapo}
Our next result concerns the representation of the Shapovalov form on $L^2(\nu)$. Namely, we construct a Hermitian form $\cQ$ on $L^2(\nu)$ such that $\cQ(\bL_{-n}F,G)=\cQ(F,\bL_nG)$ for all $n\in\Z$.

Let $\iota:\hat{\C}\to\hat{\C}$ be the anticonformal involution $\iota(z):=1/\bar{z}$. The conjugation by $\iota$ defined by $g \in \tilde \cE \mapsto\iota\circ g\circ\iota \in \cE$ is a bicontinuous involution. We define a bounded operator $\Theta$ on $L^2(\nu)$ by
\begin{equation}\label{eq:def_Theta}
\Theta F(c_\eta,f_\eta):=F(\tilde{c}_\eta,\iota\circ g_\eta\circ\iota),\qquad\forall F\in L^2(\nu),\,\forall \eta \in \cJ_{0,\infty}.
\end{equation}
This defines an involution of $L^2(\nu)$. Moreover, by reflection symmetry of $\nu$ ($\nu$ is invariant under~$\iota$ \cite[Lemma 2.6]{GQW24}), $\Theta$ is also self-adjoint. Thus, the spectrum of $\Theta$ is $\{1,-1\}$ and we will think of the corresponding eigenspaces as spaces of even and odd functions. Since $\Theta$ is self-adjoint, the expression
\begin{equation}\label{eq:def_shapo}
\cQ(F,G):=\langle F,\Theta G\rangle_{L^2(\nu)}
\end{equation}
defines a Hermitian form (i.e. a symmetric sesquilinear form) on $L^2(\nu)$. This form is non-degenerate since $\cQ(F,\Theta F)=\norm{F}_{L^2(\nu)}^2\neq0$ for all $F\neq0$, but it is neither positive nor negative: it is positive (resp. negative) on the space of even (resp. odd) functions.

\begin{theorem}\label{T:intro_shapo}
Let $\rv=v\del_z\in\C(z)\del_z$ and define
\begin{equation}\label{eq:star}
    \rv^*=v^*(z)\del_z:=z^2\overline{v(1/\bar z)}\del_z.
\end{equation}
The operator $\bL_{\rv^*}$ is the $\cQ$-adjoint of $\bL_\rv$, in the sense that for all $F,G\in \cC_\mathrm{comp}$ \eqref{E:cC_compact}, we have 
\[\cQ(\bL_{\rv^*}F,G)=\cQ(F,\bL_\rv G);\qquad\cQ(\bar{\bL}_{\rv^*}F,G)=\cQ(F,\bar{\bL}_\rv G).\]
In other words, we have the following hermiticity relations with respect to the $L^2(\nu)$-inner product:
\[\bL_\rv^*=\Theta\circ\bL_{\rv^*}\circ\Theta;\qquad\bar{\bL}_\rv^*=\Theta\circ\bar{\bL}_{\rv^*}\circ\Theta.\]
\end{theorem}

This result stands out as one of our main contributions and drives some of our proofs. Specifically, it will be a crucial input in the proof of the commutation relations of $\bL_\rv$ and $\bL_\rw$ for $\rv, \rw \in z \C[z^{-1}] \partial_z$ and in the proof of the closability of $\bL_\rv$ on $L^2(\nu)$.

As we will see,
Theorem \ref{T:intro_shapo} is a consequence of the following integration by parts formula of independent interest (see Theorem \ref{T:ibp}): for all $\rv \in \C(z)\del_z$,
\begin{equation}\label{E:intro_IBP}
\cL_\rv^*=-\bar{\cL}_\rv-\frac{c_\rM}{12}\overline{\tilde{\vartheta}_\eta(\rv)}
+\frac{c_\rM}{12}\overline{\vartheta_\eta(\rv)}; \qquad
\bar\cL_\rv^*=-\cL_\rv-\frac{c_\rM}{12}\tilde{\vartheta}_\eta(\rv)
+\frac{c_\rM}{12}\vartheta_\eta(\rv),
\end{equation}
where $\vartheta_\eta(\rv)$ and $\tilde{\vartheta}_\eta(\rv)$ are defined in \eqref{eq:def_vartheta}.

\noindent\textbf{Heuristics.}
We now discuss some heuristics that first led us to guess the aforementioned integration by parts formula. The SLE loop measure with parameter $\kappa$ can be thought of as being absolutely continuous w.r.t. the SLE loop measure with parameter $8/3$ (corresponding to a vanishing central charge):
\begin{equation}
    \label{E:intro_path_integral}
\text{``}~\d \nu_\kappa(\eta) = \exp \Big( \frac{c_\rM}{24} I^L(\eta) \Big) \d \nu_{8/3}(\eta)~\text{''}
\end{equation}
where $I_L(\eta)$ is the Loewner energy, or universal Liouville action, of $\eta$.
This path integral formulation is especially inspired by \cite{carfagnini2023onsager}, but, as we are about to argue, is also consolidated by our integration by parts formula \eqref{E:intro_IBP}.
Making sense of \eqref{E:intro_path_integral} would require renormalisations because, for $\nu_{8/3}$-almost every $\eta$, $I^L(\eta)=+\infty$ (and, indeed, $\nu_\kappa$ is not absolutely continuous w.r.t $\nu_{8/3}$ when $\kappa \neq 8/3$).
The SLE$_{8/3}$-loop measure possesses a greater level of symmetry since it is exactly invariant under restrictions. This can be used to show the integration by parts \eqref{E:intro_IBP} in this case. Making formal computations with \eqref{E:intro_path_integral} and using a formula by Takhtajan and Teo \cite[Chapter~2, Theorem~3.5]{TakhtajanTeo06} relating variations of the Loewner energy to the Schwarzian derivative of $f$ and $g$, we can then derive \eqref{E:intro_IBP}. This formal computation is detailed in Appendix~\ref{Appendix:formal}.

		\subsection{Highest-weight representations}\label{subsubsec:hw_modules}

The spectral theory of $\bL_0$ and $\bar{\bL}_0$ leads to the introduction of a family of operators $(\bL_n^\lambda,\bar{\bL}^\lambda_n)_{n\in\Z}$ indexed by $\lambda\in\C$, which are densely defined operators on $L^2(\nu^\#)$. See \eqref{eq:op_lambda} for the precise definition. They have a simple analytic description \eqref{eq:op_lambda} and can be defined more abstractly by the conjugation of the representations $(\bL_n,\bar{\bL}_n)_{n\in\Z}$ with a certain operator introduced in Section~\ref{subsec:level_operators} (see \eqref{E:cP_lambda}). For each $\lambda\in\C$, the commutation relations are preserved by this conjugation (Lemma~\ref{lem:commutations_lambda}), so we get two commuting Virasoro representations $(\bL_n^\lambda,\bar{\bL}_n^\lambda)_{n\in\Z}$ with central charge $c_\rM$. The relevance of these operators (Proposition \ref{lem:commutations_lambda}) is that
\[\bL_0^\lambda\ind=\bar{\bL}_0^\lambda\ind=\lambda\ind\quad \text{and for each} \quad n \ge 1, \quad
\bL_n^\lambda\ind=\bar{\bL}_n^\lambda\ind=0,\]
where $\ind$ is the constant function $f \in \cE \mapsto 1$.
In other words, $\ind$ is a highest-weight state of weight $\lambda$ for these representations. We are going to use this property to define highest-weight representations, but we first need to introduce some notations.

\noindent\textbf{Integer partitions and polynomials.}
We can represent an integer partition by a sequence $\bk=(k_m)_{m\in\N^*}\in\N^{\N^*}$ containing only finitely many non-zero terms, where the number partitioned by $\bk$ is $|\bk|=\sum_{m=1}^\infty mk_m$.
For $N \ge 0$, we denote by $\cT$ (resp. $\cT_N$) the set of all integer partitions (resp. the set of partitions of $N$). By convention, $\cT_0 = \{\varnothing\}$ where $\varnothing = (0,0,0,\dots)$ denotes the empty partition.
For every $\bk,\tilde{\bk}\in\cT$, we define the monomial
\begin{equation}\label{E:monomial}
\ba^\bk {\bar \ba}^{\tilde \bk}:=\prod_{m=1}^\infty a_m^{k_m} {\bar{a}_m}^{\tilde k_m}\in\C[(a_m, \bar{a}_m)_{m\geq1}]
,\end{equation}
and we call $|\bk|$ (resp. $|\tilde{\bk}|$) the \emph{left (resp. right) degree} of $\ba^\bk {\bar \ba}^{\tilde \bk}$. There are $\#\cT_N\#\cT_{\tilde{N}}$ linearly independent monomials of left degree $N$ and right degree $\tilde{N}$. The \emph{total degree} of $\ba^\bk\bar{\ba}^{\tilde{\bk}}$ is $|\bk|+|\tilde{\bk}|$, and the total degree of an arbitrary polynomial is the maximal total degree over its monomials. We will also use the same shorthand notation for operators: 
\begin{equation}\label{eq:string}
\bL_{-\bk}:=\cdots\bL_{-3}^{k_3}\bL_{-2}^{k_2}\bL_{-1}^{k_1};\qquad\bL_\bk:=\bL_1^{k_1}\bL_2^{k_2}\bL_3^{k_3}\cdots.
\end{equation}
Note the reversal of the order in which we read the partition.

With these notations at hand, we can define $\Psi_{\lambda,\bk,\tilde{\bk}}:=\bL_{-\bk}^\lambda\bar{\bL}^\lambda_{-\tilde{\bk}}\ind\in\C[(a_m,\bar{a}_m)_{m\geq1}]$, $\Psi_{\lambda,\bk}:=\Psi_{\lambda,\bk,\emptyset}$, and
\begin{equation}
    \label{eq:def_V}
\cV_\lambda:=\mathrm{span}\left\lbrace\Psi_{\lambda,\bk}\big|\,\bk\in\cT\right\rbrace;\qquad\cW_\lambda:=\mathrm{span}\left\lbrace\Psi_{\lambda,\bk,\tilde{\bk}}\big|\,\bk,\tilde{\bk}\in\cT\right\rbrace,
\end{equation}
where the linear span is understood in the algebraic sense (finite linear combinations). By definition, $\cV_\lambda$ is a highest-weight representation of weight $\lambda$. As such, they have a natural grading $\cV_\lambda=\oplus_{N\in\N}\cV_\lambda^N$ where $\cV_\lambda^N=\mathrm{span}\{\Psi_{\lambda,\bk}|\,\bk\in\cT_N\}$, and similarly $\cW_\lambda^{N,\tilde{N}}=\mathrm{span}\{\Psi_{\lambda,\bk,\tilde{\bk}}|\,\bk\in\cT_N,\,\tilde{\bk}\in\cT_{\tilde{N}}\}$. Vectors in $\cW_\lambda^{N,\tilde{N}}$ are eigenstates of $\bL_0^\lambda$ (resp. $\bar{\bL}_0^\lambda$) with eigenvalue $\lambda+N$ (resp. $\lambda+\tilde{N}$).

 The \emph{Kac table} is the countable set (see e.g. \cite[Appendix A]{Peltola19_review})
\[kac:=\left\lbrace\lambda_{r,s}:=(r^2-1)\frac{\kappa}{16}+(s^2-1)\frac{1}{\kappa}+\frac{1}{2}(1-rs)|\,r,s\in\N^*\right\rbrace.\]
In particular, we have
\[\lambda_{1,2}=\frac{6-\kappa}{2\kappa};\qquad\lambda_{2,1}=\frac{3\kappa-8}{16}.\]
The next theorem classifies the structure of these representations. The main point is that $\cV_\lambda$ is irreducible for all $\lambda\in\C$. We refer the reader to Appendix \ref{app:virasoro} for details on the terminology used in the statement. Briefly, a Verma module is the ``largest" possible highest-weight representation. It may contain non-trivial submodules, in which case the quotient by the maximal proper submodule is called the irreducible quotient.

\begin{theorem}\label{thm:module_structure}
\begin{enumerate}[leftmargin=*]
\item \label{item:verma} Suppose $\lambda\not\in kac$. Then, $\cV_\lambda$ is the Verma module and, for all $N\in\N$, $\cV_\lambda^N=\mathrm{span}\{\ba^\bk|\,\bk\in\cT_N\}$. In particular, $\cV_\lambda=\C[(a_m)_{m\geq1}]$.
\item \label{item:irred_quotient} Suppose $\lambda\in kac$. Then, $\cV_\lambda$ is the irreducible quotient of the Verma module by the maximal proper submodule.
\item \label{item:tensor_product} For all $\lambda\in\C$, $\cW_\lambda$ is isomorphic to the tensor product of two copies of $\cV_\lambda$. 
Moreover,
for all $\lambda\not\in kac$ and $N\in\N$, we have $\oplus_{n+\tilde{n}\leq N}\cW_\lambda^{n,\tilde{n}}=\mathrm{span}\{\ba^\bk\bar{\ba}^{\tilde{\bk}}|\,|\bk|+|\tilde{\bk}|\leq N\}$. In particular, $\cW_\lambda=\C[(a_m,\bar{a}_m)_{m\geq1}]$.
\end{enumerate}
\end{theorem}

Theorem \ref{thm:module_structure} implies that certain linear relations exist between the states at level $rs$ in $\cV_{\lambda_{r,s}}$. Here is one instance of such a relation when $rs=2$.

\begin{corollary}\label{C:level2}
For $\lambda\in\{\frac{6-\kappa}{2\kappa},\frac{3\kappa-8}{16}\}$, the following equality holds in $\C[(a_m)_{m\geq1}]$:
\begin{equation}\label{eq:singular_vector}
\left((\bL_{-1}^\lambda)^2-\frac{2}{3}(2\lambda+1)\bL_{-2}^\lambda\right)\ind=0.
\end{equation}
\end{corollary}

Corollary \ref{C:level2} is a consequence of Theorem \ref{thm:module_structure}, but we will also give a ``computational proof'' of this result in Section~\ref{S:level2}, by evaluating explicitly the states $\bL_{-2}^\lambda\ind$ and $(\bL_{-1}^\lambda)^2\ind$ as polynomials in the coordinates $a_1,a_2$. This result will be combined with conformal Ward identities in a follow-up work in order to prove BPZ equations for correlation functions involving a degenerate field at level 2.
We mention that such BPZ equations are well known in the chordal case (see e.g. the review \cite{Peltola19_review}) and are derived using local martingales of SLE. We stress however that our approach is orthogonal to the literature so far since in our context they are consequences of our CFT framework.

\begin{remark}
The study of (degenerate) highest-weight representations is a standard topic in the representation theory of the Virasoro algebra \cite{FeiginFuchs}, and the probabilistic construction of Liouville CFT has produced concrete examples of such representations \cite[Theorem 4.5]{BGKRV22}. In \cite{BW}, the first author and Baojun Wu showed an analogue of Theorem \ref{thm:module_structure} for these modules. It would be interesting to see if the coupling of Liouville CFT with SLE could give an alternative proof, following the methods of \cite{AngSun21_CLE}.
\end{remark}

	\subsection{Spectral resolution}\label{subsec:spectral}
Our final result is a spectral resolution for the Shapovalov form $\cQ$, which takes the form of a Plancherel-type formula analogous to \cite[Section 6.8]{GKRV20_bootstrap} and can be understood in connection with the harmonic analysis of the group $\mathrm{Diff}(\S^1)$ (see \cite[Section~1.4]{GKRV20_bootstrap} for a discussion). In Liouville CFT, such a formula is the key to the conformal bootstrap, and we hope to address similar questions for SLE in the future. 

 Our goal is to decompose $\cQ$ on the modules $\cW_{ip}\simeq\cV_{ip}\otimes\overline{\cV}_{ip}$, $p\in\R$ forming the joint spectra of the Hamiltonians $\bL_0,\bar{\bL}_0$. Interestingly, there exists an exact formula for the partition function of the Shapovalov form (i.e. the norm of the constant function in $\cW_{ip}$). To state this precisely, define for $f \in \cE$,
\begin{equation}\label{E:tilde_c_0}
    \tilde c_0(f) = \log\text{-conformal radius of } f(\S^1) \text{ viewed from } \infty;
\end{equation}
namely, the exterior of the curve $f(\S^1)$ is uniformised by $e^{-\tilde{c}_0(f)}g$ for some function $g\in\tilde{\cE}$ (depending on $f$). This is a geometrically meaningful nonpositive (\cite[Equation (21)]{Pommerenke75}) quantity: it is a K\"ahler potential for the Velling-Kirillov metric on the universal Teichm\"uller curve \cite{KirillovYurev87,TakhtajanTeo06}, and it is sometimes called the \emph{logarithmic (or analytic) capacity} \cite{KirillovYurev87}. In \cite{SSW09, AngSun21_CLE} (originally introduced by Kenyon and Wilson), it goes by the name of \emph{electrical thickness} (see Lemma \ref{lem:electrical_thickness}).
Remarkably,
the moment generating function of this random variable under the SLE$_\kappa$ shape measure $\nu^\#$ \eqref{E:SLE_shape} is explicit \cite[Theorem~1.3]{AngSun21_CLE}: for $\Re(\lambda)<\frac{1}{2}(1-\frac{\kappa}{8})$, one has\footnote{Our $\lambda$ differs from the one in \cite{AngSun21_CLE} by a factor $-2$ since we choose to parametrise by the conformal weight.}
\begin{equation}\label{eq:def_R}
\rR(\lambda):=\E^\#[e^{-2\lambda\tilde{c}_0}]=\frac{\sin(\pi(1-\kappa/4))}{\pi(1-\kappa/4)}\frac{\pi\sqrt{(1-\kappa/4)^2+\lambda\kappa}}{\sin(\pi\sqrt{(1-\kappa/4)^2+\lambda\kappa})}.
\end{equation}
For $\lambda\in[\frac{1}{2}(1-\frac{\kappa}{8}),\infty)$, the generating function is infinite. Note that the square-roots appear as a ratio, so the choice of branch is not important and $\rR(\lambda)$ admits a meromorphic continuation to the whole complex plane, whose poles form a discrete subset of the interval $[\frac{1}{2}(1-\frac{\kappa}{8}),\infty)$. The threshold $\frac{1}{2}(1-\frac{\kappa}{8})$ is easy to explain probabilistically: the probability that $e^{-\tilde{c}_0}$ gets larger than some $\epsilon^{-1}>0$ is of the same order as the probability that the SLE gets at distance $\epsilon$ to a fixed point, which is known to scale like $\epsilon^{2-(1+\frac{\kappa}{8})}=\epsilon^{1-\frac{\kappa}{8}}$ (the Hausdorff dimension of the curve is $1+\frac{\kappa}{8}$ a.s. \cite{10.1214/07-AOP364}). 

Let $\lambda\in\C$, and $\bk,\tilde{\bk},\bk',\tilde{\bk}'\in\cT$. If this is well defined and finite, we set
\begin{equation}\label{eq:def_Q_lambda}
\cQ_\lambda(\Psi_{\lambda,\bk,\tilde{\bk}},\Psi_{\bar{\lambda},\bk',\tilde{\bk'}}):=\E^\#\left[e^{-(2\lambda+|\bk|\wedge|\bk'|+|\tilde\bk|\wedge|\tilde{\bk}'|)\tilde{c}_0}\Psi_{\lambda,\bk,\tilde{\bk}}\overline{\Theta^\#\Psi_{\bar{\lambda},\bk',\tilde{\bk}'}}\right],
\end{equation}
where $\Theta^\#P(f)=P(\iota\circ g\circ\iota)$ for all polynomials $P$. From \eqref{eq:def_R}, we have $|\cQ_\lambda(\Psi_{\lambda,\bk,\tilde{\bk}},\Psi_{\bar{\lambda},\bk',\tilde{\bk}'})|<\infty$ when $\Re(\lambda)<\frac{1}{2}(1-|\bk|\wedge|\bk'|-|\tilde\bk|\wedge|\tilde{\bk}'|-\frac{\kappa}{8})$.
For the statement of the next theorem, we use the following convention for the Fourier transform of $F\in L^2(\nu)$:
\[\hat{F}(p,f):=\int_\R e^{2ipc}F(c,f)\frac{\d c}{\sqrt{\pi}}.\]
It defines a unitary transformation of $L^2(\nu)$.
For $\lambda \in \C$, we will also denote by $\rB_\lambda$ the Gram matrix of the Shapovalov form of weight $\lambda$: by definition, for all $\bk,\bk' \in \cT$, $\bL_{\bk'}^\lambda \bL_{-\bk}^\lambda \ind = \rB_\lambda(\bk,\bk')\ind$. We have $\rB_\lambda(\bk,\bk')=0$ if $|\bk|\neq|\bk'|$, so we can view the Shapovalov form as a collection of finite-dimensional matrices indexed by the level. The Shapovalov form is degenerate if and only if $\lambda$ belongs to the Kac table. When $\lambda$ does not belong to the Kac table, we denote its inverse by $\rB_\lambda^{-1}$, which is also a collection of finite-dimensional matrices indexed by the level. See Appendix~\ref{app:virasoro} for details on the Shapovalov form. 

The next statement uses the operator $\cP_\lambda:\cW_\lambda\to\cC_c^\infty(\R)'\otimes\C[(a_m,\bar{a}_m)_{m\geq1}]$, characterised by $\cP_\lambda\Psi_{\lambda,\bk,\tilde{\bk}}=e^{-(2\lambda+|\bk|+|\tilde{\bk}|)c}\Psi_{\lambda,\bk,\tilde{\bk}}$. See Section \ref{subsec:level_operators} for the precise definition and more properties of this operator.

\begin{theorem}\label{T:spectral}
\begin{enumerate}[leftmargin=*]
\item\label{item:iso_Q} For all $\Re(\lambda)<\frac{1}{2}(1-\frac{\kappa}{8})$, $\bk,\tilde{\bk},\bk',\tilde{\bk}'\in\cT$, $F,G\in\cC_c^\infty(\R)$, we have
\begin{equation}\label{E:T_spectral1}
\cQ\left(F\cP_\lambda\Psi_{\lambda,\bk,\tilde{\bk}},G\cP_{\bar{\lambda}}\Psi_{\bar{\lambda},\bk',\tilde{\bk}'}\right)=\int_\R\hat{F}(p)\overline{\hat{G}(-p)}\rR(\lambda+ip)\rB_{\lambda+ip}(\bk,\bk')\rB_{\lambda+ip}(\tilde{\bk},\tilde{\bk}')\frac{\d p}{\sqrt{\pi}}.
\end{equation}
\item Recalling that $\rR(\lambda)$ is defined in \eqref{eq:def_R}, for all $F\in\cC_\mathrm{comp}$ (dense subset of $L^2(\nu)$ defined in \eqref{E:cC_compact}), the following equality holds in $L^2(\nu)$:
\begin{equation}\label{eq:proj_F}
F=\int_\R\frac{1}{\rR(ip)}\sum_{\bk,\tilde{\bk},\bk',\tilde{\bk}'}\cQ(F,\cP_{-ip}\Psi_{-ip,\bk',\tilde{\bk}'})\rB_{ip}^{-1}(\bk,\bk')\rB_{ip}^{-1}(\tilde{\bk},\tilde{\bk}')e^{-(2ip+|\bk|+|\tilde{\bk}|)c}\Psi_{ip,\bk,\tilde{\bk}}\frac{\d p}{\sqrt{\pi}}.
\end{equation}
\item For all $\lambda\in\C$ and $\bk,\tilde{\bk},\bk',\tilde{\bk}'\in\cT$ with $\Re(\lambda)<\frac{1}{2}(1-|\bk|\wedge|\bk'|-|\tilde\bk|\wedge|\tilde{\bk}'|-\frac{\kappa}{8})$, we have $|\cQ_\lambda(\Psi_{\lambda,\bk,\tilde{\bk}},\Psi_{\bar\lambda,\bk',\tilde{\bk}'})|<\infty$ and
\begin{equation}
    \label{E:T_Q_lambda}
\cQ_\lambda(\Psi_{\lambda,\bk,\tilde{\bk}},\Psi_{\bar{\lambda},\bk',\tilde{\bk}'})=\rR(\lambda)\rB_\lambda(\bk,\bk')\rB_\lambda(\tilde\bk,\tilde\bk').
\end{equation}
\end{enumerate}
\end{theorem}

\begin{remark}
In the above statement, the integrand must be understood as being defined Lebesgue-a.e. since it is not defined on the Kac table (which is a discrete set).
\end{remark}

\subsection{Organisation and notations}\label{SS:notation}

The paper is organised as follows:

\begin{itemize}[leftmargin=*, label=\raisebox{0.25ex}{\tiny$\bullet$}]
    \item Section \ref{sec:rep}: we construct our representations of the Witt and Virasoro algebras.
    \item Section \ref{sec:hw_modules}: we construct and classify the highest-weight representations.
    \item Section \ref{sec:IBP}: we derive an integration by parts formula and build in this way the Shapovalov form.
    \item Section \ref{sec:spectral}: we prove the spectral decomposition of the Shapovalov form.
    \item Section \ref{sec:uniqueness}: we provide our proof of the uniqueness of MKS measure.
    \item Section \ref{S:applications}: we discuss two consequences of our uniqueness result: the reversibility and the duality of SLE loop measures.
    \item Section \ref{sec:future}: the current paper lays the groundwork for an in-depth study of the CFT of SLE. We explain this programme in some details in this section.
    \item Appendix \ref{app:virasoro}: we review some standard facts of the representation theory of the Virasoro algebra.
    \item Appendix \ref{app:SLE}: we show that the SLE loop measure is a Borel measure on the space of simple Jordan curves.
\end{itemize}

Here is a list of the main notations we use in this paper.

\begin{itemize}[leftmargin=*, label=\raisebox{0.25ex}{\tiny$\bullet$}]
\item $\D$: the unit disc. $\D^*=\hat{\C}\setminus\bar{\D}$: the exterior disc. $|\d z|^2$: the Euclidean volume on $\C$.
	
\item $\C[z]$: Polynomials in $z$. $\C(z)$: Laurent polynomials in $z$.

\item $\rv_n, n \in \Z$: Vector field $\rv_n = -z^{n+1} \partial_z$.

\item $\rv^*$ vector field dual to $\rv$, defined by $\rv^* =v^*(z)\del_z:=z^2\overline{v(1/\bar z)}\del_z$. 


\item $\cJ$ (resp. $\cJ_{0,\infty}$): Jordan curves in $\hat \C$ (resp. disconnecting 0 from $\infty$), endowed with the topology induced by the Hausdorff metric \eqref{E:def_cJ}.

\item $\cE$: Space of holomorphic maps $f:\D\to\C$ with continuous, injective extension to $\bar{\D}$, $f(0)=0$, $f'(0)=1$, $f(\bar\D) \subset \D$, endowed with the local uniform topology on $\D$ \eqref{E:def_cE}.

\item $\tilde{\cE}$: Space of holomorphic maps $g:\D^*\to\hat{\C}\setminus\{0\}$ with continuous, injective extension to $\overline{\D^*}$, $g(\infty)=\infty$, $g'(\infty)=1$, $g(\overline{\D^*}) \subset \hat\C\setminus\{0\}$, endowed with the local uniform topology on $\D^*$ \eqref{E:def_cE_tilde}.

\item $c:\cJ_{0,\infty}\to\R$, $f:\cJ_{0,\infty}\to\cE$, $a_m:\cJ_{0,\infty}\to\C$. Coordinates on $\cJ_{0,\infty}$.

\item $\tilde{c}:\cJ_{0,\infty}\to\R$, $g:\cJ_{0,\infty}\to\tilde{\cE}$, $b_m:\cJ_{0,\infty}\to\C$. Another set of coordinates on $\cJ_{0,\infty}$.

\item $\nu$: SLE$_\kappa$ loop measure restricted to $\cJ_{0,\infty}$, unless specified otherwise. $\nu^\#$: shape measure on $\cE$ \eqref{E:SLE_shape}. $\E^\#$: expectation wrt $\nu^\#$.

\item $\Theta$: reflection operator \eqref{eq:def_Theta}.

\item $\cQ$: Shapovalov form \eqref{eq:def_shapo}.

\item $\cV_\lambda$: highest-weight representation of weight $\lambda$ \eqref{eq:def_V}.

\item $\cQ_\lambda$: Shapovalov form of weight $\lambda$ \eqref{eq:def_Q_lambda}.

\item $\cT$ ($\cT_N$): set of integer partitions (of $N$), see Section \ref{subsubsec:hw_modules}.

\item $\vartheta$: Differential of the Weil-Petersson potential \eqref{eq:def_vartheta}.

\item $\varpi$: Differential of the Velling-Kirillov potential \eqref{eq:def_varpi}.

\item $\omega$: Virasoro cocycle \eqref{eq:def_omega}.

\item $\ann_\eta$: for any Jordan curve $\eta \in \cJ_{0,\infty}$, annulus defined by
\begin{equation}\label{eq:annulus_eta}
\ann_\eta:=\{z\in\C|\,e^c/8<|z|<8e^{-\tilde{c}}\}.
\end{equation}
By the Koebe quarter theorem, for any $\eta \in \cJ_{0,\infty}$, we have $\frac{e^c}{4}\D\subset\rmint(\eta)$ and $4e^{-\tilde{c}}\D^*\subset\rmext(\eta)$, so that $\ann_\eta$ contains $\eta$ in its interior.
\item If $X$ is a topological space, we denote by $\cC^0(X)$ the space of continuous functions on $X$.
\end{itemize}

\section{Witt and Virasoro representations}\label{sec:rep}

This section is organised as follows.
In Section \ref{subsec:dense}, we define two subspaces $\cC$ and $\cC_\mathrm{comp}$ of $L^2(\nu)$ of ``test functions'' and show that they are dense.
In Section \ref{subsec:witt}, we construct two commuting representations of the Witt algebra on $\cC$.
We will then modify these representations to get two commuting representations of the Virasoro algebra on $\cC_\mathrm{comp}$ in Section \ref{subsec:virasoro}
(we will later show in Theorem \ref{T:ibp} that these operators are closable on $L^2(\nu)$ as a corollary of the integration by parts formula).
These operators are obtained from the infinitesimal action of the group of quasiconformal homeomorphisms of the Riemann sphere on the space of Jordan curves. We will start by recalling some related standard results in Section \ref{subsec:background}.

	\subsection{Background on quasiconformal mappings}\label{subsec:background}

This section gives some background on quasiconformal homeomorphisms of the Riemann sphere. We refer to \cite{Ahlfors87_QCLectures} for more details.

A quasiconformal homeomorphism of $\hat{\C}$ is a homeomorphism $\Phi$ which is differentiable almost everywhere and satisfies
\begin{equation}\label{eq:beltrami}
\del_{\bar{z}}\Phi=\mu\del_z\Phi
\end{equation}
almost everywhere for some $\mu\in L^\infty(\C)$ with $\norm{\mu}_{L^\infty(\C)}<1$. The function $\mu$ is called the \emph{Beltrami coefficient} of $\Phi$. Conversely, for all $\mu\in L^\infty(\C)$ with $\norm{\mu}_{L^\infty(\C)}<1$, the Beltrami equation \eqref{eq:beltrami} admits a solution in the space of quasiconformal homeomorphisms of $\hat{\C}$, and this solution is unique modulo post-composition by M\"obius transformations of $\hat\C$ \cite[Chapter V, Theorem 3]{Ahlfors87_QCLectures}. By Weyl's lemma, the associated quasiconformal map is conformal outside the support of $\mu$ (here and in what follows, we will say that a map is conformal if it is biholomorphic). Given the form of the Beltrami equation, a Beltrami coefficient transforms as a $(-1,1)$-tensor: under a conformal change of coordinates $f$, a Beltrami coefficient pulls back to
\begin{equation}\label{eq:pullback_beltrami}
f^*\mu=\mu\circ f\frac{\overline{f'}}{f'}.
\end{equation}
The group of quasiconformal homeomorphisms $\Phi$ acts naturally on the set of Jordan curves $\eta$ by $\Phi\cdot\eta:=\Phi(\eta)$. Of particular interest to us will be those quasiconformal maps which are conformal away from small neighbourhoods of 0 and $\infty$. These can be generated by the Witt algebra, and we will construct the Witt representations as Lie derivatives along the induced vector fields.

We will consider two different types of quasiconformal mappings:

\begin{definition}\label{def:qc}
Let $\mu\in L^\infty(\C)$ with $\norm{\mu}_{L^\infty(\C)}<1$ and s.t. $\mu$ is compactly supported in $\D$. 
\begin{itemize}[leftmargin=*, label=\raisebox{0.25ex}{\tiny$\bullet$}]
    \item $\Phi_\mu$ is the unique solution to the Beltrami equation satisfying $\Phi_\mu(0)=0$, and $\Phi_\mu(z)=z+O_{z\to\infty}(1)$.
    \item Extend $\mu$ to the whole complex plane by the formula $\mu(z):=\overline{\mu(1/\bar{z})}\frac{z^2}{\bar{z}^2}$ for $|z|>1$. Define $\tilde{\Phi}_\mu$ to be the unique solution to the Beltrami equation fixing $0,1,\infty$. 
    Because $\mu$ vanishes in a neighbourhood of $\S^1$, $\tilde{\Phi}_\mu$ is conformal in that neighbourhood.
    Moreover, by uniqueness of solutions to the Beltrami equation fixing three points, $\tilde{\Phi}_\mu(z) = 1/\overline{\tilde\Phi_\mu(\bar z)}$ for all $z \in \hat\C$. In particular,
    $\tilde{\Phi}_\mu$ restricts to an analytic diffeomorphism of $\S^1$.
\end{itemize}
\end{definition}
The relevance of this definition appears when studying the following problem: suppose $f_0\in\cE$ and $\mu$ is compactly supported in $f(\D)$. Then, $\Phi_\mu\circ f$ is \emph{not} the uniformising map of the domain $\Phi_\mu\circ f_0(\D)$, because it is not conformal on $\D$. However, if $f_1:\D\to\Phi_\mu\circ f_0(\D)$ is such a conformal map, then $\tilde\Phi:=f_1^{-1}\circ\Phi_\mu\circ f_0$ is an analytic diffeomorphism of $\S^1$ which can be described by the second item of the above definition. This is particularly useful when $\Phi_\mu$ is close to the identity, due to explicit variational formulas which we now recall.

Indeed, we can give a first order expansion of $\tilde{\Phi}_{t\mu}$ for small complex $t$ as found in \cite[Chapter V.C, Theorem 5]{Ahlfors87_QCLectures}. This theorem shows that the solution map $t\mapsto\Phi_{t\mu}$ (or $t\mapsto\tilde{\Phi}_{t\mu}$) is of class $\cC^1$ in a complex neighbourhood of $t=0$ (with respect to the uniform topology), and gives an expression for the derivative, recalled in the next proposition.

\begin{proposition}\label{lem:variational_formulas}
For all $\mu\in L^\infty(\C)$ compactly supported in $\D$, we have
\[\tilde{\Phi}_{t\mu}(z)=z-\frac{t}{\pi}\int_\D\frac{z(z-1)\mu(\zeta)}{\zeta(\zeta-1)(\zeta-z)}|\d\zeta|^2+\frac{\bar{t}}{\pi}\int_\D\frac{z(z-1)\overline{\mu(\zeta)}}{\bar{\zeta}(\bar{\zeta}-1)(1-z\bar{\zeta})}|\d\zeta|^2+o(t),\]
where $o(t) \to 0$ as $t \to 0$ uniformly on compacts.
\end{proposition}
\begin{proof}
\cite[Chapter V, Theorem 5]{Ahlfors87_QCLectures} and the symmetry of $\mu$ gives
\begin{align}
    \tilde{\Phi}_{t\mu}(z)=z-\frac{1}{\pi}\int_\D\frac{z(z-1)t\mu(\zeta)}{\zeta(\zeta-1)(\zeta-z)}|\d\zeta|^2-\frac{1}{\pi}\int_{\D^*}\frac{z(z-1)\overline{t\mu(1/\bar{\zeta})}}{\zeta(\zeta-1)(\zeta-z)}|\d\zeta|^2+o(t).
\end{align}
By the change of variable $\zeta\mapsto1/\bar{\zeta}$, the second integral evaluates to
\[\bar{t}\int_\D\frac{z(z-1)\overline{\mu(\zeta)}}{\bar{\zeta}^{-1}(\bar{\zeta}^{-1}-1)(\bar{\zeta}^{-1}-z)}\frac{1}{\bar{\zeta}^4}|\d\zeta|^2=-\bar{t}\int_\D\frac{z(z-1)\overline{\mu(\zeta)}}{\bar{\zeta}(\bar{\zeta}-1)(1-z\bar{\zeta})}|\d\zeta|^2.\]
\end{proof}

We conclude with two lemmas on quasiconformal extensions.

\begin{lemma}\label{lem:extend_qc}
Let $r>0$ and $(v_n)_{n\ge 2} \in \C^{\N_{\ge 2}}$ with only finitely many nonzero coefficients. Let
\[\mu(z):=-\sum_{n=2}^\infty nv_nr^{-2n}\bar{z}^{n-1}z\ind_{|z|<r} \qquad \text{and} \qquad v(z) = \frac{1}{\pi}\int_{\hat{\C}}\frac{\mu(\zeta)}{z-\zeta}|\d \zeta|^2\]
be the Cauchy transform of $\mu$. Then,
\begin{equation}
    \label{E:del_bar_v}
    \del_{\bar z} v = \mu
\end{equation}
and for all $z \in r\D^*$, $v(z)=-\sum_{n=2}^\infty v_nz^{-n+1}$. Moreover,
$\Phi_{t\mu}$ is conformal in $r\D^*$, and we have $\Phi_{t\mu}(z)=z+tv(z)+o(t)$ uniformly on compacts of $r\D^*$.
\end{lemma}

\begin{proof}
The fact that $\del_{\bar z} v = \mu$ is standard; see for instance \cite[Chapter~V, Lemma~3]{Ahlfors66}.
By linearity, it suffices to work with the generators $\rv_{-n} = -z^{-n+1} \partial_z$ for $n\geq2$, so we consider the Beltrami differential $\mu(z)= nr^{-2n}\bar{z}^{n-1}z\ind_{|z|<r}$.
We have
\begin{align*}
v(z)=\frac{1}{\pi}\int_{\hat{\C}}\frac{\mu(\zeta)}{z-\zeta}|\d \zeta|^2
&=-\frac{nr^{-2n}}{\pi}\int_{r\D}\frac{\bar{\zeta}^{n-1}\zeta}{z(1-\frac{\zeta}{z})}|\d \zeta|^2
=-\frac{nr^{-2n}}{\pi}\int_{r\D}\sum_{m=0}^\infty\bar{\zeta}^{n-1}\zeta^{m+1}z^{-m-1}|\d \zeta|^2\\
&=-\frac{nr^{-2n}z^{-n+1}}{\pi}\int_{r\D}|\zeta|^{2n-2}|\d\zeta|^2=-z^{-n+1}.
\end{align*}
Hence, $v$ coincides with $z\mapsto-z^{-n+1}$ in $r\D^*$.

Moving on to the expansion of $\Phi_{t\mu}$,
by \cite[Chapter V.C, Theorem 5]{Ahlfors87_QCLectures},
$\Phi_{t\mu}(z)=z+tv_\mu(z)+o(t)$ uniformly for $z\in\hat{\C}$, for some vector field $v_\mu$ satisfying $\del_{\bar{z}}v_{\mu}=\mu$. Moreover, $\Phi_{t\mu}$ is conformal in $r\D^*$ and the normalisation of $\Phi_{t\mu}$ is such that $v_\mu(z)=O_{z\to\infty}(1)$ and $v_\mu(0)=0$. We will show that $v_\mu$ coincides with $v$ in the space of holomorphic vector fields in the neighbourhood of $\infty$.
By uniqueness (modulo M\"obius transformations) of the solution to the Beltrami equation, $v-v_\mu$ is a polynomial vector field on $\hat{\C}$ (i.e. a quadratic polynomial), so it suffices to check that $v$ has the desired normalisation. On the one hand, we have $v(z)=O_{z\to\infty}(1)$, so we have the desired behaviour at $\infty$. On the other hand, the Cauchy transform of $\mu$ at 0 evaluates to 0 since $\int_{r\D}\bar{\zeta}^{n-1}|\d\zeta|^2=0$ for all $n\geq2$. Hence $v(0)=0$ and we deduce that $v=v_\mu$.
\end{proof}

\begin{lemma}\label{L:derivative0}
Let $\eta\in\mathcal{J}_{0,\infty}$ and $(\phi_t)$ be a family of conformal maps on $O_\eta$ \eqref{eq:annulus_eta} such that $\frac1t(\phi_t-\mathrm{id})\to 0$ uniformly on compacts of $O_\eta$ as $t\to0$. Then, $\phi_t$ admits a quasiconformal extension to $O_\eta\cup\mathrm{int}(\eta)$ such that $|\partial_{\bar z}\phi_t/\partial_z\phi_t|=o(t)$ uniformly as $t\to0$.
\end{lemma}

\begin{proof}
    Assume first that $\phi_t$ is conformal in $\frac{e^c}{8}\D^*$, and up to rescaling assume without loss of generality assume $\frac{3e^c}{16}=1$. Let $\psi_t : \D \to \rmint(\phi_t(\S^1))$ be the normalised conformal map uniformising the interior of $\phi_t(\S^1)$. 
    We have a conformal welding $h_t=\phi_t^{-1}\circ\psi_t$ with $h_t\in\mathrm{Diff}^\omega(\S^1)$. 
    Then, $h_t$ is differentiable at $t=0$ in the space of quasisymmetric homeomorphisms of $\S^1$ (see e.g. \cite[Part I, Section 1.2.3]{TakhtajanTeo06}), and its differential at $t=0$ vanishes. Hence, by the Beurling--Ahlfors extension theorem \cite{BeurlingAhlfors}, $h_t$ can be extended quasiconformally to the whole Riemann sphere with Beltrami coefficient $\mu_t=\del_{\bar z}h_t/\del_zh_t$ such that $|\mu_t|=o(t)$ uniformly as $t\to0$. Then, solving the Beltrami equation in $\hat\C$ with coefficient $(\psi_t)_*(\mu_t\mathbf{1}_\D)$ defines a quasiconformal extension of $\phi_t$ with the required property.

    In the general case, we may write $\phi_t=\phi_t^-+v_t$ where (for small enough $|t|$) $\phi_t^-$ is as in the previous paragraph, and $v_t$ is holomorphic in $O_\eta\cup\mathrm{int}(\eta)$. Applying the previous procedure, we get a quasiconformal extension of $\phi_t^-$ with Beltrami coefficient of order $o(t)$, and $\phi_t^-+v_t$ satisfies the required properties in $\mathcal{U}\cup\mathrm{int}(\eta)$. 
\end{proof}

    \subsection{Dense subspaces of \texorpdfstring{$L^2(\nu)$}{L2(nu)}}\label{subsec:dense}

In this section, we introduce a dense subspace of $\cC\subset L^2(\nu)$. We will denote by $\cC^\infty_c(\R)$ the space of smooth, compactly supported functions on $\R$, and endow it with its usual Fr\'echet topology. The topological dual of $\cC^\infty_c(\R)$ is the space of distributions on $\R$, denoted $\cC^\infty_c(\R)'$. Then, we define 
\begin{equation}\label{E:cC}
    \cC:=\cC^\infty_c(\R)\otimes\C[(a_m,\bar a_m)_{m\geq1}] \qquad\text{and}
\end{equation}
\begin{equation}\label{E:cC_compact}
    \cC_\mathrm{comp}:=\mathrm{span}\left\lbrace\R\times\cE\ni(c,f)\mapsto F(c,f)\chi(\tilde c_0(f))\text{ for some }F\in\cC\text{ and }\chi\in\cC^\infty_c(\R)\right\rbrace,
\end{equation}
where we recall from \eqref{E:tilde_c_0} that $\tilde c_0:\cE\to\R$ is the $\log$-conformal radius viewed from $\infty$, and the linear span is in the algebraic sense. Importantly, by the Koebe quarter theorem, for all $F\in\cC_\mathrm{comp}$, there is $\epsilon>0$ such that the $\log$-conformal radii satisfy $c,\tilde{c}>\log\epsilon$ on the support of $F$.

The space $\cC$ will play the role of spaces of ``elementary functions'', which will serve as the common domains of our operators (before considering their closure).

\begin{lemma}\label{L:dense}
The space $\C[(a_m,\bar a_m)_{m\geq1}]$ is a dense subspace of the space $\cC^0(\cE)$ of continuous functions on $\cE$ and the space $\cC$ is a dense subspace of $L^2(\nu)$.
\end{lemma}

\begin{proof}
We first note that $\C[(a_m,\bar{a}_m)_{m\geq1}]\subset L^2(\nu^\#)$ by de Branges' theorem. Indeed, for any map $(z \in \D \mapsto z + \sum_{m \ge 1} a_m z^{m+1} ) \in \cE$, we can bound $|a_m| \le (m+1)$ and so any polynomial in $(a_m,\bar{a}_m)_{m \ge 1}$ is uniformly bounded on $\cE$ and in particular belongs to $L^2(\nu^\#)$.

Next, we prove that $\C[(a_m,\bar{a}_m)_{m\geq1}]$ is dense in $L^2(\nu^\#)$. Let $\cU$ be the space of univalent functions $f$ on $\D$, with $f(0)=0$ and $f'(0)=1$. This space comes equipped with the topology of uniform convergence of compacts of $\D$. This topology is metrisable, for instance with the distance $d_\cU(f_1,f_2):=\sum_{n=1}^\infty 2^{-n} \norm{f_1-f_2}_{\cC^0((1-2^{-n})\bar{\D})} \wedge 1$. Of course, we have a continuous inclusion $\cE\subset\cU$ and every $f\in\cU$ has a power series expansion of the form \eqref{eq:coordinates}. Therefore, the space $\C[(a_m,\bar{a}_m)_{\geq1}]$ is well defined as a subspace of $\cC^0(\cU)$.

According to \cite[Theorem 1.7]{Pommerenke75}, $\cU$ is compact. Moreover, $\C[(a_m,\bar{a}_m)_{m\geq1}]$ separates points: for all $f_1,f_2\in\cU$, we have $a_m(f_1)=a_m(f_2)$ for all $m\geq1$ if and only if $f_1=f_2$. Hence, $\C[(a_m,\bar{a}_m)_{m\geq1}]$ is dense in $\cC^0(\cU)$ by the Stone--Weierstrass theorem. In particular, $\C[(a_m,\bar{a}_m)_{m\geq1}]$ is dense in $\cC^0(\cE)$.

The spaces $\cE$ and $\cU$ are Polish (separable, completely metrisable), so they are also Radon spaces. Since $\nu^\#$ is a Borel measure on $\cE$ (Proposition \ref{prop:sle_is_borel}), it implies that it is a Radon measure. We can also view $\nu^\#$ as a Radon measure on $\cU$ giving full mass to $\cE$. Note that $\cC^0(\cU)$ is dense in $L^2(\nu^\#)$, since $\cU$ is compact Hausdorff \cite[Theorem 3.13]{Rudin74}. 

Now,
let $F\in L^2(\cE,\nu^\#)\simeq L^2(\cU,\nu^\#)$ and $\epsilon>0$. By the previous points, there exist $G\in\cC^0(\cU)$ and $P\in\C[(a_m,\bar{a}_m)_{m\geq1}]$ such that $\norm{F-G}_{L^2(\nu^\#)}\leq\frac{\epsilon}{2}$ and $\norm{P-G}_{\cC^0(\cU)}\leq\frac{\epsilon}{2}$. Since $\nu^\#$ is a probability measure, we have 
\[\norm{P-F}_{L^2(\nu^\#)}\leq\norm{P-G}_{L^2(\nu^\#)}+\norm{G-F}_{L^2(\nu^\#)}\leq\norm{P-G}_{\cC^0(\cU)}+\frac{\epsilon}{2}\leq\epsilon.\]
This proves that $\C[(a_m,\bar{a}_m)_{m\geq1}]$ is dense in $L^2(\nu^\#)$, and $\cC=\cC^\infty_c(\R)\otimes\C[(a_m,\bar a_m)_{m\geq1}]$ is dense in $L^2(\nu)=L^2(\d c\otimes\nu^\#)$.

To conclude, we prove that the space generated by finite linear combinations of functions the form $f\mapsto\chi(\tilde{c}_0)P(f)$ is also dense in $L^2(\nu^\#)$. Indeed, let $F\in L^2(\nu^\#)$ and $\epsilon>0$. We can fix a plateau function $\chi\in\cC^\infty_c(\R)$ such that $\norm{\chi F-F}_{L^2(\nu^\#)}\leq\frac{\epsilon}{2}$ (and $\norm{\chi}_{\cC^0(\R)}=1$). From the previous paragraph, we can also fix $P\in\C[(a_m,\bar a_m)_{m\geq1}]$ such that $\norm{P-F}_{L^2(\nu^\#)}\leq\frac{\epsilon}{2}$. It follows that $\norm{\chi P-F}_{L^2(\nu^\#)}\leq\norm{\chi P-\chi F}_{L^2(\nu^\#)}+\norm{\chi F-F}_{L^2(\nu^\#)}\leq\norm{\chi}_{\cC^0(\R)}\norm{P-F}_{L^2(\nu^\#)}+\frac\epsilon2\leq\epsilon$.
\end{proof}

\subsection{Witt algebra}\label{subsec:witt}

In this section, we construct and study two commuting representations $(\cL_\rv)_{\rv \in \C(z) \partial_z}$ and $(\bar \cL_\rv)_{\rv \in \C(z) \partial_z}$ of the Witt algebra as endomorphisms of the dense subspace $\cC$ \eqref{E:cC} of $L^2(\nu)$. The results of this section are similar to \cite[Section 3.2]{GQW24} and \cite[Section~3]{ChavezPickrell14}, which are obtained from the variational formulas of \cite{DurenSchiffer62}. Since we were not aware of this reference, we rederived these formulas from scratch using a different method, namely Ahlfors' variational formulas for quasiconformal mappings \cite{Ahlfors87_QCLectures}.

The intellectually satisfying aspect of working with quasiconformal maps is that we get an action of a whole \emph{group}, rather than a mere \emph{infinitesimal} action (the complex Witt algebra does not exponentiate since there is not complexification of $\mathrm{Diff}_+(\S^1)$ \cite{Segal87}). So we can view the Witt algebra as a subalgebra of the Lie algebra of infinitesimal Beltrami differentials, and the representations are obtained as (complex) Lie derivatives along the induced vector fields.

For the next definition, recall that $\ann_\eta\subset\C\setminus\{0\}$ is an annular region containing $\eta$ in its interior~\eqref{eq:annulus_eta}.  
\begin{definition}\label{Def:derivative}
Let $F\in L^2(\nu)$ and $\rv=v(z)\partial_z\in\C(z)\del_z$. We say that $F$ is differentiable at $\eta\in\cJ_{0,\infty}$ in direction $\rv$ if there exist $\cL_\rv F(\eta),\bar{\cL}_\rv F(\eta)\in\C$ such that,
\begin{equation}\label{eq:def_diff}
F(\phi_t(\eta))=F(\eta)+t\cL_\rv F(\eta)+\bar{t}\bar{\cL}_\rv F(\eta)+o(t),
\end{equation}
where $(\phi_t)$ is any $\cC^1$-family of conformal maps on $\ann_\eta$ (with respect to the local uniform topology) defined in a complex neighbourhood of $t=0$, and satisfying 
\begin{equation}\label{eq:phi}
\phi_0=\mathrm{id};\qquad\del_t\phi_t|_{t=0}=v;\qquad\del_{\bar t}\phi_t|_{t=0}=0.
\end{equation}

The function $F$ is said to be differentiable (resp. $\nu$-a.e.) in direction $\rv$ if it is differentiable at every (resp. for $\nu$-almost every) $\eta\in\cJ_{0,\infty}$ in direction $\rv$. In this case, the differentials $\cL_\rv F$, $\bar{\cL}_\rv F$ define complex-valued functions on $\cJ_{0,\infty}$ (resp. defined $\nu$-almost everywhere).
\end{definition}

Note that \eqref{eq:phi} is equivalent to $\phi_t(z)=z+tv(z)+o(t)$, where $|t|^{-1}o(t)\to0$ locally uniformly on $\ann_\eta$ as $t\to0$ in $\C$. One particular family satisfying the conditions of Definition \eqref{Def:derivative} is given by the flow of $\rv$: $\del_t\phi_t=v\circ\phi_t$, $\phi_0=\mathrm{id}_{\ann_\eta}$.
Note also that with this definition and for $\rv \equiv 0$, $\cL_\rv$ is well defined on $L^2(\nu)$ and $\cL_\rv \equiv 0$.

We start with a few remarks on the differentiable structure corresponding to Definition \ref{Def:derivative}.
Let $\mathrm{Hol}(\D)$ be the space of holomorphic functions on $\D$. It is a Fr\'echet space when endowed with the local uniform topology, and we refer to \cite[Section~I.3]{Hamilton} for background on the directional derivative in Fr\'echet spaces. We recall that this differential satisfies some basic properties such as the fundamental theorem of calculus \cite[Theorem I.3.2.2]{Hamilton} and linearity \cite[Theorem~I.3.2.5]{Hamilton}. Every $u\in\mathrm{Hol}(\D)$ has a power series expansion $u(z)=\sum_{m\geq0}u_mz^m$ around 0, and the coefficients $(u_m)_{m\geq0}$ provide coordinates on this space. We also have an embedding $i:\eta \in \mathcal{J}_{0,\infty}\mapsto e^{c_\eta}f_\eta\in\mathrm{Hol}(\D)$, and our two sets of coordinates are related by $e^c=u_1$, $e^ca_m=u_{m+1}$ for $m\geq1$ (note that $i(\mathcal{J}_{0,\infty})$ is a subset of the locus $\{u_1>0\}$).

In the next proposition, we show that differentiating in the space $\cJ_{0,\infty}$ in direction $\rv$ is equivalent to differentiating in $\mathrm{Hol}(\D)$ along the directions $u_\eta(\rv)$ and $\tilde u_\eta(\rv)$ where, for all $z \in \D$,
\begin{align}\label{E:def_urv}
    u_\eta(\rv)(z) & = \frac{e^{c_\eta} f_\eta'(z)}{4i\pi} \oint_{\gamma}\frac{z((e^{c_\eta}f_\eta)^{-1}(\zeta)+z)((e^{c_\eta}f_\eta)^{-1})'(\zeta)^2}{(e^{c_\eta}f_\eta)^{-1}(\zeta)^2((e^{c_\eta}f_\eta)^{-1}(\zeta)-z)} v(\zeta) \d\zeta,\\
    & = \frac{1}{4i\pi} \oint_{(e^{c_\eta}f_\eta)^{-1}(\gamma)}\frac{z(\zeta+z)}{\zeta^2(\zeta-z)} \frac{f_\eta'(z)}{f_\eta'(\zeta)} v(e^{c_\eta}f_\eta(\zeta))\d\zeta,\\
    \label{E:def_urv_tilde}
    \overline{\tilde u_\eta(\rv)(z)} & = - \frac{e^{c_\eta} \overline{f_\eta'(z)}}{4i\pi} \oint_{\gamma}\frac{\bar z((e^{c_\eta}f_\eta)^{-1}(\zeta)+\bar z)((e^{c_\eta}f_\eta)^{-1})'(\zeta)^2}{(e^{c_\eta}f_\eta)^{-1}(\zeta)^2(1-\bar z(e^{c_\eta}f_\eta)^{-1}(\zeta))} v(\zeta) \d\zeta.
\end{align}
The contour $\gamma$ in the integral defining $u_\eta(\rv)(z)$ (resp. $\tilde u_\eta(\rv)(z)$) can be any contour surrounding the origin and $e^{c_\eta}f_\eta(z)$ (resp. surrounding the origin) and included in $\rmint(\eta)$.
When $z \in \D\setminus \{0\}$ and if $\gamma$ is a contour surrounding the origin but not $e^{c_\eta}f_\eta(z)$, we have
\begin{align}\label{E:urveta_small_loop}
    u_\eta(\rv)(z) = \frac{e^{c_\eta} f_\eta'(z)}{4i\pi} \oint_{\gamma}\frac{z((e^{c_\eta}f_\eta)^{-1}(\zeta)+z)((e^{c_\eta}f_\eta)^{-1})'(\zeta)^2}{(e^{c_\eta}f_\eta)^{-1}(\zeta)^2((e^{c_\eta}f_\eta)^{-1}(\zeta)-z)} v(\zeta) \d\zeta + v \circ (e^{c_\eta}f_\eta)(z).
\end{align}

\begin{proposition}\label{P:equivalence_derivative}
    Let $\eta \in \cJ_{0,\infty}$, $\rv = v(z)\del_z \in \C(z)\del_z$ and $(\phi_t)$ be a $\cC^1$-family of conformal maps on $O_\eta$ \eqref{eq:annulus_eta} (with respect to the local uniform topology), defined in a complex neighbourhood of $t=0$, satisfying $\phi_0 = \id$, $\del_t \phi_t\vert_{t=0} = v$ and $\del_{\bar t} \phi_t\vert_{t=0} = v$. Denote by $e^{c_t} f_t$ the uniformising map corrresponding to the perturbed loop $\phi_t(\eta)$. Then the following expansion holds as $t \to 0$, uniformly on compact subsets of $\D$:
    \begin{equation}\label{E:P_equiv}
        \hspace{30pt} e^{c_t} f_t = e^{c_\eta} f_\eta + t u_\eta(\rv) + \bar t \tilde u_\eta(\rv) + o(t).
    \end{equation}
\end{proposition}

The rest of Section \ref{subsec:witt} is organised as follows. Section~\ref{subsubsec:statements_witt} contains the main statements about the Witt representations (which are fairly quick consequences of Proposition \ref{P:equivalence_derivative} above): Propositions \ref{P:construction_cL}, \ref{prop:generators_mobius} and \ref{prop:diff_c}. Proposition \ref{P:equivalence_derivative} will be proved in Section~\ref{SSS:proof_Prop_equiv}. Finally, Propositions~\ref{P:construction_cL}, \ref{prop:generators_mobius} and \ref{prop:diff_c} will be proved in Sections \ref{SSS:L0L1}, \ref{SSS:c} and \ref{SSS:construction}, respectively.

        \subsubsection{Statements on the Witt algebra}\label{subsubsec:statements_witt}

The first main result concerns the construction of the operators $\cL_\rv$ and $\bar\cL_\rv$, $\rv \in \C(z)\partial_z$, as endomorphisms of $\cC$ \eqref{E:cC} and $\cC_\mathrm{comp}$ \eqref{E:cC_compact} as well as the computation of their commutation relations:

\begin{proposition}\label{P:construction_cL}
Let $\rv\in \C(z)\del_z$. Every $F\in\cC$ (resp. $\cC_\mathrm{comp}$) is differentiable in direction $\rv$, and $\cL_\rv F,\bar{\cL}_\rv F\in\cC$ (resp. $\cC_\mathrm{comp}$). The maps $F\mapsto\cL_\rv F$, $F\mapsto\bar{\cL}_\rv F$ define endomorphisms of $\cC$ and $\cC_\mathrm{comp}$, and their commutators satisfy
\[[\cL_\rv,\cL_\rw]=\cL_{[\rv,\rw]},\qquad[\bar{\cL}_\rv,\bar{\cL}_\rw]=\bar{\cL}_{[\rv,\rw]}\quad \text{and} \quad[\cL_\rv,\bar{\cL}_\rw]=0,\]
where $[\rv,\rw]$ is the Lie bracket of vector fields \eqref{E:bracket_vector_field}.

Recalling that we denote for all $n \in \Z$, $\rv_n = -z^{n+1}\partial_z$, we will simply write $\cL_n$ for $\cL_{\rv_n}$.
\end{proposition}

\begin{remark}\label{rem:mobius}
Let $\nu_{\hat\C}$ be the SLE loop measure in $\hat\C$, not restricted to $\cJ_{0,\infty}$.
We would like to explain that, when the vector field $\rv \in \C_2[z] \partial_z$ is a quadratic polynomial vector field, one can directly build the operator $\cL_\rv^{\hat\C}$ on $L^2(\nu_{\hat\C})$ using conformal perturbations and Stone's theorem. The exponent $\hat\C$ emphasises that the operator $\cL_\rv^{\hat\C}$ is defined on a different set. The relation to $\cL_\rv$ is as follows: if $F \in L^2(\nu)$, then $\hat F:\eta \in \cJ \mapsto F(\eta) \mathbf{1}_{\eta \in \cJ_{0,\infty}}$ belongs to $L^2(\nu_{\hat\C})$ and
\[ 
\cL_\rv^{\hat\C}( \hat F )(\eta) = \cL_\rv F(\eta) \mathbf{1}_{\eta \in \cJ_{0,\infty}}.
\]
Recall that $\C_2[z]\del_z\simeq\mathfrak{sl}_2(\C)$ is the Lie algebra of the group $\mathrm{PSL}_2(\C)$ of Möbius transformations of the Riemann sphere $\hat \C$ which acts on $\cJ$ by $\Phi\cdot\eta:=\Phi(\eta)$, for all $\eta \in \cJ$, $\Phi \in \mathrm{PSL}_2(\C)$.
This action defines a unitary group of operators $(\bU_\Phi)_{\Phi\in\mathrm{PSL}_2(\C)}$ on $L^2(\nu_{\hat\C})$:
 \[\bU_\Phi F(\eta):=F(\Phi(\eta)).\]
Using \cite[Theorem VIII.9]{ReedSimon1}, one can easily check that the mapping $\Phi\mapsto\bU_\Phi$ is a strongly continuous group of unitary operators on $L^2(\nu_{\hat\C})$.
By Stone's theorem \cite[Theorem VIII.8]{ReedSimon1}, this fact implies that for every $\rv\in\C_2[z]\del_z\simeq\mathfrak{sl}_2(\C)$, there exists a (unique) self-adjoint operator $\bH_\rv^{\hat\C}$ on $L^2(\nu_{\hat\C})$ such that 
\[
\bU_{e^\rv}=e^{i\bH_\rv^{\hat\C}}.
\]
In the above display, $\C_2[z]\del_z\simeq\mathfrak{sl}_2(\C)$ is a six-dimensional real vector space and the exponential of $\rv\in\C_2[z]\del_z$ corresponds to the matrix exponential in $\mathfrak{sl}_2(\C)$.
The generator of dilations is the operator
\[\bH^{\hat\C}:=\bH_{-z\del_z}^{\hat\C}.\]
Following the standard terminology in conformal field theory \cite{Segal87}, we call it the \emph{Hamiltonian}. The operators $\cL_\rv^{\hat\C}$ and $\bar \cL_\rv^{\hat\C}$ can then be obtained by setting
\[
\cL_\rv^{\hat\C}:=\frac{i}{2}(\bH_\rv^{\hat\C}-i\bH_{i\rv}^{\hat\C});\qquad\bar{\cL}_\rv^{\hat\C}:=\frac{i}{2}(\bH_\rv^{\hat\C}+i\bH_{i\rv}^{\hat\C}).
\]
Note that the map $\bH_{\rv}^{\hat\C}$ depends only $\R$-linearly on $\rv$, while $\cL_\rv^{\hat\C}$ (resp. $\bar\cL_\rv^{\hat\C}$) depends $\C$-linearly (resp. $\C$-antilinearly) on $\rv$.
\end{remark}

The next proposition gives the explicit expression for $\cL_0$, $\bar \cL_0$, $\cL_{-1}$ and $\bar \cL_{-1}$, which will be of use in the sequel. We mention that expressions with a similar flavour have been given for Kirillov's operators; see for instance \cite[(7-0') and (7-1')]{Kirillov98} and recall that the difference with our setup is explained in Section~\ref{subsec:comparison}.

\begin{proposition}\label{prop:generators_mobius}
We have
\begin{align*}
&\cL_0=-\frac{\del_c}{2}+\frac{1}{2}\sum_{m=1}^\infty m(a_m\del_{a_m}-\bar{a}_m\del_{\bar{a}_m}),\qquad\bar{\cL}_0=-\frac{\del_c}{2}-\frac{1}{2}\sum_{m=1}^\infty m(a_m\del_{a_m}-\bar{a}_m\del_{\bar{a}_m}),\\
&\cL_{-1}=e^{-c}a_1\del_c+e^{-c}\sum_{m=1}^\infty(m+2)(a_{m+1}-a_1a_m)\del_{a_m}+m(a_1\bar{a}_m-\bar{a}_{m-1})\del_{\bar{a}_m},\\
&\bar\cL_{-1}=e^{-c}\bar a_1\del_c+e^{-c}\sum_{m=1}^\infty(m+2)(\bar a_{m+1}-\bar a_1\bar a_m)\del_{\bar a_m}+m(\bar a_1a_m-a_{m-1})\del_{a_m},
\end{align*}
where we have set $a_0=1$.
\end{proposition}

Finally, we consider the case of the coordinate function $c: \cJ_{0,\infty} \to \R$ and compute explicitly $\cL_\rv c$ for all vector fields $\rv$.
To this end, define
$\varpi:\cJ_{0,\infty}\times\C(z)\del_z\to\C$ by
\begin{equation}
    \label{eq:def_varpi}
    \varpi_\eta(\rv):=\frac{1}{4i\pi}\oint\left(\frac{((e^{c_\eta}f_\eta)^{-1})'(\zeta)}{(e^{c_\eta}f_\eta)^{-1}(\zeta)}\right)^2v(\zeta)\d \zeta.
\end{equation}

\begin{proposition}\label{prop:diff_c}
For all $\rv\in\C(z)\del_z$, the coordinate $c:\cJ_{0,\infty}\to\R$ is differentiable in direction $\rv$, and
\[\cL_\rv c=\varpi(\rv).\]
\end{proposition}
This result has been known for other topologies (smooth Jordan curves \cite{KirillovYurev87}, Weil-Petersson class \cite{TakhtajanTeo06}). Our strategy will follow closely the proof of \cite[Part 2, Theorem 4.5]{TakhtajanTeo06}.

The rest of Section \ref{subsec:witt} is dedicated to the proof of Propositions \ref{P:construction_cL}, \ref{prop:generators_mobius} and \ref{prop:diff_c}.

\subsubsection{Proof of Proposition \ref{P:equivalence_derivative}}\label{SSS:proof_Prop_equiv}
Let $\eta \in \cJ_{0,\infty}$, $\rv = v(z)\del_z = -\sum_{n \in \Z} v_n z^{n+1} \del_z \in \C(z)\del_z$ and $(\phi_t)$ be a $\cC^1$-family of conformal maps on $O_\eta$ \eqref{eq:annulus_eta}, defined in a complex neighbourhood of $t=0$, satisfying $\phi_0 = \id$, $\del_t \phi_t\vert_{t=0} = v$ and $\del_{\bar t} \phi_t\vert_{t=0} = v$. Denote by $e^{c_t} f_t$ and $e^c f$ the uniformising map corrresponding to the loops $\phi_t(\eta)$ and $\eta$ respectively. We want to study the asymptotic behaviour of $e^{c_t}f_t$ as $t\to 0$ in the space $\mathrm{Hol}(\D)$.
Define for $z \in \D\setminus\{0\}$
\begin{gather}\label{E:def_omega_rv}
    \omega_\eta(\rv)(z) = \frac1{4i\pi} \oint\frac{z((e^cf)^{-1}(\zeta)+z)((e^cf)^{-1})'(\zeta)^2}{(e^cf)^{-1}(\zeta)^2((e^cf)^{-1}(\zeta)-z)} v(\zeta) \d\zeta,\\
    \label{E:def_omega_rv_tilde}
    \overline{\tilde\omega_\eta(\rv)(z)} = - \frac{1}{4i\pi} \oint_{\gamma}\frac{\bar z((e^{c_\eta}f_\eta)^{-1}(\zeta)+\bar z)((e^{c_\eta}f_\eta)^{-1})'(\zeta)^2}{(e^{c_\eta}f_\eta)^{-1}(\zeta)^2(1-\bar z(e^{c_\eta}f_\eta)^{-1}(\zeta))} v(\zeta) \d\zeta,
\end{gather}
where the contours are small loops surrounding 0 but not $e^cf(z)$. As observed in \eqref{E:urveta_small_loop}, the functions $u_\eta(\rv)$ and $\tilde u_\eta(\rv)$ agree with
\begin{gather}
    u_\eta(\rv) =v\circ(e^cf)+ \omega_\eta(\rv) e^{c_\eta} f_\eta'
    \quad\text{and}\quad 
    \tilde u_\eta(\rv) = \tilde \omega_\eta(\rv) e^{c_\eta} f_\eta'
\end{gather}
We decompose $\rv=\rv_++\rv_-$ with
\[
\rv_+ = -\sum_{n \ge -1} v_n z^{n+1} \del_z \in\C[z]\del_z
\quad\text{and}\quad
\rv_- = - \sum_{n \le -2} v_n z^{n+1} \del_z \in z^{-1}\C[z^{-1}]\del_z.
\]
As noted in \eqref{eq:annulus_eta}, $\eta$ is compactly included in $O_\eta$, so if $t$ is small enough, $\phi_t(\eta) \subset O_\eta$ \eqref{eq:annulus_eta}.
As in Lemma~\ref{lem:extend_qc}, we let
\begin{equation}\label{E:blipblip}
    \mu(z)=-\sum_{n=2}^\infty nv_{-n}r^{-2n}\bar{z}^{n-1}z\ind_{|z|<r}, \quad \text{with} \quad r=\frac{e^c}{8} \quad \text{and let} \quad \Phi_t:=(\id + tv_+)\circ \Phi_{t\mu}.
\end{equation}
By Lemma \ref{lem:extend_qc}, $\Phi_t$ and $\phi_t$ agree on $O_\eta$ up to $o(t)$. This choice of extension changes $e^{c_t} f_t$ by at most $o(t)$: indeed, by Lemma \ref{L:derivative0}, one can alternatively consider a quasiconformal extension of $\phi_t$ that exactly agrees with $\phi_t$ in a neighbourhood of $\eta$ and whose Beltrami differential equals $t\mu+o(t)$. The rest of the argument below will then show that $e^{c_t} f_t$ has been changed by at most $o(t)$. In the rest of the proof, we consider the perturbation by $\Phi_t$ so as to have a fixed Beltrami differential.
We keep denoting by $\rv_-+\rv_+$ the vector field in $\hat{\C}$ which generates $\Phi_t$. Recall that $\rv_-=\delbar\mu$ and is given by the variational formula of Lemma~\ref{lem:extend_qc}. 

We will show that the following expansion holds uniformly in $\bar\D$ as $t \to 0$:
\begin{equation}\label{E:L_useful}
        \Psi_t:= (e^{c_t} f_t)^{-1} \circ \Phi_t \circ (e^c f) = \id - t\omega_\eta(\rv) - \bar t\tilde \omega_\eta(\rv)+ o(t),
    \end{equation}
where $\omega_\eta(\rv)$ and $\tilde \omega_\eta(\rv)$ are defined in \eqref{E:def_omega_rv} and \eqref{E:def_omega_rv_tilde}.
Let $M_t$ be the unique Möbius transformation of $\D$ that sends $\Psi_t(0) = - tv_{-1}e^{-c}+o(t)$ to 0 and $\Psi_t(1) =: e^{i\theta_t}$ to 1:
\begin{equation}
    \label{E:Mt}
M_t(z) = \frac{1- \overline{\Psi_t(0)} e^{i\theta_t}}{e^{i\theta_t} - \Psi_t(0)} \frac{z-\Psi_t(0)}{1-\overline{\Psi_t(0)}z} = e^{-i\theta_t}(1-tv_{-1}e^{-c}+\overline{tv_{-1}e^{-c}}+o(t))\frac{z-\Psi_t(0)}{1-\overline{\Psi_t(0)}z}, \quad z \in \bar\D,
\end{equation}
and let $\tilde\Phi_t := M_t \circ \Psi_t$.
An elementary computation using the chain rule (see also below) shows that $\del_{\bar z} \tilde\Phi_t$ vanishes in some annular neighbourhood $\{z \in \D: |z|>1-\eps\}$ of $\S^1$ in $\D$. The map $\tilde\Phi_t$ is thus analytic in that set. It also restricts to a diffeomorphism of $\S^1$. Reflecting it across $\S^1$ by defining $\tilde\Phi_t (z) = 1/\overline{\tilde\Phi_t(1/\bar z)}$ for $z \in \hat\C\setminus \bar\D$ yields a quasiconformal homeomorphism of $\hat{\C}$.
The Möbius $M_t$ guarantees that $\tilde\Phi_t$ fixes $0$, $1$ and $\infty$, i.e. that it is normalised in the sense of Definition \ref{def:qc}, second item.
It satisfies
\begin{equation}\label{eq:intertwine_by_qc}
e^{c_t}f_t\circ(M_t^{-1}\circ\tilde{\Phi}_t)=\Phi_t\circ(e^cf).
\end{equation}

It is easy to evaluate the Beltrami differential generating $\tilde{\Phi}_t$. Indeed, the functions $M_t$, $f_t$ and $f$ are holomorphic in $\D$, so the chain rule gives
\begin{align*}
&\del_{\bar{z}}\left( M_t\circ(e^{c_t}f_t)^{-1}\circ\Phi_t\circ(e^cf)\right)=(M_t\circ(e^{c_t}f_t)^{-1})'\circ\Phi_t\circ(e^cf)\,\del_{\bar{z}}\Phi_t\circ(e^cf)\,e^c\overline{f'};\\
&\del_z\left(M_t\circ(e^{c_t}f_t)^{-1}\circ\Phi_t\circ(e^cf)\right)=(M_t\circ(e^{c_t}f_t)^{-1})'\circ\Phi_t\circ(e^cf)\,\del_z\Phi_t\circ(e^cf)\,e^cf',
\end{align*}
so that
\[\frac{\del_{\bar{z}}\tilde{\Phi}_t}{\del_z\tilde{\Phi}_t}=\frac{\del_{\bar{z}}\Phi_t\circ(e^cf)}{\del_z\Phi_t\circ(e^cf)}\frac{\overline{f'}}{f'}=\left(\frac{\del_{\bar{z}}\Phi_t}{\del_z\Phi_t}\right)\circ(e^cf)\frac{\overline{f'}}{f}.\]
Hence, the Beltrami differential generating $\tilde{\Phi}_t$ coincides with the pullback $(e^cf)^*\mu=\mu\circ(e^cf)\frac{\overline{f'}}{f'}$ in $\D$ (recall \eqref{eq:pullback_beltrami}), and is extended to the whole sphere by reflection. In particular, by Proposition~\ref{lem:variational_formulas},
\begin{equation}\label{E:estimate11}
    \tilde{\Phi}_t(z) = z + t w(z) + \bar t \tilde w (z) + o(t)
\end{equation}
uniformly on compact subsets, where
\begin{align*}
    w(z) = -\frac{1}{\pi}\int_\D\frac{z(z-1)(e^c f)^*\mu(\zeta)}{\zeta(\zeta-1)(\zeta-z)}|\d\zeta|^2, \quad
    \tilde w(z) = \frac{1}{\pi}\int_\D\frac{z(z-1)\overline{(e^c f)^*\mu(\zeta)}}{\bar{\zeta}(\bar{\zeta}-1)(1-z\bar{\zeta})}|\d\zeta|^2.
\end{align*}
By doing a change of variable, we can rewrite
\begin{align*}
    w(z) = -\frac{1}{\pi}\int_{\rmint(\eta)}\frac{z(z-1)((e^c f)^{-1})'(\zeta)^2 \mu(\zeta)}{(e^c f)^{-1}(\zeta)((e^c f)^{-1}(\zeta)-1)((e^c f)^{-1}(\zeta)-z)}|\d\zeta|^2
\end{align*}
By \eqref{E:del_bar_v}, $\del_{\bar z} v_- = \mu$ so by Stokes' formula, we deduce that
\begin{equation}
    \label{E:expressionw1}
w(z) = -\frac{1}{2i\pi}\oint\frac{z(z-1)((e^cf)^{-1})'(\zeta)^2}{(e^cf)^{-1}(\zeta)((e^cf)^{-1}(\zeta)-1)((e^cf)^{-1}(\zeta)-z)} v_-(\zeta) \d\zeta.
\end{equation} 
Similarly,
\begin{equation}
    \label{E:expressionw2}
\overline{\tilde w(z)} = \frac{1}{2i\pi}\oint\frac{\bar z(\bar z-1)((e^cf)^{-1})'(\zeta)^2}{(e^cf)^{-1}(\zeta)((e^cf)^{-1}(\zeta)-1)(1-\bar z(e^cf)^{-1}(\zeta))} v_-(\zeta) \d\zeta, \quad z \in \C.
\end{equation}

Now that we know the behaviour of $\tilde\Phi_t$, it remains to study the angle $\theta_t$ appearing in the Möbius $M_t$ \eqref{E:Mt} to get back to $\Psi_t$.
Applying $\del_z$ to both sides of \eqref{eq:intertwine_by_qc} and specialising to $z=0$ gives
\begin{equation}\label{eq:var_c_theta}
e^{c_t-c+i\theta_t}(1 + tv_{-1}e^{-c}-\overline{tv_{-1}e^{-c}}-tv_{-1}e^{-c}f''(0))+o(t)=\frac{\del_z\Phi_t(0)}{\del_z\tilde{\Phi}_t(0)}.
\end{equation}
Now, we use the variational formulas of Proposition \ref{lem:variational_formulas} and Lemma \ref{lem:extend_qc} in order to expand $\del_z\tilde\Phi_t(0)$ and $\del_z\Phi_t(0)$, respectively. Recall also that the $o(t)$ in the aforementioned results have also negligible derivatives compared to $t$; see the discussion above Proposition \ref{lem:variational_formulas}. Using that the $\del_z$-derivative of the Cauchy transform is the Beurling transform \cite[Chapter V, Lemma 3]{Ahlfors87_QCLectures}, we get
\begin{align*}
&\del_z\Phi_t(z)=1-\frac{t}{\pi}\int_\D\frac{\mu(\zeta)}{(z-\zeta)^2}|\d\zeta|^2+t v_+'(z) + o(t)\\
&\del_z\tilde{\Phi}_t(z)=1-\frac{t}{\pi}\int_\D\frac{(2z-1)(\zeta-z)+z(z-1)}{(\zeta-z)^2}\frac{(e^cf)^*\mu(z)}{\zeta(\zeta-1)}|\d\zeta|^2\\
&\quad\qquad\qquad+\frac{\bar{t}}{\pi}\int_\D\frac{(2z-1)(1-z\bar{\zeta})+z(z-1)\bar{\zeta}}{(1-z\bar{\zeta})}\frac{\overline{(e^cf)^*\mu(z)}}{\bar{\zeta}(\bar{\zeta}-1)}|\d\zeta|^2+o(t),
\end{align*}
where the integrals are in the principal value sense. Evaluating these integrals at $z=0$ gives
\begin{align*}
&\del_z\Phi_t(0)=1-\frac{t}{\pi}\int_{\D}\mu(\zeta)\frac{|\d\zeta|^2}{\zeta^2}-tv_0+o(t);\\
&\del_z\tilde{\Phi}_t(0)=1+\frac{t}{\pi}\int_\D\frac{(e^cf)^*\mu(\zeta)}{\zeta^2(\zeta-1)}|\d\zeta|^2-\frac{\bar{t}}{\pi}\int_\D\frac{\overline{(e^cf)^*\mu(z)}}{\bar{\zeta}(\bar{\zeta}-1)}|\d\zeta|^2+o(t).
\end{align*}
As before, using that $\del_{\bar z} v_- = \mu$ \eqref{E:del_bar_v}, the exterior derivative of $\zeta^{-2} v_-(\zeta) \d\zeta$ agrees with $2i \zeta^{-2} \mu(\zeta)|\d\zeta|^2$ and so by Stokes' formula, 
\begin{align*}
    \int_{\D}\mu(\zeta)\frac{|\d\zeta|^2}{\zeta^2} = \lim_{\eps \to 0} \int_{\D\setminus \eps \bar\D}\mu(\zeta)\frac{|\d\zeta|^2}{\zeta^2} = \frac1{2i} \lim_{\eps\to0} \Big(\oint_{\S^1} v_-(\zeta) \frac{\d\zeta}{\zeta^2} - \oint_{\eps\S^1} v_-(\zeta) \frac{\d\zeta}{\zeta^2}\Big).
\end{align*}
On $\S^1$, $v_-$ has an expansion in $\zeta^{n+1}$ with $n<0$; on $\varepsilon\S^1$ with $\varepsilon<r$, the expansion is in $\bar\zeta^n\zeta$ with $n<0$. Hence, the two contour integrals vanish. Going back to \eqref{eq:var_c_theta}, we get that
\begin{align}\label{E:ctite}
    & e^{c_t-c+i\theta_t}(1 + tv_{-1}e^{-c}-\overline{tv_{-1}e^{-c}}-tv_{-1}e^{-c}f''(0))\\
    \notag
    & =1-tv_0-\frac{t}{\pi}\int_\D\frac{(e^cf)^*\mu(\zeta)}{\zeta^2(\zeta-1)}|\d\zeta|^2+\frac{\bar{t}}{\pi}\int_\D\frac{\overline{(e^cf)^*\mu(z)}}{\bar{\zeta}(\bar{\zeta}-1)}|\d\zeta|^2+o(t)
\end{align}
and
\begin{align*}
& e^{i\theta_t}(1 + tv_{-1}e^{-c}-\overline{tv_{-1}e^{-c}}) \\
& = \big(e^{c_t-c+i\theta_t}(1 + tv_{-1}e^{-c}-\overline{tv_{-1}e^{-c}}) \big)^{1/2} \big(\overline{e^{c_t-c+i\theta_t}(1 + tv_{-1}e^{-c}-\overline{tv_{-1}e^{-c}})} \big)^{-1/2} + o(t) \\
& = 1 + t a - \overline{t a} + o(t), \qquad \text{where} \qquad
a = \frac{-v_0+v_{-1}e^{-c}f''(0)}{2} - \frac{1}{2\pi}\int_\D\frac{(\zeta+1)(e^cf)^*\mu(\zeta)}{\zeta^2(\zeta-1)}|\d\zeta|^2.
\end{align*}
By Stokes' formula, we find that
\begin{equation}
    \label{E:expressiona}
a=\frac{-v_0+v_{-1}e^{-c}f''(0)}{2} - \frac{1}{4i\pi}\oint \frac{((e^cf)^{-1}(\zeta)+1) ((e^cf)^{-1})'(\zeta)^2}{(e^cf)^{-1}(\zeta)^2((e^cf)^{-1}(\zeta)-1)} v_-(\zeta)\d\zeta.
\end{equation}
We deduce that
\begin{equation}
    \label{E:estimate12}
M_t(z) = (1+\overline{ta}-ta+o(t)) \frac{z-\Psi_t(0)}{1-\overline{\Psi_t(0)}z}.
\end{equation}

Altogether, the estimate \eqref{E:estimate11} on $\tilde\Phi_t$ and the estimate \eqref{E:estimate12} on $M_t$ yield
\begin{align*}
    \Psi_t(z) & = M_t^{-1} \circ \tilde\Phi_t(z) = (1+ta-\overline{ta}) \tilde\Phi_t -tv_{-1}e^{-c} + \overline{tv_{-1}e^{-c}} z^2 + o(t) \\
    & = z + t(w(z) + a z-v_{-1}e^{-c}) + \bar t(\tilde w(z) - \overline{a} z +\overline{v_{-1}e^{-c}}z^2) + o(t).
\end{align*}
Using the expressions \eqref{E:expressionw1} and \eqref{E:expressiona} of $w$ and $a$, we find that
\begin{align*}
    w(z) + a z-v_{-1}e^{-c}
    = \frac{-v_0+v_{-1}e^{-c}f''(0)}{2}z -v_{-1}e^{-c} - \omega_\eta(\rv_-)(z),
\end{align*}
where $\omega_\eta(\rv_-)$ is defined in \eqref{E:def_omega_rv}.
Elementary computations using the residue theorem show that
$\omega_{\rv_0}(\eta)(z) = z/2$,
$\omega_\eta(\rv_{-1})(z) = e^{-c}(1-zf''(0)/2)$ and $\omega_{\rv_n}(\eta)=0$ for all $n \ge 1$. We have thus identified $w(z) + a z-v_{-1}e^{-c}$ with $-\omega_\eta(\rv_+)(z)-\omega_\eta(\rv_-)(z) = \omega_\eta(\rv)(z)$.
Similarly, the $\bar t$ coefficient in the expansion of $\Psi_t$ agrees with $-\tilde\omega_\eta(\rv)(z)$ which concludes the proof of \eqref{E:L_useful}.

Now, since $\Phi_t(z) = z + tv(z) + o(t)$ uniformly in $\bar\D \setminus r\bar\D$, we get that uniformly on compacts of $\D$ minus a neighbourhood of the origin,
\begin{align*}
    e^{c_t}f_t
=\Phi_t\circ(e^cf)\circ\Psi^{-1}_t
=e^cf+tv\circ(e^cf)&+t\omega_\eta(\rv)e^cf'+\bar{t}\tilde \omega_\eta(\rv)e^cf'+o(t).
\end{align*}
By Cauchy's integral formula, this asymptotic expansion is valid on compact subsets of $\D$.
This concludes the proof of \eqref{E:P_equiv} and completes the proof of Proposition \ref{P:equivalence_derivative}.
\qed

\subsubsection{Existence of derivatives on the space \texorpdfstring{$\cC$}{C}}\label{SSS:construction}

\begin{proof}[Proof of Proposition \ref{P:construction_cL}]
Existence of derivatives follows by Proposition \ref{P:equivalence_derivative} and the existence of directional derivatives in Fréchet spaces. Let us give more details.
Recall from the discussion preceding Proposition \ref{P:equivalence_derivative} that $\cJ_{0,\infty}$ is embedded in the space $\mathrm{Hol}(\D)$ of holomorphic functions $u(z) = \sum_{m\ge 0} u_m z^m$ in $\D$ by the map $i:\eta \in \mathcal{J}_{0,\infty}\mapsto e^{c_\eta}f_\eta\in\mathrm{Hol}(\D)$.
In Proposition \ref{P:equivalence_derivative}, we showed that a small motion $\phi_t(\eta)$ of $\eta$, where $\phi_t(z) = z + t v(z) + o(t)$ is as in Definition \ref{Def:derivative}, induces a small motion $e^{c_t} f_t = e^{c_\eta} f_\eta + tu_\eta(\rv) + \bar t \tilde u_\eta(\rv) + o(t)$ in the space $\mathrm{Hol}(\D)$.
Now every $G:\mathrm{Hol}(\D)\to\C$ induces a function $F:\cJ_{0,\infty}\to\C$ by the formula $F(\eta):=G(e^{c_\eta}f_\eta)$. Moreover, if $G$ is continuously differentiable (in the sense of \cite[Section I.3]{Hamilton}, i.e. admits continuous directional derivatives), then as $t\to0$:
\begin{align*}
F(\phi_t(\eta))
&=G(e^cf+tu_\rv+\bar t\tilde u_\rv+o(t))\\
&=G(e^cf)+t(DG_{e^cf}(u_\rv)+\bar DG_{e^cf}(\tilde u_\rv))+\bar t(DG_{e^cf}(\tilde u_\rv)+\bar DG_{e^cf}(u_\rv))+o(t),
\end{align*} 
where $DG_ab,\bar DG_ab$ denote the complex directional derivative of $G$ at $a\in\mathrm{Hol}(\D)$ in direction $b\in\mathrm{Hol}(\D)$. This means that $F$ is differentiable in direction $\rv$ at every $\eta\in\cJ_{0,\infty}$ in the sense of Definition \ref{Def:derivative}, with $\cL_\rv F=DG(u_\rv)+\bar DG(\tilde u_\rv)$, $\bar\cL_\rv F=DG(\tilde u_\rv)+\bar DG(u_\rv)$. Since $\C(z)\del_z\to\mathrm{Hol}(\D),\,\rv\mapsto u_\rv$ (resp. $\C(z)\del_z\to\mathrm{Hol}(\D),\,\rv\mapsto\tilde u_\rv$) is $\C$-linear (resp. $\C$-antilinear), and $DG$ (resp. $\bar DG$) is $\C$-linear (resp. $\C$-antilinear) on $\C(z)\del_z$, we get that $\rv\mapsto\cL_\rv F(\eta)$ (resp. $\rv\mapsto\bar\cL_\rv F(\eta)$) is a $\C$-linear (resp. $\C$-antilinear) form on $\C(z)\del_z$. Now, every $F\in\cC$ is induced by a function $G$ on $\mathrm{Hol}(\D)$, which is a smooth function of finitely many Taylor coefficients, and is in particular of class $\cC^1$. This concludes the proof.

It remains to prove that $\cL_\rv$ and $\bar\cL_\rv$ preserve $\cC$ and $\cC_\mathrm{comp}$ and to compute their commutators.
Proposition~\ref{prop:diff_c} will give an expression for the derivative of $c$ so we focus on the other coefficients in this proof.
By linearity, we can assume that $\rv=\rv_n$ for some $n \in \Z$. Proposition~\ref{prop:generators_mobius} will give explicit expressions for $\cL_0$ and $\cL_{-1}$, so we only treat the following cases in this proof.

\noindent\textbf{Case 1. $\cL_n$ for $n \ge 1$.}
Applying \eqref{E:P_equiv} to $\rv=-z^{n+1}\del_z$, we see that $c_t=c+o(t)$ (due to $n\geq1$), and $a_m(t)=a_m+te^{nc}P_{n,m}+o(t)$, where $P_{n,m}(f)=-\frac{1}{2i\pi}\oint f(z)^{n+1}\frac{\d z}{z^{m+2}}$ is a polynomial in $(a_m)_{m\geq1}$. Therefore, we have a (unique) first order expansion
\[F(\phi_t(\eta))=F(\eta)+t\cL_n F(\eta)+\bar{t}\bar{\cL}_n F(\eta)+o(t),\]
where $\cL_n=e^{nc}\sum_{m=1}^\infty P_{n,m}\del_{a_m}$ and $\bar{\cL}_n=e^{nc}\sum_{m=1}^\infty \bar{P}_{n,m}\del_{\bar{a}_m}$. These reduce to finite sums when acting on $\cC$, so we get indeed operators $\cL_n,\bar{\cL}_n$ preserving $\cC$.

\noindent\textbf{Case 2. $\cL_{-n}$ for $n \ge 2$.}
As already mentioned, Proposition \ref{prop:diff_c} shows that $c_t=c+t\varpi(\rv)+\overline{t\varpi(\rv)}+o(t)$, so we only need to focus on the other coefficients. As in \eqref{E:blipblip}, let $\mu$ be the Beltrami differential associated to $\rv_{-n}$ given by Lemma \ref{lem:extend_qc}. By definition, we have
\begin{equation}\label{E:last_equation_of_the_world}
a_m(t)=\frac{1}{2i\pi}\oint f_t(z)z^{-m-2}\d z,
\end{equation}
where we can pick the contour to be $r\S^1$ with $r<1$ such that $r\D$ contains the support of $f^*\mu$ in its interior. In the neighbourhood of this contour, the three small motions appearing in the expansion \eqref{E:P_equiv} are each holomorphic for $t$ small enough.
Then,  we have $\mu(\zeta)=n8^{2n}e^{-2nc}\bar{\zeta}^{n-1}\zeta\ind_{|\zeta|<\frac{e^c}{8}}$, and $(e^cf)^*\mu(\zeta)=n8^{2n}e^{-nc}\overline{f(\zeta)}^{n-1}f(\zeta)\frac{\overline{f'(\zeta)}}{f'(\zeta)}\ind_{|f(\zeta)|<\frac{1}{8}}$. This scaling relation and the explicit formulas for $u_\eta(\rv)$ and $\tilde u_\eta(\rv)$ show that the RHS of \eqref{E:P_equiv} depend on $c$ only through a global factor of $e^{-nc}$. Therefore, we will assume $c=0$ in the remainder of the proof. On the one hand, it is easy to see that
\begin{align*}
A_{n,m}:=-\frac{1}{2i\pi}\oint f(z)^{-n+1}z^{-m-2}\d z
\end{align*}
is a polynomial in $(a_m)_{m\geq1}$. From the explicit formula of $\mu$, we have
\begin{equation}\label{eq:expand_w}
\begin{aligned}
B_{n,m}&:=
\oint w(z)f'(z)z^{-m-2}\d z
=\oint\int_\D\frac{z(z-1)f^*\mu(\zeta)}{\zeta(\zeta-1)(\zeta-z)}|\d\zeta|^2\frac{f'(z)\d z}{z^{m+2}}\\
&=n8^{2n}\oint\int_{f^{-1}(\frac{1}{8}\D)}\sum_{k,l=0}^\infty\zeta^{k+l-1}z^{-k-2-m}(z-1)f'(z)\overline{f(\zeta)}^{n-1}f(\zeta)\frac{\overline{f'(\zeta)}}{f'(\zeta)}|\d\zeta|^2\d z.
\end{aligned}
\end{equation}
By a change of variables and Stokes' formula, we have for all $k,l\in\N$
\begin{equation}\label{eq:bulk}
\begin{aligned}
\int_{f^{-1}(\frac{1}{8}\D)}&\overline{f(\zeta)}^{n-1}f(\zeta)\frac{\overline{f'(\zeta)}}{f'(\zeta)}\zeta^{k+l-1}|\d\zeta|^2
=\int_{\frac{1}{8}\D}(f^{-1})'(\zeta)^2f^{-1}(\zeta)^{k+l-1}\bar{\zeta}^{n-1}\zeta|\d\zeta|^2\\
&=\frac{1}{2in}\oint_{\frac{1}{8}\S^1}(f^{-1})'(\zeta)^2f^{-1}(\zeta)^{k+l-1}\bar{\zeta}^n\zeta\d\zeta
=\frac{8^{-2n}}{2in}\oint_{\frac{1}{8}\S^1}(f^{-1})'(\zeta)^2f^{-1}(\zeta)^{k+l-1}\zeta^{-n+1}\d\zeta.
\end{aligned}
\end{equation}
This last expression shows that the integral is a polynomial in $(a_m)_{m\geq1}$ (with no dependence on $(\bar{a}_m)_{m\geq1}$), and it also vanishes for $k+l$ large enough (depending only on $n$). It follows that the series in \eqref{eq:expand_w} converges absolutely, so that we can exchange the sum and in the integral to get
\begin{align*}
B_{n,m}
&=\frac{1}{2i}\sum_{k,l=0}^\infty\oint f'(z)(z-1)z^{-k-m-2}\d z\oint_{\frac{1}{8}\S^1}(f^{-1})'(\zeta)^2f^{-1}(\zeta)^{k+l-1}\zeta^{-n+1}\d\zeta.
\end{align*}
For each fixed $k$, $\oint f'(z)(z-1)z^{-k-m-1}\d z$ is a polynomial in the coefficients of $f$. Keeping in mind that the sum evaluates to zero for $k,l$ large enough, we get that $B_{n,m}\in\C[(a_m)_{m\geq1}]$.

Finally, we study
\begin{align*}
C_{n,m}
&:=\oint\tilde{w}(z)f'(z)z^{-m-2}\d z
=\oint\int_\D\frac{z(z-1)\overline{f^*\mu(\zeta)}}{\bar{\zeta}(\bar{\zeta}-1)(1-\bar{\zeta}z)}|\d\zeta|^2f'(z)z^{-m-2}\d z\\
&=-n\oint\int_{f^{-1}(\frac{1}{8}\D)}\sum_{k,l=0}^\infty\bar{\zeta}^{k+l-1}z^{k-m-1}(z-1)f'(z)f(\zeta)^{n-1}\overline{f(\zeta)}\frac{f'(\zeta)}{\overline{f'(\zeta)}}|\d\zeta|^2\d z.
\end{align*}
For each $k,l\in\N$, the integral over $\frac{e^c}{8}\D$ is the complex conjugate of \eqref{eq:bulk}, so it is a polynomial in $(\bar{a}_m)_{m\geq1)}$ and vanishes for $k+l$ large enough (depending only on $n$). It follows that
\begin{align*}
C_{n,m}=\sum_{k,l=0}^\infty\oint f'(z)(z-1)z^{k-m-1}\d z\int_{f^{-1}(\frac{1}{8}\D)}f(\zeta)^{n-1}\overline{f(\zeta)}\frac{f'(\zeta)}{\overline{f'(\zeta)}}\bar{\zeta}^{k+l-1}|\d\zeta|^2.
\end{align*}
Each term in this (finite) sum is a polynomial in $(a_m,\bar{a}_m)$, so indeed $C_{n,m}$ is a polynomial. Contrary to $A_{n,m}$ and $B_{n,m}$, the polynomial $C_{n,m}$ is in both $(a_m)$ and $(\bar{a}_m)_{m\geq1}$.

Putting things together, \eqref{E:P_equiv} implies that the operators $\cL_{-n}$ and $\bar{\cL}_{-n}$ act on $\cC$ as
\begin{equation}\label{eq:witt_general_formula}
\begin{aligned}
\cL_{-n}=\varpi(\rv_{-n})\del_c+e^{-nc}\sum_{m=1}^\infty A_{n,m}\del_{a_m}-B_{n,m}\del_{a_m}-\overline{C}_{n,m}\del_{\bar{a}_m};\\
\bar{\cL}_{-n}=\overline{\varpi(\rv_{-n})}\del_c+e^{-nc}\sum_{m=1}^\infty\overline{A}_{n,m}\del_{\bar{a}_m}-\overline{B}_{n,m}\del_{\bar{a}_m}-C_{n,m}\del_{a_m}.
\end{aligned}
\end{equation}
As in the proof of the case dealing with $\rv_n, n \ge -1$, it follows from this analysis that $\cL_{-n}$, $\bar{\cL}_{-n}$ are operators on $\cC$ which preserve $\cC$, and they are first order differential operators with polynomial coefficients in $(a_m,\bar{a}_m)$ and a global prefactor $e^{nc}$. These indeed define endomorphisms of $\cC$ and $\cC_\mathrm{comp}$.

\noindent \textbf{Commutation relations.}
Now that we have constructed the operators $\cL_\rv$ and $\bar{\cL}_\rv$, $\rv \in \C(z)\partial_z$, it remains to prove the commutation relations. These hold by construction since the operators $\cL_\rv$ are Lie derivatives in directions of fundamental vector fields, but we include a proof ``by hand" since we work in an infinite dimensional setting.

Let $\rv,\rw\in\C(z)\del_z$, and let $\eta\in\cJ_{0,\infty}$. Let $\phi_t(z)=z+tv(z)+o(t)$, $\psi_s(z)=z+sw(z)+o(s)$ be the corresponding flows, defined for small complex $t,s$ in the region $\ann_\eta$ \eqref{eq:annulus_eta}. Using the differential equation $\del_t\phi_t=v\circ\phi_t$, it is immediate that the flow $t\mapsto\phi_t$ is smooth with respect to the local uniform topology in $\ann_\eta$ (in particular, $\del_t^2\phi_t=v'\circ\phi_tv\circ\phi_t$).

Let $F\in\cC$. The map $(t,s)\mapsto F(\psi_s\circ\phi_t(\eta))$ is of class $\cC^2$ in a neighbourhood of $0$ in $\C^2$, and $s\mapsto\cL_\rv F(\psi_s(\eta))$ is differentiable in direction $\rw$ for $s$ in a neighbourhood of 0. Hence, the mixed partial derivatives evaluate to
\begin{equation}\label{eq:comm}
    \frac{\del^2}{\del t\del s}_{|t,s=0}\left(F(\psi_s\circ\phi_t(\eta))-F(\phi_t\circ\psi_s(\eta))\right)=\frac{\del}{\del s}_{|s=0} \cL_\rv F(\psi_s(\eta))-\frac{\del}{\del t}_{|t=0}\cL_\rw F(\phi_t(\eta))=[\cL_\rv,\cL_\rw]F(\eta).
\end{equation}
The other mixed partial derivatives in $\del_{\bar t}\del_{\bar s}$, $\del_t\del_{\bar s}$, $\del_{\bar t}\del_s$ give respectively $[\bar\cL_{\rv},\bar\cL_{\rw}]F$, $[\cL_\rv,\bar\cL_\rw]F$, $[\bar\cL_\rv,\cL_\rw]F$. On the other hand, the map $(t,s)\mapsto\psi_s\circ\phi_t$ is smooth for the local uniform topology on $\ann_\eta$ (by the same argument above). We have $\frac{\del^2\psi_s\circ\phi_t}{\del t\del s}_{|t,s=0}=w'v$, and the other partial derivatives involving $\del_{\bar t}$ or $\del_{\bar s}$ vanish. It follows that
\begin{align*}
\frac{\del^2}{\del t\del s}_{|t,s=0}\left(F(\psi_s\circ\phi_t(\eta))-F(\phi_t\circ\psi_s\cdot \eta)\right)
=(\cL_{vw'\del_z}-\cL_{wv'\del_z})F(\eta)=\cL_{[\rv,\rw]}F(\eta).
\end{align*}
In the expansion of $F(\psi_s\circ\phi_t(\eta))$, the order $ts$ also contains a term in $\cL_{vw\del_z}$, but this term disappears in the above display by antisymmetry. Finally, the term in $\del_{\bar t}\del_{\bar s}$ gives $\bar\cL_{[\rv,\rw]}F$, while the cross terms $\del_t\del_{\bar s}$ and $\del_{\bar t}\del_s$ vanish. Putting this together with \eqref{eq:comm} gives:
\[
[\cL_\rv,\cL_\rw]=\cL_{[\rv,\rw]};
\qquad
[\bar{\cL}_\rv,\bar{\cL}_\rw]=\bar{\cL}_{[\rv,\rw]};
\qquad
[\cL_\rv,\bar{\cL}_\rw]=[\bar{\cL}_\rv,\cL_\rw]=0.
\]
This finishes the proof of Proposition \ref{P:construction_cL}.
\end{proof}

\subsubsection{Computation of \texorpdfstring{$\cL_0$}{L0}, \texorpdfstring{$\bar\cL_0$}{bar L0}, \texorpdfstring{$\cL_{-1}$}{L-1} and \texorpdfstring{$\bar\cL_{-1}$}{bar L-1}}\label{SSS:L0L1}

\begin{proof}[Proof of Proposition \ref{prop:generators_mobius}]
This proof could be directly obtained from Proposition \ref{P:equivalence_derivative}, but for pedagogical reasons, we will derive an independent and elementary proof.
We will compute successively $\bH_{\rv_0}$, $\bH_{i\rv_0}$, $\bH_{\rv_{-1}}$, and $\bH_{i\rv_{-1}}$, the existence of which follows from Proposition \ref{P:construction_cL} or Remark \ref{rem:mobius} (or simply from the arguments below), which are the generators of a finite dimensional group of operators. The explicit expressions given below have to be understood as acting on $\cC$ where the infinite sums reduce to finite sums.
In this proof, the parameter $t$ will be real.

\emph{Case $\bH_{\rv_0}$}.
The flow generated by $\rv_0=-z\del_z$ is $z\mapsto e^{-t}z$. In the coordinates $(c,f)\in\R\times\cE$, the motion of the curve is $(c_t,f_t)=(c-t,f)$. Hence, 
\[\bH_{\rv_0}=i\del_c.\]

\emph{Case $\bH_{i\rv_0}$}.
The flow generated by $i\rv_0$ is $z\mapsto e^{-it}z$, creating the motion $(c_t,f_t)=(c,e^{-it}f((e^{it}\cdot))$ in the coordinate $(c,f)\in\R\times\cE$. For the Taylor coefficients of $f$ as in \eqref{eq:coordinates}, we get $a_m(t)=e^{imt}a_m$, so that 
\[\bH_{i\rv_0}=\sum_{m=1}^\infty(ma_m\del_{a_m}-m\bar{a}_m\del_{\bar{a}_m}).\]

The formulas for $\cL_0=\frac{i}{2}(\bH_{\rv_0}-i\bH_{i\rv_0})$ and $\bar{\cL}_0=\frac{i}{2}(\bH_{\rv_0}+i\bH_{i\rv_0})$ follow.

\emph{Case $\bH_{\rv_{-1}}$}.
The flow generated by $-\del_z$ is $z\mapsto z-t$. Let $(c_t,f_t)\in\R\times\cE$ be uniformising the interior of $-t+\eta(\S^1)$. Since $e^cf-t$ is another uniformising map, we have $e^{c_t}f_t=e^cf\circ\phi_t-t$ for some M\"obius transformation $\phi_t$ of the unit disc. The generator of this transformation is of the form
\[\phi_t(z)=z+t(\alpha-i\theta z-\bar{\alpha}z^2)+o(t),\]
for some $\alpha\in\C$ and $\theta\in\R$. We also have $c_t=c+tx+o(t)$ for some $x\in\R$. We have $0=e^{c_t}f_t(0)=e^cf(\phi_t(0))-t=t(\alpha e^c-1)+o(t)$, so that $\alpha=e^{-c}$. Moreover,
\begin{align*}
1=f_t'(0)=e^{c-c_t}\phi_t'(0)f'(\phi_t(0))
&=(1-tx)(1-i\theta t)(1+\alpha tf''(0))+o(t)\\
&=1+t(e^{-c}f''(0)-(x+i\theta))+o(t),
\end{align*}
so that 
\begin{equation}
    \label{E:alpha_x_theta}
\alpha=e^{-c};\qquad x=e^{-c}\Re(f''(0));\qquad\theta=e^{-c}\Im(f''(0)).
\end{equation}
Setting $a_0:=1$, we have
\begin{align*}
f_t(z)
&=-te^{-c_t}+e^{c-c_t}\sum_{m=0}^\infty a_m\phi_t(z)^{m+1}\\
&=-te^{-c_t}+(1-tx)\sum_{m=0}^\infty a_m\left(z^{m+1}+t(m+1)(\alpha-i\theta z-\bar{\alpha}z^2)z^m\right)+o(t).
\end{align*}
Rearranging according to the power of $z$ and then using \eqref{E:alpha_x_theta}, we get that
\begin{align*}
f_t(z) & = f(z) + t(\alpha-e^{-c}) + t(-x + 2 a_1 \alpha - i\theta)z \\
& + t \sum_{m=1}^\infty ((m+2)\alpha a_{m+1}-(x+i\theta)a_m-im\theta a_m-m\bar{\alpha}a_{m-1})z^{m+1}+o(t) \\
& = f(z) + t e^{-c} \sum_{m=1}^\infty ((m+2)a_{m+1} - 2a_1a_m -2im \Im(a_1) a_m - ma_{m-1} )z^{m+1} + o(t).
\end{align*}
It follows that
\begin{align*}
i\bH_{\rv_{-1}}=2e^{-c}\Re(a_1)\del_c&+e^{-c}\sum_{m=1}^\infty\left((m+2)a_{m+1}-2a_1a_m-2im\Im(a_1)a_m-ma_{m-1}\right)\del_{a_m}\\
&+e^{-c}\sum_{m=1}^\infty((m+2)\bar{a}_{m+1}-2\bar{a}_1\bar{a}_m+2im\Im(a_1)\bar{a}_m-m\bar{a}_{m-1})\del_{\bar{a}_m}.
\end{align*}

\emph{Case $\bH_{i\rv_{-1}}$}.
We apply the same strategy for the flow $z\mapsto z-it$. Using the same parametrisation, we have this time
\[\alpha=ie^{-c};\qquad x=-e^{-c}\Im(f''(0));\qquad\theta=e^{-c}\Re(f''(0)).\]
Replacing with these values gives
\begin{align*}
i\bH_{i\rv_{-1}}=-2e^{-c}\Im(a_1)\del_c&+ie^{-c}\sum_{m=1}^\infty((m+2)a_{m+1}-2a_1a_m-2m\Re(a_1)a_m+ma_{m-1})\del_{a_m}\\
&-ie^{-c}\sum_{m=1}^\infty((m+2)\bar{a}_{m+1}-2\bar{a}_1\bar{a}_m-2m\Re(a_1)\bar{a}_m+m\bar{a}_{m-1})\del_{\bar{a}_m}.
\end{align*}
The formulas for $\cL_{-1}=\frac{i}{2}(\bH_{\rv_{-1}}-i\bH_{i\rv_{-1}})$ and $\bar{\cL}_{-1}=\frac{i}{2}(\bH_{\rv_1}+i\bH_{i\rv_1})$ follow.
\end{proof}

\subsubsection{The coordinate function \texorpdfstring{$c : \cJ_{0,\infty} \to \R$}{c}}\label{SSS:c}

\begin{proof}[Proof of Proposition \ref{prop:diff_c}]
Let $\rv=v(z)\partial_z \in \C(z) \partial_z$ and $\eta \in \cJ_{0,\infty}$.
By Proposition \ref{P:construction_cL}, any function $F \in C^\infty_c(\R)$ of the variable $c$ only, which agrees with $c \mapsto c$ in a neighbourhood of $c_\eta$, is differentiable in the direction $\rv$. So the differentiability of the function $c: \cJ_{0,\infty} \to \R$ in direction $\rv$ at $\eta$ has already been established in Proposition \ref{P:construction_cL}.
Moreover, by Proposition \ref{P:equivalence_derivative},
\[
\cL_\rv (e^c)(\eta) = u_\eta(\rv)'(0).
\]
From the expression \eqref{E:def_urv} of $u_\eta(\rv)$, we see that $u_\eta(\rv)'(0)$ agrees with $e^{c_\eta}\varpi_\eta(\rv)$ \eqref{eq:def_varpi}. After a change of variable we obtain that $\cL_\rv c(\eta) = \varpi_\eta(\rv)$.
\end{proof}

\subsection{Virasoro representation}\label{subsec:virasoro}
In this section, we modify our representation of the Witt algebra to a representation of the Virasoro algebra. The same modification is considered in \cite[Section 4.1]{GQW24}.
Recall the definition \eqref{eq:def_vartheta} of $\vartheta$ and $\tilde \vartheta$ and that, for all $\rv \in \C(z)\partial_z$, $\bL_\rv$ and $\bar \bL_\rv$ are defined from $\cL_\rv$ and $\bar \cL_\rv$ by:
\[\bL_\rv:=\cL_\rv-\frac{c_\rM}{12}\vartheta(\rv),\qquad\bar{\bL}_\rv:=\bar{\cL}_\rv-\frac{c_\rM}{12}\overline{\vartheta(\rv)}.\]

For ease of future record, we state the following lemma which, in conjunction with Proposition~\ref{P:construction_cL}, implies that $\bL_\rv$ and $\bar \bL_\rv$ preserve the subspace $\cC$.

\begin{lemma}\label{L:preserve_cC}
For all $\rv \in \C(z) \partial_z$, the multiplication by $\vartheta(\rv)$ preserves $\cC$ \eqref{E:cC} and the multiplication by $\tilde{\vartheta}(\rv)$ sends $\cC_\mathrm{comp}$ to $L^2(\nu)$ \eqref{E:cC_compact}.
\end{lemma}
\begin{proof}
We will consider $\tilde{\vartheta}(\rv)$. 
It is enough to treat the case $\rv = \rv_n$ for some $n \in \Z$.
We have $\tilde{\vartheta}(\rv_n)=0$ for $n\leq-2$, so we assume $n\geq-1$.
Recalling the definition \eqref{E:tilde_c_0} of $\tilde c_0$, we have by scaling that $\tilde{\vartheta}_\eta(\rv_n)=e^{-n(\tilde{c}_0-c)}\tilde{P}_n(g)$ for some polynomial $\tilde{P}_n$ in the coefficients of the expansion of $g$ at $\infty$. These coefficients are uniformly bounded by de Branges' theorem, so $\tilde{P}_n$ is uniformly bounded. Moreover, by the Koebe quarter theorem, for each $F \in \cC_\mathrm{comp}$, $\tilde c_0$ and $c$ are bounded on the support of $F$. This proves that $e^{-n(\tilde{c}_0-c)} \tilde P_n(g)$ is bounded on the support of $F$, so its product with $F$ belongs to $L^2(\nu)$.
\end{proof}

\begin{proposition}\label{P:commutation_bL}
As endomorphisms of $\cC$, the following commutation relations are satisfied:
\[[\bL_\rv,\bL_\rw]=\bL_{[\rv,\rw]}+\frac{c_\rM}{12}\omega(\rv,\rw);\qquad[\bar{\bL}_\rv,\bar{\bL}_\rw]=\bar{\bL}_{[\rv,\rw]}+\frac{c_\rM}{12}\overline{\omega(\rv,\rw)};\qquad[\bL_\rv,\bar{\bL}_\rw]=0,\]
for all $\rv,\rw\in\C(z)\del_z$,
where $\omega(\rv,\rw)$ is the Virasoro cocycle \eqref{eq:def_omega}.
\end{proposition}

\begin{proof}
To prove the commutation relations of $(\bL_\rv)_{\rv \in \C(z)\partial_z}$,
it is enough to treat the three cases below by linearity. The first one will follow directly from Proposition \ref{P:construction_cL}. The second one will require an extra computation related to the Schwarzian derivative. The third one will follow from the first case and a duality argument based on the integration by parts formula of Corollary \ref{cor:ibp}. 

\noindent \textbf{Case 1.} $\rv, \rw \in \C[z]\partial_z$. In this case $\bL_\rv = \cL_\rv$ and $\bL_\rw = \cL_\rw$ and by Proposition \ref{P:construction_cL} we deduce that $[\bL_\rv,\bL_\rw]=[\cL_\rv,\cL_\rw]=\cL_{[\rv,\rw]}$. Since $[v,w]$ also belongs to $\C[z]\partial_z$, $\cL_{[\rv,\rw]}$ agrees with $\bL_{[\rv,\rw]}$ and we conclude that
$[\bL_\rv,\bL_\rw]=\bL_{[\rv,\rw]}$.

\noindent \textbf{Case 2.}
$\rv=v\del_z\in z^2\C[z]\del_z$ and $\rw=w\del_z\in z\C[z^{-1}]\del_z$. Recall the definition of $\vartheta$ in \eqref{eq:def_vartheta}: in this case, $\vartheta(\rv)=0$ and we only need to deal with $\vartheta(\rw)$.
The latter can be seen as either an element of $\cC$ or as an endomorphism of $\cC$ acting by multiplication. To clarify, we will denote these two objects in this proof by $\vartheta(\rw) \in \cC$ and $\vartheta(\rw) \id_{\cC}$ respectively. We have
\[
\cL_\rv \circ (\vartheta(\rw) \id_{\cC}) = (\cL_\rv \vartheta(\rw)) \id_{\cC} + \vartheta(\rw) \cL_\rv,
\]
since, by the Leibniz rule, for all $F \in \cC$,
\[
(\cL_\rv \circ (\vartheta(\rw) \id_{\cC})) F = \cL_\rv(\vartheta(\rw) F) = (\cL_\rv \vartheta(\rw)) F + \vartheta(\rw) \cL_\rv F.
\]
This implies that
\[
[\bL_\rv,\bL_\rw]
=\left[\cL_\rv,\cL_\rw-\frac{c_\rM}{12}\vartheta(\rw) \id_{\cC}\right]
=[\cL_\rv,\cL_\rw]-\frac{c_\rM}{12}(\cL_\rv\vartheta(\rw)) \id_{\cC}.
\]
By Proposition \ref{P:construction_cL}, we already know that $[\cL_\rv,\cL_\rw] = \cL_{[\rv,\rw]}$. We are going to show that
\begin{equation}\label{E:cL_vartheta}
\cL_\rv\vartheta(\rw) = - \omega(\rv,\rw) + \vartheta([\rv,\rw]).
\end{equation}
We will then deduce that
\[
[\bL_\rv,\bL_\rw]
= \cL_{[\rv,\rw]} - \frac{c_\rM}{12} \vartheta([\rv,\rw]) \id_{\cC} + \frac{c_\rM}{12} \omega(\rv,\rw) \id_{\cC}
= \bL_{[\rv,\rw]} + \frac{c_\rM}{12} \omega(\rv,\rw) \id_{\cC}
\]
as desired.

It remains to prove \eqref{E:cL_vartheta}.
To this end, let $\eta \in \cJ_{0,\infty}$ parametrised by $(c,f) \in \R \times \cE$ and
let $\phi_t(z)=z+tv(z)+o(t)$ be the flow generated by $\rv$, defined for small $t$ in a neighbourhood $\overline{\rmint(\eta)}$. In the coordinates $(c,f)\in\R\times\cE$, the motion $\phi_t(\eta)$ is represented by a normalised conformal map $e^{c_t}f_t$. Since $v'(0)=0$, $\phi_t(e^cf)$ is normalised at first order and uniformises the interior of $\phi_t(\eta)$. Hence, $e^{c_t}f_t=\phi_t(e^cf)+o(t)$ where $o(t)$ is a function whose derivatives go to zero as $t \to 0$, uniformly on compact subsets of $\D$. By the chain rule for the Schwarzian, we have
\begin{align}\notag
\cS((\phi_t\circ(e^cf))^{-1})
=\cS((e^cf)^{-1}\circ\phi_t^{-1})
&=\cS(\phi_t^{-1})+((\phi_t^{-1})')^2\cS((e^cf)^{-1})\circ\phi_t^{-1}\\
&=\cS((e^cf)^{-1})-t\left(v'''+2v'\cS((e^cf)^{-1})+v\cS((e^cf)^{-1})'\right)+o(t).\label{eq:diff_schwarzian}
\end{align}
By integration by parts, $\cL_\rv\vartheta_\eta(\rw)$ is thus equal to
\begin{align*}
&\underset{t\to0}\lim\,\frac{1}{2i\pi t}\oint\left(\cS((\phi_t\circ(e^cf))^{-1}(z)-\cS(e^cf)^{-1}(z)\right)w(z)\d z\\
&=-\frac{1}{2i\pi}\oint v'''(z)w(z)\d z-\frac{1}{2i\pi}\oint\left(2v'(z)w(z)\cS((e^cf)^{-1})(z)+v(z)w(z)\cS((e^cf)^{-1})'(z)\right)\d z\\
&=-\omega(\rv,\rw)+\frac{1}{2i\pi}\oint\left(v(z)w'(z)-v'(z)w(z)\right)\cS((e^cf)^{-1})(z)\d z
=-\omega(\rv,\rw)+\vartheta_\eta([\rv,\rw])
\end{align*}
which concludes the proof of \eqref{E:cL_vartheta} and thus the proof of Case 2.

\noindent \textbf{Case 3.} $\rv, \rw \in z\C[z^{-1}]\del_z$.
Recall from \eqref{eq:star} that we define an antilinear involution of $\C(z)\del_z$ by
$
\rv^*=z^2\overline{v(1/\bar z)}\del_z.
$
On the generators, we have $\rv_n^*=\rv_{-n}$. Moreover, for all $n, m \in \Z$,
\begin{align*}
[\rv_n,\rv_m]^*=(n-m)\rv_{n+m}^*=(n-m)\rv_{-n-m}=-[\rv_{-n},\rv_{-m}]=-[\rv_n^*,\rv_m^*],
\end{align*}
so that for all $\rv, \rw \in \C(z) \partial_z$,
\begin{equation}\label{lem:star}
[\rv,\rw]^*=-[\rv^*,\rw^*].
\end{equation}
Let us fix $\rv, \rw \in z\C[z^{-1}]\del_z$.
The integration by parts formula of Corollary \ref{cor:ibp} (which is independent of the commutation relations and was also stated in Theorem \ref{T:intro_shapo} above) gives:
\begin{equation}\label{eq:negative_commutations}
[\bL_{\rv},\bL_{\rw}]
=[\Theta\circ\bL_{\rv^*}^*\circ\Theta,\Theta\circ\bL_{\rw^*}^*\circ\Theta]
=\Theta\circ[\bL_{\rv^*}^*,\bL_{\rw^*}^*]\circ\Theta
=-\Theta\circ[\bL_{\rv^*},\bL_{\rw^*}]^*\circ\Theta.
\end{equation}
Now, since $\rv^*, \rw^* \in z \C[z]\partial_z$, we know from Case 1 that
$[\bL_{\rv^*},\bL_{\rw^*}] = \bL_{[\rv^*,\rw^*]}$ ($\omega(\rv^*,\rw^*)$ vanishes in this case). Using once more the integration by parts formula, we deduce that
\[
[\bL_{\rv},\bL_{\rw}] = - \Theta \circ \bL_{[\rv^*,\rw^*]}^* \circ \Theta
= -\bL_{[\rv^*,\rw^*]^*}
= \bL_{[\rv,\rw]}.
\]
In the last equality we used that $[\rv^*,\rw^*]^* = -[\rv,\rw]$; see \eqref{lem:star}.
This is the desired result since $\omega(\rv,\rw)=0$ in this case.

  \medskip
  
  This proves the commutation relations for the representation $(\bL_\rv)_{\rv\in\C(z)\del_z}$. The proof for the representation $(\bar{\bL}_\rv)_{\rv\in\C(z)\del_z}$ is identical. The last step is to prove that the two representations commute. If $\rv,\rw\in\C[z]\del_z$, we have $[\bL_\rv,\bar{\bL}_\rw]=[\cL_\rv,\bar{\cL}_\rw]$ which vanishes by Proposition \ref{P:construction_cL}. In this case, we also have
\begin{equation}\label{eq:negative_commutations_bis}
[\bL_{\rv^*},\bar{\bL}_{\rw^*}]=[\Theta\circ\bL_\rv^*\circ\Theta,\Theta\circ\bar{\bL}_\rw^*\circ\Theta]=\Theta\circ[\bL_\rv^*,\bar{\bL}_\rw^*]\circ\Theta=-\Theta\circ[\bL_\rv,\bar{\bL}_\rw]^*\circ\Theta=0.
\end{equation}  
   Finally, suppose $\rv\in z^2\C[z]\del_z$ and $\rw\in z\C[z^{-1}]\del_z$. Then, $[\bar{\bL}_\rv,\bL_\rw]=[\bar{\cL}_\rv,\cL_\rw-\frac{c_\rM}{12}\vartheta(\rw)\id_{\cC}]=-\frac{c_\rM}{12}\bar{\cL}_\rv\vartheta(\rw)\id_{\cC}$, which vanishes by analyticity of the Schwarzian on $\cE$ (there is no $\bar{t}$ term in \eqref{eq:diff_schwarzian}). Similarly, we get that $[\bL_\rv,\bar{\bL}_\rw]=0$ for $\rv\in z\C[z^{-1}]\del_z$ and $\rw\in z^2\C[z]\del_z$, so we are done. 
\end{proof}

\begin{remark}
In this proof, we used the integration by parts formula of Corollary \ref{cor:ibp} to reduce Case 3 to Case 1. Without this result, one would need to compute $\cL_\rv\vartheta(\rw)$ with $\rv,\rw\in\C[z^{-1}]\del_z$, which is a much more involved geometric computation than Case 2.
\end{remark}

\section{Highest-weight representations}\label{sec:hw_modules}

In the previous section, we defined two representation of the Virasoro algebra on $L^2(\nu)$. In this section, we do a Laplace transform in the ($\log$) conformal radius, which will allow us to define a family of representations $(\bL_n^\lambda,\bar{\bL}_n^\lambda)$ acting on $L^2(\nu^\#)$, indexed by $\lambda\in\C$. The relevance of this family is that it allows us to construct highest-weight representations of weight $\lambda$. The main result of the section is Theorem \ref{thm:module_structure} giving a complete classification of these representations.

	\subsection{Construction}\label{subsec:level_operators}
 The main idea is that the eigenstates of $\del_c$ will label highest-weight states for the highest-weight representations. We could view the diagonalisation of this operator as an analytic question (consisting of a Fourier/Laplace transform), but since we have in mind algebraic statements (commutation relations etc.) we prefer to take it as such. This will result in a family of Virasoro representations $(\bL_n^\lambda,\bar\bL_n^\lambda)_{n\in\Z}$ indexed by $\lambda\in\C$, which are initially defined as endomorphisms of $\C[(a_m,\bar a_m)_{m\geq1}]$, but can also be viewed as densely defined operators on $L^2(\nu^\#)$. This will be used in Section \ref{sec:spectral} to write a Plancherel formula.

Recall that the left and right levels of a polynomial in $\C[(a_m,\bar{a}_m)_{m\geq1}]$ have been defined below \eqref{E:monomial}.
Let $\bN,\overline{\bN}$ be the endomorphisms of $\C[(a_m,\bar{a}_m)_{m\geq1}]$ characterised by
\[\bN(\ba^{\bk}\bar{\ba}^{\tilde{\bk}})=|\bk|\ba^{\bk}\bar{\ba}^{\tilde{\bk}},\qquad\overline{\bN}(\ba^{\bk}\bar{\ba}^{\tilde{\bk}})=|\tilde{\bk}|\ba^{\bk}\bar{\ba}^{\tilde{\bk}},\qquad\forall\bk,\tilde{\bk}\in\cT.\]
We call these operators the \emph{degree operators}. They commute and are diagonal with spectrum $\N$. For each $N\in\N$, the $N$-eigenspace of $\bN$ (resp. $\overline{\bN}$) is the space of polynomials with left (resp. right) degree equal to $N$. Observe that $\bN=\sum_{m=1}^\infty ma_m\del_{a_m}$ and $\overline{\bN}=\sum_{m=1}^\infty m\bar{a}_m\del_{\bar{a}_m}$, so that by Proposition \ref{prop:generators_mobius},
\[\cL_0=-\frac{\del_c}{2}+\frac{1}{2}(\bN-\overline{\bN});\qquad\bar{\cL}_0=-\frac{\del_c}{2}-\frac{1}{2}(\bN-\overline{\bN}).\]

We have seen in the previous section that $\bL_n$ is a first order differential operator whose coefficients are polynomials in $a_m,\bar{a}_m$, and $e^c$. The first step is to refine this result by giving the degrees of these polynomials. The next lemma uses the standard notation $x_-=\max\{-x,0\}$.
\begin{lemma}\label{lem:level}
There exist polynomials $P_{n,m}, Q_{n,m}, \tilde P_{n,m}, \tilde Q_{n,m} \in \C[(a_k,\bar a_k)_{k\geq1}]$, $n \in \Z, m \ge 0$ such that
for all $n\in\Z$, we have
\begin{align*}
&\bL_n=e^{nc}P_{n,0}\del_c+e^{nc}Q_{n,0}+e^{nc}\sum_{m=1}^\infty(P_{n,m}\del_{a_m}+Q_{n,m}\del_{\bar{a}_m}),\\
&\bar{\bL}_n=e^{nc}\tilde{P}_{n,0}\del_c+e^{nc}\tilde{Q}_{n,0}+e^{nc}\sum_{m=1}^\infty(\tilde{P}_{n,m}\del_{\bar{a}_m}+\tilde{Q}_{n,m}\del_{a_m}),
\end{align*}
where, for all $m\in\N$:
\begin{itemize}[leftmargin=*, label=\raisebox{0.25ex}{\tiny$\bullet$}]
    \item $P_{n,m}$ (resp. $\tilde{P}_{n,m}$) has degree $((n-m)_-,0)$ (resp. $(0,(n-m)_-)$).
    \item If $n\leq0$, $Q_{n,m}$ (resp. $\tilde{Q}_{n,m}$) is a linear combination of monomials of degree $(k,\tilde{k})$ with $k-\tilde{k}=(n+m)_-$ (resp. $\tilde{k}-k=(n+m)_-$). If $n>0$, $Q_{n,m}=\tilde{Q}_{n,m}=0$.
\end{itemize}
Moreover, recalling the definitions \eqref{eq:def_vartheta} of $\vartheta$ and \eqref{eq:def_varpi} of $\varpi$, we have for all $n \in \Z$,
\[e^{nc}P_{n,0}=\varpi(\rv_n);\qquad e^{nc}Q_{n,0}=-\frac{c_\rM}{12}\vartheta(\rv_n);\qquad e^{nc}\tilde{P}_{n,0}=\overline{\varpi(\rv_n)};\qquad e^{nc}\tilde{Q}_{n,0}=-\frac{c_\rM}{12}\overline{\vartheta(\rv_n)}.\]
\end{lemma}
Note that the polynomial $Q_{n,0}$ is proportional to the Neretin polynomial introduced after \eqref{eq:def_vartheta}.

\begin{remark}
At this stage, we are not saying anything about the total degree of $Q_{n,m}$ and $\tilde{Q}_{n,m}$. However, a consequence of the proof of Theorem \ref{thm:module_structure} is that this degree is bounded by $m-n$. For instance, the reader can check from Proposition \ref{prop:generators_mobius} that $Q_{-1,1}=|a_1|^2-1$, a non-homogeneous polynomial of total degree 2.
\end{remark}

\begin{proof}
The general form of $\cL_n$ and $\bar{\cL}_n$ has been given in \eqref{eq:witt_general_formula}, and gives the value of $P_{n,0},\tilde{P}_{n,0}$. The values of $Q_{n,0},\tilde{Q}_{n,0}$ also follow by definition of $\bL_n,\bar{\bL}_n$. It just remains to prove the statements on the degrees. We only treat the case of $\bL_n$, the case of $\bar{\bL}_n$ being identical. The degrees of $P_{n,0}$ and $Q_{n,0}$ are clear by definition of $\varpi$ and $\vartheta$. It was shown in the proof of Proposition \ref{P:construction_cL} that, for $n>0$, $Q_{n,m}=0$ and the right degree of $P_{n,m}$ vanishes (it is a polynomial in the coefficients $a_m$, with no $\bar a_m$). 

We turn to the left degree of $P_{n,m}$ for $m\geq1$. We have $P_{n,m}=e^{-nc}\bL_n(a_m)-Q_{n,0}a_m$. Applying $\bar{\bL}_0=-\frac{1}{2}(\del_c+\bN-\overline{\bN})$ on both sides and using $[\bar{\bL}_0,\bL_n]=0$, we find $\bN P_{n,m}=-(n-m)P_{n,m}$, i.e. $P_{n,m}$ has left degree $(n-m)_-$.

Finally, we study $Q_{n,m}$ for $n\leq0$ and $m\geq1$. We have $Q_{n,m}=e^{-nc}\bL_n(\bar{a}_m)-Q_{n,0}\bar{a}_m$. Applying again $\bar{\bL}_0$ to both sides, we find $(\bN-\overline{\bN})Q_{n,m}=-(n+m)Q_{n,m}$. Hence, $Q_{n,m}$ is a combination of monomials of degree $(k,\tilde{k})$ with $k-\tilde{k}=(n+m)_-$.
\end{proof}

Since $\del_c$ and the multiplication by $e^{nc}$ are continuous on $\cC_c^\infty(\R)$, they extend to bounded operators on $\cC_c^\infty(\R)'$ by duality. It is then clear from the expressions of the previous lemma that the operators $\bL_n,\bar{\bL}_n$ extend to endomorphisms of $\cC_c^\infty(\R)'\otimes\C[(a_m,\bar{a}_m)_{m\geq1}]$, and this extension is unique by density of $\cC_c^\infty(\R)$ in $\cC_c^\infty(\R)'$. The level operators can be exponentiated on $\cC_c^\infty(\R)'\otimes\C[(a_m,\bar{a}_m)_{m\geq1}]$: the operator $e^{-\bN c}$ acts by multiplication by $e^{-Nc}$ on the space of polynomials with left level $N$. This defines an endomorphism of $\cC_c^\infty(\R)'\otimes\C[(a_m,\bar{a}_m)_{m\geq1}]$.

Let $\lambda\in\C$. The previous lemma gives
\begin{equation}\label{E:highest_weight}
 \bL_n(e^{-2\lambda c})=\bar{\bL}_n(e^{-2\lambda c})=0 \quad \text{for all } n \ge1; \qquad \bL_0(e^{-2\lambda c}) =\bar{\bL}_0(e^{-2\lambda c})= \lambda e^{-2\lambda c},
\end{equation}
and similarly for the conjugate representation. By definition, $e^{-2\lambda c}$ is a highest-weight state of weight $\lambda$. We will use this to construct highest-weight representations, but first we want to get rid of the exponential prefactors in $c$ in order to view all the states in the space of polynomials. To this end, given $\bk,\tilde{\bk}\in\cT$, we define
\begin{equation}
    \label{E:def_Psi}
\Psi_{\lambda,\bk,\tilde{\bk}}:=e^{(2\lambda+|\bk|+|\tilde{\bk}|)c}\bL_{-\bk}\bar{\bL}_{-\tilde{\bk}}(e^{-2\lambda c}),
\end{equation}
and set $\Psi_{\lambda,\bk}:=\Psi_{\lambda,\bk,\emptyset}$, $\Psi_\lambda:=\Psi_{\lambda,\emptyset,\emptyset}=\ind$. By Lemma \ref{lem:level}, we have $\Psi_{\lambda,\bk}\in\C[(a_m)_{m\geq1}]$ and $\Psi_{\lambda,\bk,\tilde{\bk}}\in\C[(a_m,\bar{a}_m)_{m\geq1}]$. Next, define
\[\cV_\lambda:=\mathrm{span}\left\lbrace\Psi_{\lambda,\bk}\big|\bk\in\cT\right\rbrace;\qquad\cW_\lambda:=\mathrm{span}\left\lbrace\Psi_{\lambda,\bk,\tilde{\bk}}\Big|\,\bk,\tilde{\bk}\in\cT\right\rbrace.\]
We can define the diagonal endomorphisms $\bL_0^\lambda,\bar{\bL}_0^\lambda$ of $\cW_\lambda$ characterised by
\begin{equation}\label{E:L_0^lambda}
\bL_0^\lambda\Psi_{\lambda,\bk,\tilde{\bk}}=(\lambda+|\bk|)\Psi_{\lambda,\bk,\tilde{\bk}};\qquad\bar{\bL}_0^\lambda\Psi_{\lambda,\bk,\tilde{\bk}}=(\lambda+|\tilde{\bk}|)\Psi_{\lambda,\bk,\tilde{\bk}}.
\end{equation}
While the states $(\Psi_{\lambda,\bk,\tilde{\bk}})_{\bk,\tilde{\bk}\in\cT}$ may not be linearly independent, linear relations may only exist at the same $\lambda$-level by standard representation theory (see Appendix \ref{app:virasoro}), i.e. for $\Psi_{\lambda, \bk, \tilde \bk}$ and $\Psi_{\lambda, \bk', \tilde \bk'}$ with $(|\bk|,|\tilde \bk|) = (|\bk'|,|\tilde \bk'|)$; so these operators are indeed well defined.

Next, for all $n\in\Z$, we define the endomorphisms $\bL_n^\lambda,\bar{\bL}_n^\lambda$ of $\C[(a_m,\bar{a}_m)_{m\geq1}]$ by
\begin{equation}\label{eq:op_lambda}
\begin{aligned}
&\bL_n^\lambda:=-P_{n,0}(\bL_0^\lambda+\bar{\bL}_0^\lambda)+Q_{n,0}+\sum_{m=1}^\infty(P_{n,m}\del_{a_m}+Q_{n,m}\del_{\bar{a}_m})\\
&\bar{\bL}_n^\lambda:=-\tilde{P}_{n,0}(\bL_0^\lambda+\bar{\bL}_0^\lambda)+\tilde{Q}_{n,0}+\sum_{m=1}^\infty(\tilde{P}_{n,m}\del_{\bar{a}_m}+\tilde{Q}_{n,m}\del_{a_m}).
\end{aligned}
\end{equation}
Concretely, these operators are just obtained from $\bL_n,\bar{\bL}_n$ by removing the prefactor $e^{nc}$, and replacing $\del_c$ by the operator $-(\bL_0^\lambda+\bar{\bL}_0^\lambda)$. These operators exponentiate to a linear map
\begin{equation}
    \label{E:cP_lambda}
\cP_\lambda:=e^{-(\bL_0^\lambda+\bar{\bL}_0^\lambda)c}:\cW_\lambda\to\cC_c^\infty(\R)'\otimes\C[(a_m,\bar{a}_m)_{m\geq1}].
\end{equation}
We emphasise that, since $\cW_\lambda$ is made of finite linear combinations, the definition of the exponential above reduces to that of an exponential of a linear operator on a finite dimensional space.
The next proposition essentially shows that $\cP_\lambda$ conjugates the operators $\bL_n^\lambda$ and $\bL_n$.

\begin{proposition}\label{lem:commutations_lambda}
For all $n \in \Z$, $\bL_n^\lambda = \cP_\lambda^{-1} \circ \bL_n \circ \cP_\lambda$ and $\bar{\bL}_n^\lambda = \cP_\lambda^{-1} \circ \bar{\bL}_n \circ \cP_\lambda$. In particular, the following commutation relations of endomorphisms of $\C[(a_m,\bar{a}_m)_{m\geq1}]$ hold: for all $n,m \in \Z$,
\begin{align*}
&[\bL_n^\lambda,\bL_m^\lambda]=(n-m)\bL_{n+m}^\lambda+\frac{c_\rM}{12}(n^3-n)\delta_{n,-m};\\
&[\bar{\bL}_n^\lambda,\bar{\bL}_m^\lambda]=(n-m)\bar{\bL}_{n+m}^\lambda+\frac{c_\rM}{12}(n^3-n)\delta_{n,-m};
\qquad
[\bL_n^\lambda,\bar{\bL}_m^\lambda]=0.
\end{align*}
Moreover, for all $\bk, \tilde \bk \in \cT$, we have $\Psi_{\lambda,\bk,\tilde{\bk}}=\bL_{-\bk}^\lambda\bar{\bL}_{-\tilde{\bk}}^\lambda\ind$.
In particular, $\bL_0^\lambda\ind=\bar{\bL}_0^\lambda\ind=\lambda\ind$ and $\bL_n^\lambda\ind=\bar{\bL}_n^\lambda\ind=0$ for all $n\geq1$.
\end{proposition}

\begin{proof}
Let $n \in \Z$. We verify that 
\begin{equation}\label{eq:L_lambda}
\bL_n^\lambda=\cP_\lambda^{-1}\circ\bL_n\circ\cP_\lambda
\end{equation}
on the generators $\Psi_{\lambda,\bk,\tilde{\bk}}$ of $\cW_\lambda$.
Let $\bk,\tilde{\bk}\in\cT$. We have
\begin{equation}
    \label{E:pf_Psi}
    \hspace{50pt} \cP_\lambda\Psi_{\lambda,\bk,\tilde{\bk}}=e^{-(2\lambda+|\bk|+|\tilde{\bk}|)c}\Psi_{\lambda,\bk,\tilde{\bk}},
    \hspace{60pt} \text{so}
\end{equation}
\[e^{nc}P_{n,0}\del_c(\cP_\lambda\Psi_{\lambda,\bk,\tilde{\bk}})=-(2\lambda+|\bk|+|\tilde{\bk}|)e^{-(2\lambda+|\bk|+|\tilde{\bk}|-n)c}P_{n,0}\Psi_{\lambda,\bk,\tilde{\bk}}.\]
From this, we get that $\bL_n(\cP_\lambda\Psi_{\lambda,\bk,\tilde{\bk}})$ is equal to
\begin{align*}
&e^{-(2\lambda+|\bk|+|\tilde{\bk}|-n)c}\left(-(2\lambda+|\bk|+|\tilde{\bk}|)P_{n,0}+Q_{n,0}+\sum_{m=1}^\infty( P_{n,m}\del_{a_m}+Q_{n,m}\del_{\bar{a}_m})\right)\Psi_{\lambda,\bk,\tilde{\bk}}\\
&=e^{-(2\lambda+|\bk|+|\tilde{\bk}|-n)c}\bL_n^\lambda\Psi_{\lambda,\bk,\tilde{\bk}}=\cP_\lambda(\bL_n^\lambda\Psi_{\lambda,\bk,\tilde{\bk}}).
\end{align*}
The operator $\cP_\lambda$ is clearly injective, so it has an inverse $\cP_\lambda^{-1}$ defined on its image, and we have $\bL_n^\lambda\Psi_{\lambda,\bk,\tilde{\bk}}=\cP_\lambda^{-1}\circ\bL_n\circ\cP_\lambda\Psi_{\lambda,\bk,\tilde{\bk}}$ as desired.
From this, we deduce that the operators $(\bL_n^\lambda,\bar{\bL}_n^\lambda)_{n\in\Z}$ satisfy the same commutation relations as the $(\bL_n,\bar{\bL}_n)_{n\in\Z}$. To check that $\Psi_{\lambda,\bk,\tilde{\bk}}=\bL_{-\bk}^\lambda\bar{\bL}^\lambda_{-\tilde{\bk}}\ind$, recall first that, by \eqref{E:pf_Psi} applied to $\bk=\tilde{\bk}=\varnothing$, $\cP^\lambda \ind = e^{-2\lambda c}$. From the conjugation relations, we deduce that
$\bL_{-\bk}^\lambda\bar{\bL}^\lambda_{-\tilde{\bk}}\ind =\cP_\lambda^{-1} \bL_{-\bk}\bar{\bL}_{-\tilde{\bk}} e^{-2\lambda c}$. By definition \eqref{E:def_Psi} of $\Psi_{\lambda,\bk,\tilde \bk}$, we then have
$\bL_{-\bk}^\lambda\bar{\bL}^\lambda_{-\tilde{\bk}}\ind = \cP_\lambda^{-1} e^{-(2\lambda + |\bk| + |\tilde \bk|) c} \Psi_{\lambda,\bk,\tilde{\bk}}$ which agrees with $\Psi_{\lambda,\bk,\tilde \bk}$ by \eqref{E:pf_Psi}.
\end{proof}

		\subsection{Classification}		
In the previous section, we defined highest-weight representations $(\cV_\lambda,(\bL_n^\lambda)_{n\in\Z}$ and $(\cW_\lambda,(\bL_n^\lambda,\bar{\bL}_n^\lambda)_{n\in\Z})$	of weight $\lambda\in\C$. In this section, we characterise the structure of these representations, i.e. we prove Theorem \ref{thm:module_structure}. This classification shows that the modules $\cV_\lambda$ are all irreducible, i.e. there are no singular vectors at positive level. Interestingly, we can give two different proofs of this result: a purely algebraic proof and an analytic proof. The existence of an analytic proof relies on the fact that the Virasoro algebra is represented as differential operators on $\cJ_{0,\infty}$: roughly speaking, a singular vector is a function on $\cJ_{0,\infty}$ (rather, a section of a line bundle over $\cE$) with vanishing differential, so it must be a constant.

\begin{proof}[Algebraic proof of Theorem \ref{thm:module_structure}]
Let $M_{c_\rM,\lambda}$ be the (abstract) Verma module of central charge $c_\rM$ and $V_{c_\rM,\lambda}=M_{c_\rM,\lambda}/I_{c_\rM,\lambda}$ the irreducible quotient of $M_{c_\rM,\lambda}$ by the maximal proper submodule $I_{c_\rM,\lambda}$. We refer the reader to Appendix \ref{app:virasoro} for some details on these notions.

If $\lambda\not\in kac$, the Verma module is irreducible, i.e. $I_{c_\rM,\lambda}=\{0\}$. Since $\cV_\lambda$ is a quotient of the Verma module by a submodule, we have necessarily $\cV_\lambda\simeq M_{c_\rM,\lambda}$ as highest-weight representations. Hence $\dim\cV_\lambda^N=\#\cT_N$. On the other hand, $\cV_\lambda^N$ is contained in the space of polynomials at level $(N,0)$, which has dimension $\#\cT_N$. We deduce by dimension counting that $\cV_\lambda^N$ is exactly the space of polynomials at level $(N,0)$.

We turn to Item \ref{item:irred_quotient}. We claim that $\C[(a_m)_{m\geq1}]$ equipped with the representation $(\bL_n^\lambda)_{n\in\Z}$ is isomorphic to the dual Verma module $M_{c_\rM,\lambda}^\vee$ (the space of linear forms on $M_{c_\rM,\lambda}$). Recall the notation \eqref{eq:string} of $L_{-\bk}$ and denote by $(e_\bk=L_{-\bk}\cdot e_\emptyset)_{\bk\in\cT}$ a linear basis of $M_{c_\rM,\lambda}$, where $e_{\eset}$ denotes the highest-weight vector in $M_{c_\rM,\lambda}$. For all $\bk,\bk'\in\cT$ with $|\bk|=|\bk'|$, the polynomial $\bL_\bk^\lambda(\ba^{\bk'})$ has level 0 (since $\ba^{\bk'}$ is at level $|\bk'|$ and $\bL_{\bk}^\lambda$ decreases the level by $|\bk|$), i.e. it is proportional to the constant function $\ind$, and we identify it with this constant. This allows us to define a linear map $\Phi_\lambda:\C[(a_m)_{m\geq1}]\to M_{c_\rM,\lambda}^\vee$ by (recall the notation \eqref{eq:string})
\begin{equation*}
\Phi_\lambda\ba^\bk(e_{\bk'}):=\left\lbrace\begin{aligned}
&\bL_{\bk'}(\ba^\bk),&&\text{if }|\bk|=|\bk'|,\\
&0,&&\text{otherwise}.
\end{aligned}\right.
\end{equation*}

 This linear map preserves the Virasoro representations by definition. In particular, it preserves the level. Thus, to prove that $\Phi_\lambda$ is a linear isomorphism, it suffices to prove that it is injective on each level. Let $\bk=(k_m)_{m\geq1}\in\cT$ and $m_0:=\max\{m\geq1|\,k_m>0\}$. By Lemma \ref{lem:level} and its proof, for $n>0$, we have $\bL_n^\lambda=-\del_{a_n}+\cdots$, where the dots contain only differential operators in $a_m$ with $m>n$. Hence,
 \[\Phi_\lambda\ba^\bk(e_\bk)=(\bL_1^\lambda)^{k_1}(\bL_{2}^\lambda)^{k_2}\cdots(\bL_{m_0}^\lambda)^{k_{m_0}}\ba^\bk=\prod_{m=1}^{m_0}(-1)^{k_m}k_m!\neq0.\]
 This proves that $\Phi_\lambda$ is injective, and we deduce that $(\C[(a_m)_{m\geq1}],(\bL_n^\lambda)_{n\in\Z})\simeq M_{c_\rM,\lambda}^\vee$ as required. Hence, $\cV_\lambda=\mathrm{span}\{\bL_{-\bk}^\lambda\ind|\,\bk\in\cT\}$ is isomorphic to the highest-weight representation obtained by acting with the dual representation on the top vector in $M_{c_\rM,\lambda}^\vee$ (denoted $V_{c_\rM,\lambda}^\vee$ in Appendix \ref{app:virasoro}). By Lemma \ref{lem:irreducible_quotient}, this representation is irreducible.

 Finally, we turn to Item \ref{item:tensor_product}. The fact that $\cW_\lambda$ is the tensor product of two copies of $\cV_\lambda$ is obvious since the two representations commute. In the sequel, we assume $\lambda\not\in kac$. By dimension counting, we just need to show that $\oplus_{n+\tilde{n}\leq N}\cW_\lambda^{n,\tilde{n}}$ is a subspace of the space of polynomials of total degree bounded by $N$, for all $N\in\N$. Let $\bk_0,\tilde{\bk}_0\in\cT$ with $|\bk_0|+|\tilde{\bk}_0|=N$. On the one hand, the Shapovalov form is non-degenerate, so there exists $\bk,\tilde{\bk}\in\cT$ with $|\bk|+|\tilde{\bk}|=N$ and $\bL^\lambda_\bk\bL^\lambda_{\tilde{\bk}}\Psi_{\lambda,\bk_0,\tilde{\bk}_0}\neq0$, i.e. this polynomial is a non-zero constant. On the other hand, Lemma \ref{lem:level} shows that $\bL_\bk^\lambda\bar{\bL}_{\tilde{\bk}}^\lambda$ decreases the degree of $\Psi_{\lambda,\bk_0,\tilde{\bk}_0}$ by $|\bk|+|\tilde{\bk}|=N$. Hence, the total degree of $\Psi_{\lambda,\bk_0,\tilde{\bk}_0}$ is exactly $N$.
 \end{proof}

\begin{proof}[Analytic proof of Theorem \ref{thm:module_structure}]
  Let $P\in\C[(a_m)_{m\geq1}]$. We can view $P$ as a $c$-independent function on $\cJ_{0,\infty}$, and we note that $\bar{\cL}_nP=0$ for all $n\geq1$. 

Let $f\in\cE$ and consider the path $\R_+\ni t\mapsto f_t(z):=e^tf(e^{-t}z)$ in $\cE$, from $f$ to $\mathrm{id}_\D$. We have $\del_tf_t(z)=e^tf_t(z)-zf'(e^{-t}(z))=v_t\circ f_t(z)$, with $v_t(z)=z-\frac{f_t^{-1}(z)}{(f_t^{-1})'(z)}$. For each $t>0$, the function $f_t^{-1}$ extends analytically to a neighbourhood of $f_t(\bar{\D})$, so $P$ is differentiable in direction $\rv_t$ at $f_t$. By the fundamental theorem of calculus, we get for all $0<a<b$,
\begin{equation}\label{eq:ftc}
P(f_b)-P(f_a)=\int_a^b(\cL_{\rv_t}+\bar{\cL}_{\rv_t})P(f_t)\d t=\int_a^b\cL_{\rv_t}P(f_t)\d t.
\end{equation}
Since $P$ is continuous on $\cE$ (with respect to the local uniform topology), the LHS converges to $P(f)-P(\mathrm{id}_\D)$ as $a\to0$ and $b\to\infty$.

Now, suppose that $P$ is a singular vector for some representation $(\cV_\lambda,(\bL_n^\lambda)_{n\in\Z})$. Given the expression of $\bL_n^\lambda$ for $n\geq1$, we have $\cL_nP=\bL_n^\lambda P=0$ for all $n\geq1$ (with a slight abuse of notation). In the neighbourhood of 0, we have the expansion $\frac{f_t^{-1}(z)}{(f_t^{-1})'(z)}=z+O(z^2)$, so $v_t(z)=O(z^2)$. Hence, $\cL_{\rv_t}P=0$ for all $t>0$, and the RHS of \eqref{eq:ftc} vanishes for all $0<a<b$. Taking the limit $a\to0$, $b\to\infty$ gives $P(f)-P(\mathrm{id}_\D)=0$.
This shows that $P$ is constant on $\cE$, so there is no singular vector in $\cV_\lambda$ at a strictly positive level, i.e. $\cV_\lambda$ is irreducible. The rest of the proof of Theorem~\ref{thm:module_structure} follows as before.
\end{proof}


	\subsection{Computation at level 2}\label{S:level2}

Theorem \ref{thm:module_structure} characterises the structure of the module $\cV_\lambda$ for each $\lambda\in\C$. If $\lambda$ belongs to the Kac table, this implies some linear relations between states at certain levels, of which the ones in Corollary \ref{C:level2} is a special case. As an illustration of this phenomenon, we will now give the explicit computation at level 2.

\begin{proof}[Computational proof of Corollary \ref{C:level2}]
In this proof, we will denote by $\cA \psi = \psi''/\psi'$ the pre-Schwarzian derivative of a holomorphic function $\psi$. For a function $f\in\cE$, we have in our parametrisation:
\[\cA f^{-1}(0)=-\cA f(0)=-2a_1,\qquad\cS f^{-1}(0)=-\cS f(0)=-6(a_2-a_1^2).\]
We will slightly abuse notations by identifying $\varpi$ and $\vartheta$ with their restriction to $\cE$ (so we evaluate them on $c=0$).

First, we compute $\bL_{-1}^\lambda\ind$ and $\bL_{-2}^\lambda\ind$. By Lemma \ref{lem:level}, we have
\begin{align*}
&\bL_{-1}^\lambda\ind=-2\lambda P_{-1,0}+Q_{-1,0}=-2\lambda\varpi(\rv_{-1});\\
&\bL_{-2}^\lambda\ind=-2\lambda P_{-2,0}+Q_{-2,0}=-2\lambda\varpi(\rv_{-2})-\frac{c_\rM}{12}\vartheta(\rv_{-2}).
\end{align*}
In order to evaluate $\varpi(\rv_{-1})$ and $\varpi(\rv_{-2})$, we record the following expansion, where $\psi(z)=z(1+\sum_{m=1}^\infty b_mz^m)$ is conformal in the neighbourhood of $z=0$:
\begin{align*}
\frac{z^2\psi'(z)^2}{\psi(z)^2}
&=\frac{(1+2b_1z+3b_2z^2+o(z^2))^2}{(1+b_1z+b_2z^2+o(z^2))^2}
=1+2b_1z+(4b_2-b_1^2)z^2+o(z^2)\\
&=1+\cA\psi(0)z+(\frac{2}{3}\cS\psi(0)+\frac{3}{4}\cA\psi(0)^2)z^2+o(z^2).
\end{align*}
Applying this computation to $\psi=f^{-1}$, we find that
\[
\varpi(\rv_{-1}) = -\frac12 \cA f^{-1}(0)=a_1
\quad \text{and} \quad
\varpi(\rv_{-2}) = -\frac12 \Big( \frac23 \cS f^{-1}(0) + \frac34 \cA f^{-1}(0)^2 \Big)=2(a_2-a_1^2)-\frac{3}{2}a_1^2.
\]
Recalling that $\bL_\rv=\cL_\rv-\frac{c_\rM}{12}\vartheta(\rv)$ and that $\vartheta(\rv_{-1})=0$ and $\vartheta(\rv_{-2}) = -\cS f^{-1}(0)=6(a_2-a_1)^2$, we deduce that
\begin{equation}\label{eq:l_minus_one}
\bL_{-1}^\lambda\ind
=-2\lambda a_1
\quad \text{and} \quad
\bL_{-2}^\lambda\ind =3\lambda a_1^2 -\Big(4\lambda+\frac{c_\rM}{2}\Big)(a_2-a_1^2).
\end{equation}
Next, Proposition \ref{prop:generators_mobius} gives $P_{-1,1}=3(a_2-a_1^2)$ in the expression of $\bL_{-1}^\lambda$ given in \eqref{eq:op_lambda}, so that
\begin{align*}
(\bL_{-1}^\lambda)^2\ind=-2\lambda\bL_{-1}^\lambda(a_1)
&=2\lambda(2\lambda+1)a_1^2-6\lambda(a_2-a_1^2).
\end{align*}

  Wrapping up, we can present the values $(\bL_{-1}^2\ind,\bL_{-2}\ind)$ as a $2\times2$-linear system:
\[\begin{pmatrix}
(\bL_{-1}^\lambda)^2\ind \\ 
\bL_{-2}^\lambda\ind
\end{pmatrix}=\begin{pmatrix}
4\lambda(2\lambda+1) & 6\lambda \\ 
6\lambda & 4\lambda+\frac{c_\rM}{2}
\end{pmatrix} \begin{pmatrix}
\frac{1}{2}a_1^2\\ 
-(a_2-a_1^2)
\end{pmatrix}. \]
We recover the Gram matrix of the Shapovalov form as it appears in \cite[(7.23)]{dFMS_BigYellowBook}, with the dictionary $\lambda\leftrightarrow h$, $c_\rM\leftrightarrow c$. The determinant of this system is \cite[(7.25-26)]{dFMS_BigYellowBook} 
\[\det\begin{pmatrix}
4\lambda(2\lambda+1) & 6\lambda \\ 
6\lambda & 4\lambda+\frac{c_\rM}{2}
\end{pmatrix}=32\lambda(\lambda-\lambda_{1,2})(\lambda-\lambda_{2,1}),
\qquad \text{with}
\]
\begin{align*}
&\lambda_{1,2}=\frac{1}{16}\left(5-c_\rM+\sqrt{(c_\rM-1)(c_\rM-25)}\right)
\quad \text{and} \quad
\lambda_{2,1}=\frac{1}{16}\left(5-c_\rM-\sqrt{(c_\rM-1)(c_\rM-25)}\right).
\end{align*}
Elementary algebra gives $\lambda_{1,2}=\frac{6-\kappa}{2\kappa}$, $\lambda_{2,1}=\frac{3\kappa-8}{16}$ \cite[(7.31)]{dFMS_BigYellowBook}.
We deduce that $(\bL_{-1}^\lambda)^2\ind$ and $\bL_{-2}^\lambda\ind$ are colinear for these values and these values only. It is an easy check that the corresponding linear combination for these values is $((\bL_{-1}^\lambda)^2-\frac{2}{3}(2\lambda+1)\bL_{-2}^\lambda)\ind=0$.
%
%
\end{proof}

\section{Integration by parts}\label{sec:IBP}

One of the main goals of this section is to prove the following integration by parts formula (recall the notations \eqref{eq:def_vartheta} of $\vartheta$ and $\tilde \vartheta$):

\begin{theorem}\label{T:ibp}
For all $\rv\in\C(z)\del_z$, the $L^2(\nu)$-adjoints $\cL_\rv^*$ and $\bar \cL_\rv^*$ of $\cL_\rv$ and $\bar \cL_\rv$, respectively, are densely defined. Hence, $\cL_\rv$ and $\bar\cL_\rv$ are closable, and their closures satisfy
\begin{equation}
    \label{E:T_ibp}
\cL_\rv^*=-\bar{\cL}_\rv-\frac{c_\rM}{12}\overline{\tilde{\vartheta}_\eta(\rv)}
+\frac{c_\rM}{12}\overline{\vartheta_\eta(\rv)}; \qquad
\bar\cL_\rv^*=-\cL_\rv-\frac{c_\rM}{12}\tilde{\vartheta}_\eta(\rv)
+\frac{c_\rM}{12}\vartheta_\eta(\rv).
\end{equation}
We will keep denoting by $\cL_\rv$ and $\bar\cL_\rv$ the corresponding closures.
\end{theorem}
This result is essentially equivalent to \cite[Corollary 3.11]{GQW24}, but the proof methods are different: we relate the differential of the mass of Brownian loops to the Brownian bubble measure, while \cite{GQW24} rely on results from \cite{carfagnini2023onsager,SungWang23_QCdef}.

In Proposition \ref{prop:ibp_disc} we derive an analogous result for the SLE loop measure in the unit disc.

Theorem \ref{T:ibp} allows us to construct the Shapovalov form $\cQ$. To this end, recall from \eqref{eq:star} that we define an antilinear involution of $\C(z)\del_z$ by $\rv^*=z^2\overline{v(1/\bar z)}\del_z$, an involution $\Theta$ of $L^2(\nu)$ (see \eqref{eq:def_Theta}) and a Hermitian $\cQ$ on $L^2(\nu)$ by $\cQ(F,G) = \langle F,\Theta G \rangle_{L^2(\nu)}$, $F,G \in L^2(\nu)$.

\begin{corollary}\label{cor:ibp}
For all $\rv \in \C(z)\partial_z$,
the $\cQ$-adjoint of $\bL_\rv$ is densely defined. Hence $\bL_\rv$ is closable and its closure satisfies
\begin{equation}
    \label{E:C_ibp}
\bL_\rv^*=\Theta\circ\bL_{\rv^*}\circ\Theta.
\end{equation}
In particular, for all $F\in\cD(\cL_{\rv^*}),\,G\in\cD(\cL_\rv)$, we have
$\cQ(\bL_{\rv^*}F,G)=\cQ(F,\bL_\rv G).$
\end{corollary}

\begin{remark}
In Section \ref{sec:hw_modules}, we defined the modules $\cW_\lambda$ as algebraic subspaces of $\C[(a_m,\bar{a}_m)_{m\geq1}]$, and the representations $(\bL_n^\lambda,\bar{\bL}_n^\lambda)_{n\in\Z}$ as endomorphisms acting on these spaces. Such definitions are sufficient for the purpose of characterising these modules algebraically. Now, we can also consider the closure of $\cW_\lambda$ with respect to the $L^2(\nu^{MKS,\#})$; this is a closed subspace of $L^2(\nu^{\#})$, which is proper if and only if $\lambda$ lies on the Kac table. The operators $\bL_n^\lambda,\bar{\bL}_n^\lambda$ extend to unbounded operators on the closure, and one can deduce from the above integration by parts formula that they are closable (the adjoints are densely defined).
\end{remark}

\subsection{Proof of Corollary \ref{cor:ibp}, assuming Theorem \ref{T:ibp}}

In this section, we obtain Corollary~\ref{cor:ibp}, assuming Theorem \ref{T:ibp}. We start with a preliminary lemma that will also be useful in the proof of Theorem \ref{T:ibp}.
Recall the space $\cC_\mathrm{comp}$ introduced in \eqref{E:cC_compact}.
As in the proof of Proposition \ref{P:commutation_bL} and to make the following computations clearer, we will denote by $\vartheta_\cdot(\rv)$ the function $\eta \in \cJ_{0,\infty} \mapsto \vartheta_\eta(\rv)$ and by $\vartheta_\cdot(\rv)\id_{\cC_\mathrm{comp}}$ the operator $F \in \cC_\mathrm{comp} \mapsto \vartheta_\cdot(\rv) F$.

\begin{lemma}\label{L:Theta}
    For all $\rv \in \C(z)\del_z$,
    \begin{equation}\label{E:L_Theta1}
        \Theta\circ\cL_{\rv^*}\circ\Theta=-\bar{\cL}_\rv; \qquad \Theta\circ\bar\cL_{\rv^*}\circ\Theta=-\cL_\rv;
    \end{equation}
    \begin{equation}\label{E:L_Theta2}
        \Theta \circ (\vartheta_\cdot(\rv^*) \id_{\cC_\mathrm{comp}}) \circ \Theta = \overline{\tilde \vartheta_\cdot (\rv)} \id_{\cC_\mathrm{comp}};
        \qquad \Theta \circ (\tilde \vartheta_\cdot(\rv^*) \id_{\cC_\mathrm{comp}}) \circ \Theta = \overline{\vartheta_\cdot (\rv)} \id_{\cC_\mathrm{comp}}.
    \end{equation}
\end{lemma}

\begin{proof}
We start by proving \eqref{E:L_Theta1}. 
From the explicit expressions of $\cL_0$ and $\bar\cL_0$ from Proposition~\ref{prop:generators_mobius}, one can directly check \eqref{E:L_Theta1} in this case. 
Since $\rv \mapsto \rv^*$ maps $z^2\C[z]\del_z$ to $\C[z^{-1}] \del_z$, if \eqref{E:L_Theta1} holds for all vector fields in $z^2\C[z]\del_z$, by applying $\Theta$ to the left and to the right of \eqref{E:L_Theta1}, we see that it also holds for all vector fields in $\C[z^{-1}] \del_z$.
By linearity, the vector field $\rv_0$ covering the case of polynomial vectors fields of degree one in $z$, we would then get \eqref{E:L_Theta1} for all $\rv \in \C(z)\del_z$.
It is thus sufficient to consider the case $\rv \in z^2\C[z]\del_z$.
Let $\phi_t= \id + t v^* + o(t)$ be the flow generated by $\rv^*$ for small complex $t$. Let $\eta_t:=\phi_t(\eta)$ and $(c_{\eta_t},f_{\eta_t})\in\R\times\cE$, $(\tilde{c}_{\eta_t},g_{\eta_t})\in\R\times\tilde{\cE}$ be the uniformising maps of the interior and exterior of $\eta$. Since $\rv^*\in\C[z^{-1}]\del_z$, we have $e^{-\tilde{c}_{\eta_t}}g_{\eta_t}=\phi_t\circ(e^{-\tilde{c}_\eta}g_\eta)$. A small computation then shows that $\iota\circ g_{\eta_t}\circ\iota=\iota\circ g_\eta\circ\iota-\bar{t}v\circ(\iota\circ g_\eta\circ\iota)+o(t)$.
By definition of $\cL_\rv$ and $\bar\cL_\rv$, we then get that $(\Theta F)(c_{\eta_t},f_{\eta_t})$ is equal to
\begin{align*}
F(\tilde{c}_{\eta_t},\iota\circ g_{\eta_t}\circ\iota)
=F(\tilde{c}_\eta,\iota\circ g_\eta\circ\iota)-t\bar{\cL}_\rv F(\tilde{c}_\eta,\iota\circ g_\eta\circ\iota)-\bar{t}\cL_\rv F(\tilde{c}_\eta,\iota\circ g_\eta\circ\iota)+o(t).
\end{align*}
Thus,
\[\cL_{\rv^*}\circ\Theta F(c_\eta,f_\eta)=-\bar{\cL}_\rv F(\tilde{c}_\eta,\iota\circ g_\eta\circ\iota)
\quad \text{and} \quad
\bar \cL_{\rv^*}\circ\Theta F(c_\eta,f_\eta)=-\cL_\rv F(\tilde{c}_\eta,\iota\circ g_\eta\circ\iota)
\]
Applying $\Theta$ again gives \eqref{E:L_Theta1}.

We now prove \eqref{E:L_Theta2}. It is enough to prove the first equality. Let $\rv \in \C(z)\del_z$.
Note first that
\[
\Theta \circ (\vartheta_\cdot(\rv^*) \id_{\cC_\mathrm{comp}}) \circ \Theta = \Theta(\vartheta_\cdot(\rv^*)) \id_{\cC_\mathrm{comp}}.
\]
By definition of $\Theta$ and $\vartheta$,
we have
\[
\Theta(\vartheta_\cdot(\rv^*))(\eta)
= \vartheta_{\iota(\eta)}(\rv^*)
= \frac{e^{-\tilde c_\eta}}{2 i \pi} \oint \cS ( \iota \circ g_\eta^{-1} \circ \iota)(z) v^*(e^{\tilde c_\eta}z) \d z.
\]
By the chain rule of the Schwarzian,
$
\cS(\iota\circ g_\eta^{-1}\circ\iota)=\frac{1}{z^4}\overline{\cS g_\eta^{-1}}(1/z).
$
Plugging $v^*(z) = z^2 \overline{v(1/\bar z)}$ and performing the change of variables $z\mapsto1/\bar{z}$, we deduce that
\[\Theta(\vartheta_\cdot(\rv^*))(\eta)=-\frac{e^{\tilde{c}_\eta}}{2i\pi}\oint\overline{\cS g_\eta^{-1}}(\bar{z})\overline{v(e^{\tilde{c}_\eta}z)}\d\bar{z}=\overline{\frac{e^{\tilde{c}_\eta}}{2i\pi}\oint\cS g_\eta^{-1}(z)v(e^{\tilde{c}_\eta}z)\d z}=\overline{\tilde\vartheta_\eta(\rv)}\]
as desired.
\end{proof}

\begin{proof}[Proof of Corollary \ref{cor:ibp}, assuming Theorem \ref{T:ibp}]
Let $\rv \in \C(z)\del_z$.
By definition \eqref{E:bL} of $\bL_{\rv^*}$ in terms of $\cL_{\rv^*}$ and
combining \eqref{E:L_Theta1} and \eqref{E:L_Theta2}, we have
\[
\Theta \circ \bL_{\rv^*} \circ \Theta 
= \Theta \circ ( \cL_{\rv^*} - \frac{c_\rM}{12} \vartheta_\cdot(\rv^*) \id_{\cC_\mathrm{comp}} ) \circ \Theta
= - \bar \cL_\rv - \frac{c_\rM}{12} \overline{\tilde \vartheta_\cdot (\rv)} \id_{\cC_\mathrm{comp}}.
\]
By Theorem \ref{T:ibp}, the right-hand side agrees with $\bL_\rv^*$ showing the desired relation
$\bL_\rv^*=\Theta\circ\bL_{\rv^*}\circ\Theta.$
The last statement of Corollary \ref{cor:ibp} is immediate since then
\[\cQ(\bL_{\rv^*}F,G)=\langle\bL_{\rv^*}F,\Theta G\rangle_{L^2(\nu)}=\langle F,\Theta\circ\bL_\rv G\rangle_{L^2(\nu)}=\cQ(F,\bL_\rv G).\]
By \cite[Theorem~VIII.1.(b)]{ReedSimon1}, the fact that the $L^2(\nu)$-adjoint $\bL_\rv^*$ of $\bL_\rv$ is densely defined implies that the operator $\bL_\rv$ is closable.
\end{proof}

\subsection{Integration by parts in the unit disc}

As in Definition \ref{D:restriction_measure}, for any simply connected domain $D$ with the topology of the disc, define
\[ 
\d\nu_D(\eta):=\ind_{\eta \subset D} e^{\frac{c_\rM}{2}\Lambda^*(\eta,\del D)}\d\nu(\eta).
\]
In this section, we study the case of the unit disc.
We start by introducing a few notations:

\begin{notation}\label{N:annulus}
    For $q \in (0,1)$, we will denote by $\A_q$ the annulus $\D \setminus \overline{q \D}$. Let $\eta$ be a Jordan curve in $\D$. Let $\A_\eta:=\D\setminus\overline{\rmint(\eta)}$ be the annulus-type domain encircled by $\S^1$ and $\eta$. There exists a unique $q(\eta)\in(0,1)$ and a conformal map $\psi_\eta:\A_\eta\to\A_{q(\eta)}$ with $\psi_\eta(1)=1$.

Note that for all Jordan curve $\eta$ in $\D$, $\psi_\eta$ extends to a neighbourhood of $\S^1$ by Schwarz reflection, and in particular $\cS\psi_\eta$ is holomorphic in a neighbourhood of $\S^1$.
Given $\rv\in\C(z)\del_z$, we will write
\begin{equation}
    \label{E:residue_pairing}
    (\cS\psi_\eta,\rv):=\frac{1}{2i\pi}\oint_{\S^1}\cS\psi_\eta(z)v(z)\d z
\end{equation}
the residue pairing and similarly for $(\psi_\eta^{-2} {\psi_\eta'}^2,\rv)$.
\end{notation}

The main result of this section is the following integration by parts formula for $\nu_\D$:

\begin{proposition}\label{prop:ibp_disc}
For all $\rv\in\C[z]\del_z$ and $F,G\in\cC$, we have
\begin{equation}\label{E:prop_ibp_disc}
    (\cL_\rv F,G)_{L^2(\nu_\D)}
    = \Big(F, -\bar \cL_\rv G - \Big(\frac{c_\rM}{12} \overline{(\cS \psi_\eta,\rv)} + \frac{c_\rM}{2} U(q(\eta)) \overline{( \psi_\eta^{-2}{\psi_\eta'}^2, \rv)} \hspace{1pt} \Big) G \Big)_{L^2(\nu_\D)},
\end{equation}
\begin{equation}
    \label{E:U}
    \text{where} \qquad
    U(q) := \frac{1}{12} + \frac{\pi^2}{12|\log q|^2} - \frac{\pi^2}{2|\log q|^2} \sum_{n \ge 1} \sinh^{-2} \Big( \frac{n\pi^2}{|\log q|} \Big), \qquad q \in (0,1).
\end{equation}
Moreover, the same conclusion holds if $\nu$ is replaced by any weak restriction measure on $\cJ_{0,\infty}$.
\end{proposition}

 Notice that $\lim_{q \to 0^+} U(q) = 0$. As a consequence, the term featuring $U(q(\eta))$ disappears when passing to the limit to get the integration by parts formula in the sphere as stated in Theorem~\ref{T:ibp}. This term arises in computations with the Brownian bubble measure on the annulus $\A_q$, and is related to the boundary excursion kernel (see Section \ref{S:pf_ibp} for details). The limit $U(q)\to0$ can be interpreted probabilistically as the fact that $\{0\}$ is polar for the bubble measure in $\D$.

\noindent\textbf{Brownian loop measure and Brownian bubble measure}
The proof of Proposition \ref{prop:ibp_disc} will make use of two families of measures on Brownian paths (measures on continuous paths of finite duration).
The first one is the family of Brownian loop measures $\loopmeasure_D$, for any domain $D$, whose definition is given in \eqref{E:def_BLM} and \eqref{E:BLM_restriction}.
Secondly,
for a smooth bounded simply connected domain $D$ (we will only consider discs) and a boundary point $z \in \partial D$, the Brownian bubble measure $\bubmeasure_{D,z}$ in $D$ rooted at $z$ is a $\sigma$-finite measure on Brownian excursions from $z$ to $z$ in $D$. It is defined as
\[ 
\bubmeasure_{D,z} = \pi \lim_{\eps \to 0^+} \frac{1}{\eps} H_D(z+\eps n_z,z) \mu_D^\#(z+\eps n_z, z),
\]
where $n_z$ denotes the inward unit normal vector at $z$, $H_D$ is the Poisson kernel and $\mu_D^\#(z+\eps n_z, z)$ is the probability law of Brownian motion starting at $z+\eps n_z$, killed upon reaching $\partial D$, and conditioned on leaving $\partial D$ through $z$. We refer to \cite[Section 3.4]{lawler2004brownian} for more details.
The Brownian loop measure and the bubble measures are related by
(see \cite[Proposition 8]{lawler2004brownian}):
\begin{equation}
    \label{E:decompo_BLM}
\loopmeasure_\C = \frac{1}{\pi} \int_\C |\d z|^2 ~ \bubmeasure_{|z|\D,z}.
\end{equation}
This amounts to rerooting the loop at its point of maximal modulus. Hence, the identity above in fact holds when one sees the measures on both sides as measures on \emph{unrooted} loops.

The definition \eqref{E:def_BLM} of $\loopmeasure_\C$ uses Brownian bridge measures $\P_\C^{t,z,z}$. For fixed $t$ and $z$, computing the $\P_\C^{t,z,z}$-probability of a given crossing event can be quite involved, whereas the probability of the same event with respect to an excursion measure (with random duration) is usually much more explicit. We will thus use the bubble measure (and the decomposition \eqref{E:decompo_BLM}) as an effective tool to compute the $\loopmeasure_\C$-measure of some crossing event.
See in particular \eqref{E:pf_cross_bub0}.

\medskip

Given a subdomain $D \subset \D$ of the unit disc containing 0 and such that $\C \setminus D$ is non polar,
the Brownian loop measure naturally appears in the Radon--Nikodym derivative of $\nu_D$ with respect to $\nu_\D$. Indeed, the SLE loop measure is a weak restriction measure on $\cJ_{0,\infty}$ and thus satisfies (see Definition \ref{D:restriction_measure}):
\[ 
\d \nu_D(\eta) = e^{\frac{c_\rM}{2}(\Lambda^*(\eta,\partial D) - \Lambda^*(\eta,\partial \D)) } \indic{\eta \subset D}\d\nu_\D(\eta).
\]
By the restriction property \eqref{E:BLM_restriction} of the Brownian loop measure, one further gets that
\begin{equation}
    \label{E:RN}
\d \nu_D(\eta) = e^{\frac{c_\rM}{2}\Lambda_\D(\eta,\D \setminus D)} \indic{\eta \subset D}\d\nu_\D(\eta),
\quad \text{where} \quad
\Lambda_\D(\eta,\D \setminus D) = \loopmeasure_\D (\text{hit both } \eta \text{ and } \D \setminus D).
\end{equation}
Here, ``$\text{hit both } \eta \text{ and } \D \setminus D$" is the event that the sample of $\loopmeasure_D$ intersects both $\eta$ and the boundary of $\D\setminus D$. 
This transparent notation will be used throughout the proof.

\begin{proof}[Proof of Proposition \ref{prop:ibp_disc}]
In this proof, we assume that $\nu$ is any weak restriction measure on $\cJ_{0,\infty}$.
We will prove \eqref{E:prop_ibp_disc} for so-called Markovian vector fields
$\rv=v\del_z\in\C[z]\del_z$ (following the terminology of \cite{BGKRV22}), i.e. vector fields satisfying $\Re(\bar{z} v(z)) < 0$ for all $z \in \S^1$. This will be enough since they linearly generate $\C[z]\del_z$ (any vector field in $\C[z]\del_z$ can be made Markovian by adding a sufficiently large multiple of $\rv_0$).
So consider a Markovian vector field $\rv \in \C[z] \partial_z$ and the associated flow $\phi_t(z) = z + tv(z) + o(t)$ defined for small complex values of $t$. Because $\rv$ is Markovian, when $t$ belongs to a cone $\{ \arg(t) \in (-\delta, \delta) \}$ that is narrow enough,
the maps $\phi_t$ are well defined for all $t$ in this cone and contract the unit disc, i.e. $\phi_t(\bar{\D}) \subset \D$ \cite[Lemma 2.2]{BGKRV22}. In this proof we will always restrict ourselves to such values of $t$.
Let $F,G \in \cC$.
By dominated convergence theorem, $(\cL_\rv F, G)_{L^2(\nu_\D)}$ corresponds to the $t$-coefficient in the expansion of
\begin{align*}
\int ( F(\phi_t(\eta)) - F(\eta) ) \overline{G(\eta)} \d\nu_\D(\eta).
\end{align*}
By conformal invariance (see Definition \ref{D:restriction_measure}) and then by \eqref{E:RN} (using that $\phi_t(\bar{\D}) \subset \D$), the above display equals
\begin{align}
    \label{E:pf_P_ibp_disc}
& \int F(\eta) \overline{G(\phi_t^{-1}(\eta))} \d\nu_{\phi_t(\D)}(\eta)
- \int F(\eta) \overline{G(\eta)} \d\nu_{\D}(\eta) \\
& = \int F(\eta) \Big( \overline{G(\phi_t^{-1}(\eta))}  e^{\frac{c_\rM}{2}\Lambda_\D(\eta,\phi_t(\S^1))} \indic{\eta \subset \phi_t(\D)} - \overline{G(\eta)}\Big)\d\nu_\D(\eta).
\notag
\end{align}
In Lemmas \ref{L:loopmeasure_bub} and \ref{L:bub} below, we will show that
\begin{equation}\label{E:pf_prop_ibp_disc2}
    \Lambda_\D(\eta,\phi_t(\S^1))
= - t \Big( \frac{1}{6} (\cS \psi_\eta, \rv) + U(q(\eta)) (\psi_\eta^{-2}{\psi_\eta'}^2, \rv) \Big)
- \bar t \Big( \frac{1}{6} (\cS \psi_\eta, \rv^*) + U(q(\eta)) (\psi_\eta^{-2}{\psi_\eta'}^2, \rv^*) \Big) + o(t),
\end{equation}
where we recall that $\psi_\eta$ is defined in Notation \ref{N:annulus}.
The $t$-coefficient in the expansion of the term in parenthesis on the right hand side of \eqref{E:pf_P_ibp_disc} is thus equal to
\[
- \overline{\bar \cL_\rv G(\eta)} - \frac{c_\rM}{12} \overline{ G(\eta) } (\cS \psi_\eta, \rv) - \frac{c_\rM}{2} U(q(\eta)) \overline{G(\eta)} (\psi_\eta^{-2}{\psi_\eta'}^2, \rv).
\]
By dominated convergence theorem we then get \eqref{E:prop_ibp_disc}.

The rest of the section is dedicated to the proof of \eqref{E:pf_prop_ibp_disc2}. Up to considering the vector field $e^{i\theta}\rv$ for $|\theta|<\delta$, it suffices to consider the case of small positive values of $t$. We decompose the proof into two steps, corresponding to Lemmas \ref{L:loopmeasure_bub} and \ref{L:bub} respectively.

\begin{lemma}\label{L:loopmeasure_bub}
The following derivative estimate holds:
\begin{equation}
\lim_{\substack{t \to 0\\t>0}} \frac{1}{t} \Lambda_\D(\eta,\phi_t(\S^1))
=
-\frac{1}{\pi} \int_0^{2\pi} \d \theta~ \bubmeasure_{\D,e^{i\theta}}(\text{hit } \eta) \Re( e^{-i\theta} v(e^{i\theta}) ).
\end{equation}
\end{lemma}

\begin{proof}
Using the decomposition \eqref{E:decompo_BLM} of the Brownian loop measure, we have
\begin{equation}
\label{E:pf_cross_bub0}
\Lambda_\D(\eta,\phi_t(\S^1)) = \frac{1}{\pi} \int_0^1 \d r ~ r \int_0^{2\pi} \d \theta ~\bubmeasure_{r\D,r e^{i\theta}}(\text{hit~both~} \phi_t(\S^1) \text{ and } \eta).
\end{equation}
We also record the following estimate: for any $\theta \in [0,2\pi]$,
$\phi_t(e^{i\theta}) = e^{i\theta} + t v(e^{i\theta}) + o(t)$
implying that
\begin{equation}
\label{E:pf_cross_bub1}
|\phi_t(e^{i\theta})| = 1 + \Re(e^{-i\theta} v(e^{i\theta})) t + o(t).
\end{equation}

We now provide an upper bound for $\limsup_{t \to 0} \Lambda_\D(\eta,\phi_t(\S^1)) / t$.
Let $\delta >0$ be small but fixed and $C := \max_{0\le\theta \le 2\pi} \{- \Re(e^{-i\theta}v(e^{i\theta})) \}$. Let $\theta \in [0,2\pi]$.
By \eqref{E:pf_cross_bub1}, if $r \le 1 - (C+\delta) t$, then $\bubmeasure_{r\D,r e^{i\theta}}(\text{hit~} \phi_t(\S^1) )$ vanishes for $t$ small enough. Moreover, we claim that
if $1-(C+\delta)t \le r \le |\phi_t(e^{i\theta})| - \delta t$, then
\begin{equation}
\label{E:pf_cross_bub2}
\bubmeasure_{r\D,r e^{i\theta}}(\text{hit~} \phi_t(\S^1) ) \to 0 \quad \text{as} \quad t \to 0.
\end{equation}
Indeed, in this case and by \eqref{E:pf_cross_bub1} and continuity of $v$, there exists $c = c(\delta)>0$ such that for all $\theta' \in [\theta-c,\theta+c]$, $\phi_t(e^{i\theta'}) \notin r \D$ and for $\theta'$ outside this neighbourhood, $|\phi_t(e^{i \theta'})| \ge 1 - C t+o(t)$. So to get a chance of hitting $\phi_t(\S^1)$, the loop rooted at $re^{i\theta}$ would have to at least travel a macroscopic distance and hit a small region located near the boundary of $r\D$ (in the case where $\phi_t(\S^1) \cap r \D \neq \varnothing$). This shows \eqref{E:pf_cross_bub2}. Since the length of the interval $[1-(C+\delta)t , |\phi_t(e^{i\theta})| - \delta t]$ is of order $t$, it shows that
\begin{equation}
\label{E:pf_cross_bub3}
\lim_{t \to 0} \frac{1}{t} \int_0^{|\phi_t(e^{i\theta})| - \delta t} \d r ~r \bubmeasure_{r\D,r e^{i\theta}}(\text{hit~both~} \phi_t(\S^1) \text{ and } \eta) = 0.
\end{equation}
Consider now the case $r \ge |\phi_t(e^{i\theta})| - \delta t$.
We can bound
\[
\bubmeasure_{r\D,r e^{i\theta}}(\text{hit~both~} \phi_t(\S^1) \text{ and } \eta)
\le \bubmeasure_{r\D,r e^{i\theta}}(\text{hit~} \eta).
\]
The right hand side converges to $\bubmeasure_{\D,e^{i\theta}}(\text{hit~} \eta)$ as $r \to 1$.
Together with \eqref{E:pf_cross_bub1} and \eqref{E:pf_cross_bub3}, we obtain that 
\[
\limsup_{t \to 0} \frac{1}{t} \int_0^1 \d r ~r \bubmeasure_{r\D,r e^{i\theta}}(\text{hit~both~} \phi_t(\S^1) \text{ and } \eta) \le (-\Re(e^{-i\theta}v(e^{i\theta})) + \delta) \bubmeasure_{\D,e^{i\theta}}(\text{hit~} \eta).
\]
Since the left hand side does not depend on $\delta$, and going back to \eqref{E:pf_cross_bub0}, this shows that
\[
\limsup_{t \to 0} \frac{1}{t} \Lambda_\D(\eta,\phi_t(\S^1)) \le -\frac{1}{\pi} \int_0^{2\pi} \d \theta~ \bubmeasure_{\D,e^{i\theta}}(\text{hit } \eta) \Re( e^{-i\theta} v(e^{i\theta}) )
\]
as desired.

To get a lower bound, we simply bound
\[
\liminf_{t \to 0} \frac{1}{t} \Lambda_\D(\eta,\phi_t(\S^1))
\ge \frac{1}{\pi} \liminf_{t \to 0} \frac1t \int_0^{2\pi} \d \theta \int_{|\phi_t(e^{i\theta})| + \delta t}^1 \d r ~ r  \bubmeasure_{r\D,r e^{i\theta}}(\text{hit~both~} \phi_t(\S^1) \text{ and } \eta)
\]
where $\delta>0$ is small but fixed. If $r \ge |\phi_t(e^{i\theta})|+\delta t$ and $t$ is small enough, $re^{i\theta} \in \phi_t(\D)^c$ and
\[
\bubmeasure_{r\D,r e^{i\theta}}(\text{hit~both~} \phi_t(\S^1) \text{ and } \eta) = \bubmeasure_{r\D,r e^{i\theta}}(\text{hit~} \eta) = \bubmeasure_{\D,e^{i\theta}}(\text{hit~} \eta) - o(1)
\]
so that
\[
\liminf_{t \to 0} \frac{1}{t} \Lambda_\D(\eta,\phi_t(\S^1))
\ge - \frac{1}{\pi} \int_0^{2\pi} \d \theta \bubmeasure_{\D,e^{i\theta}}(\text{hit~} \eta) (\Re(e^{-i\theta}v(e^{i\theta})) + \delta).
\]
Sending $\delta \to 0$ finishes the proof.
\end{proof}

We now relate $\bubmeasure_{\D,e^{i\theta}}(\text{hit } \eta)$ to the quantity appearing in the statement of Proposition \ref{prop:ibp_disc}. Recall that $q(\eta)$ and $\psi_\eta$ are introduced in Notation \ref{N:annulus} and that $U(q)$ is defined in \eqref{E:U}.

\begin{lemma}\label{L:bub}
    Let $\eta$ be a Jordan curve in $\D$. Then, for any $\theta \in [0,2\pi]$,
    $e^{2i\theta} \cS \psi_\eta (e^{i\theta}) \in \R$ and $e^{i\theta} \psi_\eta'(e^{i\theta})/\psi_\eta(e^{i\theta}) \in \R$ are real numbers and
    \begin{equation}
        \label{E:L_bub}
\bubmeasure_{\D,e^{i\theta}}(\text{hit } \eta)
= \frac16 e^{2i\theta} \cS \psi_\eta (e^{i\theta})
+ e^{2i\theta} \frac{\psi_\eta'(e^{i\theta})^2}{\psi_\eta(e^{i\theta})^2} U(q(\eta)).
    \end{equation}
\end{lemma}

\begin{remark}
    This formula can be seen as an extension of \cite[Equation (13)]{lawler2004brownian} to a non simply connected setup.
\end{remark}

\begin{proof}
We start by checking that $e^{2i\theta} \cS \psi_\eta(e^{i\theta}) \in \R$.
This follows from the following more general result.
Let $\chi$ be an analytic diffeomorphism of $\S^1$. Then 
\begin{equation}
    \label{E:chi_real}
\text{for all } z\in\S^1, \quad z^2\cS\chi(z)\in\R.
\end{equation}
To prove \eqref{E:chi_real}, we write locally $\chi(e^{i\theta})=e^{ih(\theta)}$ for some analytic function $h$ defined in a complex neighbourhood of the real line, with $h(\theta)\in\R$ if $\theta\in\R$. In particular, $\cS h\in\R$ on $\R$ and $h'\in\R$ on $\R$. Let us denote $e_\alpha(z):=e^{\alpha z}$ and observe that $\cS e_\alpha(z)=-\frac{\alpha^2}{2}$. By the chain rule for the Schwarzian, we get
\[\frac{1}{2}-e^{2i\theta}\cS\chi(e^{i\theta})=\cS h(\theta)+\frac{1}{2}h'(\theta)^2.\]
Hence, $e^{2i\theta}\cS\chi(e^{i\theta})=-\cS h(\theta)+\frac{1}{2}(1-h'(\theta)^2)$ is real for any $\theta\in\R$. This proves \eqref{E:chi_real}. The proof that
\begin{equation}\label{E:pf_L_bub1}
    e^{i\theta} \frac{\psi_\eta'(e^{i\theta})}{\psi_\eta(e^{i\theta})} = |\psi_\eta'(e^{i\theta})|
\end{equation}
proceeds similarly and we omit the details.

We now turn to the proof of \eqref{E:L_bub}. Let $\theta \in [0,2\pi]$. 
We start by writing
\[
\bubmeasure_{\D, e^{i\theta}}(\text{hit } \eta) = \pi \lim_{\theta' \to \theta} \mu_{\D, e^{i\theta}, e^{i\theta'}}^\exc (\text{hit } \eta)
\]
where $\mu_{\D, e^{i\theta}, e^{i\theta'}}^\exc$ is a measure on Brownian excursions from $e^{i\theta}$ to $e^{i\theta'}$ whose total mass equals the boundary Poisson kernel $H_\D(e^{i\theta}, e^{i\theta'})$. See \cite[Sections 3.3 and 3.4]{lawler2004brownian}. By the restriction property, we further have
\[
\bubmeasure_{\D, e^{i\theta}}(\text{hit } \eta) = \pi \lim_{\theta' \to \theta} (|\mu_{\D, e^{i\theta}, e^{i\theta'}}^\exc| - |\mu_{\A_\eta, e^{i\theta}, e^{i\theta'}}^\exc|)
= \pi \lim_{\theta' \to \theta} (H_\D(e^{i\theta}, e^{i\theta'}) - H_{\A_\eta}(e^{i\theta}, e^{i\theta'})).
\]
The rest of the proof is then concerned with Poisson kernel computations. The boundary Poisson kernel in the unit disc is explicit and given by $H_\D(z,w) = \frac{1}{\pi|z-w|^2}$, for all $z,w \in \S^1$.
To deal with the Poisson kernel in $\A_\eta$, we first notice that, by conformal covariance of the boundary Poisson kernel \cite[Section 3.3]{lawler2004brownian},
\[
H_{\A_\eta}(e^{i\theta}, e^{i\theta'}) = |\psi_\eta'(e^{i\theta})| |\psi_\eta'(e^{i\theta'})| H_{\A_{q(\eta)}}(\psi_\eta(e^{i\theta}),\psi_\eta(e^{i\theta'})).
\]
The multivalued conformal map $z \mapsto \log z/|\log q|$ maps the annulus $\A_q$ to the vertical strip $V:= \{ z \in \C: -1 < \Re(z) < 0\}$, so
\[
H_{\A_q}(e^{i\theta},e^{i\theta'}) = \frac{1}{|\log q|^2} \sum_{n \in \Z} H_V \Big( i\frac{\theta}{|\log q|},i \frac{\theta' + 2n\pi}{|\log q|} \Big),
\]
the factor $1/|\log q|^2$ coming from the derivative of the conformal map.
    The boundary Poisson kernel in the strip $V$ is explicit (this boils down to finding an explicit conformal map from $V$ to the unit disc; see e.g. \cite{widder1961functions}): for all $y, y' \in \R$,
    \[
    H_V(iy,iy') = \frac{\pi}{4} \sinh^{-2} \Big( \frac{\pi}{2}(y-y') \Big).
    \]
    This shows that
    \begin{align*}
    & H_{\A_q}(e^{i\theta},e^{i\theta'}) = \frac{\pi}{4} \frac{1}{|\log q|^2} \sum_{n \in \Z} \sinh^{-2} \Big( \frac{\pi}{2} \frac{\theta'-\theta+2n\pi}{|\log q|} \Big) \\
    & ~~~ = \frac{1}{\pi|e^{i\theta} - e^{i\theta'}|^2} - \frac{1}{12 \pi} - \frac{\pi}{12|\log q|^2} + \frac{\pi}{2|\log q|^2} \sum_{n \ge 1} \sinh^{-2} \Big( \frac{n\pi^2}{|\log q|} \Big) + o(1)
    \qquad \text{as } \theta'\to \theta.
    \end{align*}
    Putting things together and using the notation $U(q)$ \eqref{E:U},
    \begin{align}
    \label{E:pf_L_bub}
        \bubmeasure_{\D,e^{i\theta}}(\text{hit } \eta)
        & = \lim_{\theta' \to \theta} \Big( \frac{1}{|e^{i\theta}-e^{i\theta'}|^2} - \frac{|\psi_\eta'(e^{i\theta})\psi_\eta'(e^{i\theta'})|}{|\psi_\eta(e^{i\theta'}) - \psi_\eta(e^{i\theta})|^2} \Big)
        + |\psi_\eta'(e^{i\theta})|^2 U(q(\eta)).
    \end{align}
    To compute the above limit, one can do some tedious Taylor expansion computations and find the result. We proceed differently by noticing that for any distinct $z, w \in \S^1$,
    \[
    \frac{1}{|z-w|^2} = - \Re \Big\{ \frac{z w}{(z-w)^2} \Big\}.
    \]
    Combining this with \eqref{E:pf_L_bub1}, we deduce that
    \[
    \lim_{\theta' \to \theta} \Big( \frac{1}{|e^{i\theta}-e^{i\theta'}|^2} - \frac{|\psi_\eta'(e^{i\theta})\psi_\eta'(e^{i\theta'})|}{|\psi_\eta(e^{i\theta'}) - \psi_\eta(e^{i\theta})|^2} \Big)
    = \lim_{\theta' \to \theta} \Re \Big\{ e^{i\theta} e^{i\theta'} \Big( \frac{\psi_\eta'(e^{i\theta})\psi_\eta'(e^{i\theta'})}{(\psi_\eta(e^{i\theta}) - \psi_\eta(e^{i\theta'}))^2} - \frac{1}{(e^{i\theta} - e^{i\theta'})^2} \Big) \Big\}.
    \]
    The term in parenthesis on the right hand side converges as $\theta' \to \theta$ to
    \[
    \frac{\partial^2}{\partial z \partial w} \log \Big( \frac{\psi_\eta(z) - \psi_\eta(w)}{z-w} \Big) \Big\vert_{z=w=e^{i\theta}}
    \]
    which agrees with $\cS \psi_\eta(e^{i\theta})/6$ by an elementary property of the Schwarzian derivative. Going back to \eqref{E:pf_L_bub}, recalling that $e^{2i\theta} \cS \psi_\eta(e^{i\theta}) \in \R$ and using \eqref{E:pf_L_bub1} once more, we have obtained \eqref{E:L_bub}
    as desired.
\end{proof}

Combining Lemmas \ref{L:loopmeasure_bub} and \ref{L:bub}, we obtain
\begin{align*}
\lim_{\substack{t \to 0\\t>0}} \frac{1}{t} \Lambda_\D(\eta,\phi_t(\S^1))
& = - \Re \Big\{ \frac{1}{\pi} \int_0^{2\pi} \d \theta \Big( \frac16 e^{2i\theta} \cS \psi_\eta (e^{i\theta})
+ e^{2i\theta} \frac{\psi_\eta'(e^{i\theta})^2}{\psi_\eta(e^{i\theta})^2} U(q(\eta)) \Big) e^{-i \theta} v(e^{i\theta}) \Big\} \\
& = -2 \Re \Big\{ \frac{1}{6} (\cS \psi_\eta, \rv) + U(q(\eta)) (\psi_\eta^{-2}{\psi_\eta'}^2, \rv) \Big\}
\end{align*}
which corresponds to the asymptotic behaviour \eqref{E:pf_prop_ibp_disc2} when $t$ is real.
This concludes the proof of Proposition \ref{prop:ibp_disc}.
\end{proof}

	\subsection{Proof of Theorem \ref{T:ibp}}\label{S:pf_ibp}

In this section, we will use a limiting procedure together with Proposition~\ref{prop:ibp_disc} to prove Theorem \ref{T:ibp}. We start with two lemmas.

\begin{lemma}\label{L:convergence_nu}
We have $(\log R)^\frac{c_\rM}{2}\nu_{R\D}\to\nu$ weakly as $R\to\infty$.
Moreover, the same holds for any weak restriction measure on $\cJ_{0,\infty}$.
\end{lemma}
\begin{proof}
In the case of the SLE loop measure $\nu$, this follows from \cite{Zhan21_SLEloop}. Indeed,
by the last sentence before \cite[Remark 5.3]{Zhan21_SLEloop}, $|\log \epsilon|^\frac{c_\rM}{2}\nu_{\epsilon\D^*}\to\nu$ weakly as $\epsilon\to0$. The result then follows by taking an inversion and by conformal invariance of $\nu_{\hat \C}$.
Consider now any any weak restriction measure $\nu^\MKS$ on $\cJ_{0,\infty}$.
By Definition \ref{D:restriction_measure} of weak restriction measures,
\[
\nu^\MKS = \lim_{R \to \infty} \ind_{\eta \subset R \D} \d \nu^\MKS(\eta)
= \lim_{R\to\infty} e^{-\frac{c_\rM}{2}\Lambda^*(\eta,R \S^1)}  \d \nu^\MKS_{R \D}(\eta)
\]
and $\Lambda^*(\eta,R \S^1) = \Lambda^*(\iota(\eta),R^{-1}\S^1) = -\log \log R + o(1)$ (this term does not depend on the specific choice of weak restriction measure on $\cJ_{0,\infty}$ and is thus the same as the one in \cite[p.~385]{Zhan21_SLEloop}).
\end{proof}

\begin{lemma}\label{L:conf_eps}
Let $\eta \in \cJ_{0,\infty}$ and
$e^{-\tilde c}g: \D^* \to \C \setminus \overline{\inte(\eta)}$ be a conformal map uniformising the exterior of $\eta$. Let $R_0>0$ be such that $\overline{\inte(\eta)} \subset R_0 \D$. For $R\ge R_0$, let $\psi_{\eta,R} : R \D \setminus \overline{\inte(\eta)} \to\A_{q(\eta,R)}$, for some $q(\eta,R) \in (0,1)$, be the conformal map with $q(\eta,R)^{-1} \psi_{\eta,R} \circ (e^{-\tilde c} g)(1) = 1$. Then, $q(\eta,R)^{-1} \psi_{\eta,R} \circ (e^{-\tilde c} g) \to \id_{\D^*}$ uniformly on $\overline{\D^*}$ as $R \to \infty$. More precisely,
$q(\eta,R)^{-1} \psi_{\eta,R} \circ (e^{-\tilde c} g)$ extends to a quasiconformal diffeomorphism of $\overline{\D^*}$ that converges uniformly.
\end{lemma}

\begin{proof}
This might be standard, but since we do not know any reference for this result, we provide a proof.
By conjugating with an inversion, it is equivalent to proving the following statement:

Let $\eta$ be a bounded Jordan curve containing 0 in its interior, and $f:\D\to\rmint(\eta)$ a conformal map. Let $r_0$ such that $r_0\bar{\D}\subset\rmint(\eta)$ and $\psi_r:\rmint(\eta)\setminus r\bar{\D}\to\A_{q(r)}$ be the conformal map such that $\psi_r\circ f(1)=1$, for all $0<r<r_0$. Then, $\psi_r\circ f\to\mathrm{id}_\D$ uniformly on $\bar{\D}$ as $r\to0$.

Assume without loss of generality that $f\in\cE$. The map $\phi_r:=\psi_r\circ f$ restricts to an analytic diffeomorphism of $\S^1$, so it extends analytically to a neighbourhood of $\S^1$ in $\D^*$ by Schwarz reflection. Pick any smooth quasiconformal extension of $\phi_r$ to $\bar{\D}$ and keep denoting it the same way. We can choose this quasiconformal extension so that it fixes 0. 

First, we give a (very rough) bound on $q(r)$. Let $C:=\max\{|f(z)|~|\,z\in\S^1\}$. By the Koebe $1/4$-theorem, we have $\frac{1}{4}\in\rmint(\eta)$. Fix $c>0$ such that $\psi_r(1/4)\geq c$ for all $r>0$. On the one hand, the harmonic measure of $r\S^1$ viewed from $\frac{1}{4}$ in $\rmint(\eta)$ is upper-bounded by $\log(1/4C)/\log(r/C)$. On the other hand, the harmonic measure of $q(r)\S^1$ viewed from $c$ in $\A_{q(r)}$ is $\log c/\log q(r)$. Hence, by conformal invariance of harmonic measure
\begin{equation}\label{eq:ub_q}
q(r)\leq\tilde{C}r^\alpha,
\end{equation}
with $\tilde{C}=C^{\frac{\log c}{\log(4C)}}$ and $\alpha=-\frac{\log c}{\log(4C)}>0$. 

Second, we adapt the proof of Montel's theorem in order to extract a subsequence converging uniformly on compacts of $\D\setminus\{0\}$. Let $(z_m)_{m\in\N}$ be a dense sequence in $\D\setminus\{0\}$. Since $\phi_r$ is uniformly bounded by 1, a standard diagonal argument shows that there exists a sequence $(\phi(z_m))_{m\in\N}$ and a subsequence $(\phi_{r_n})$ such that $\phi_{r_n}(z_m)\to\phi(z_m)$ simultaneously for all $m$ as $n\to\infty$. 

By Koebe $1/4$-theorem, $r\D\subset f(4\D)$ so $\phi_r$ is defined on $\A_{4r}$. By Cauchy's integral formula, we have for all $z,w\in\A_{8r}$, 
\[|\phi_r(z)-\phi_r(w)|\leq C'r^{-2}|z-w|,\]
for some $C'>0$. This implies that $\phi$ extends uniquely to a continuous function on $\D\setminus\{0\}$, and the convergence $\phi_{r_n}\to\phi$ is uniform on compacts of $\D\setminus\{0\}$. In particular, $\phi$ is holomorphic on $\D\setminus\{0\}$. Moreover, the first step gives $\phi(z)\to0$ as $z\to0$, so $\phi$ extends analytically on $\D$.

Since $\phi$ is not constant, it restricts to an analytic diffeomorphism of $\S^1$ as a uniform limit of analytic diffeomorphisms. Moreover, $\phi$ extends to the whole plane by Schwarz reflection, with a unique pole at $\infty$. Hence, $\phi$ is affine. Since $\phi$ fixes 1 and 0, we have $\phi=\mathrm{id}_\C$. 

Finally, by uniqueness of the subsequential limit, we deduce that $\psi_\epsilon\circ f$ converges uniformly on $\bar{\D}$ to the identity as $\epsilon\to0$.
\end{proof}

We now have all the ingredients to prove Theorem \ref{T:ibp}.

\begin{proof}[Proof of Theorem \ref{T:ibp}]
Let us first assume that \eqref{E:T_ibp} holds for all holomorphic vector fields $\rv \in \C[z]\del_z$. Let $\rv \in \C[z]\del_z$. Applying $\Theta$ to the left and to the right of the first identity in \eqref{E:T_ibp} and using Lemma~\ref{L:Theta}, we obtain that
\[
-\bar \cL_{\rv^*}^* = \cL_{\rv^*} - \frac{c_\rM}{12} \vartheta(\rv^*).
\]
This is the second identity in \eqref{E:T_ibp} for $\rv^*$. Since $\rv \mapsto \rv^*$ maps $\C[z]\del_z$ to $z^2\C[1/z]\del_z$ and by linearity, this shows that the second identity in \eqref{E:T_ibp} holds for all vector fields in $\C(z)\del_z$. The case of the first identity of \eqref{E:T_ibp} is symmetric and it remains to prove \eqref{E:T_ibp} for $\rv \in \C[z] \partial_z$. We focus on the first identity.

Let $\rv \in \C[z] \partial_z$, $F, G \in \cC$ and $R>0$. 
For any Jordan curve $\eta$ contained in $R\D$, we will denote by $\psi_{\eta,R}$ the conformal map introduced in Lemma \ref{L:conf_eps} which sends $R\D \setminus \overline{\inte(\eta)}$ to $\A_{q(\eta,R)}$.
By scaling and Proposition~\ref{prop:ibp_disc},
\begin{align*}
    (\cL_\rv F, G)_{L^2(\nu_{R\D})}
    & = \Big( F, -\bar \cL_\rv G - \Big(\frac{c_\rM}{12} \overline{(\cS \psi_{\eta,R},\rv)} + \frac{c_\rM}{2} U(q(\eta,R)) \overline{(\psi_{\eta,R}^{-2}{\psi_{\eta,R}'}^2, \rv)} \hspace{1pt} \Big) G \Big)_{L^2(\nu_{R \D})}.
\end{align*}
Since $\cS \psi_{\eta,R} = \cS (q(\eta,R)^{-1} \psi_{\eta,R})$, Lemma \ref{L:conf_eps} implies that $(\cS\psi_{\eta,R},\rv)\to(\cS((e^{-\tilde c} g_\eta)^{-1}),\rv)$ in $\cC^0(\cJ_{0,\infty})$.
By Lemma \ref{L:convergence_nu}, $(\log R)^\frac{c_\rM}{2}\nu_{R\D}\to\nu$ weakly. So $(\cS\psi_{\eta,R},\rv)(\log R)^\frac{c_\rM}{2}\nu_{R\D}\to(\cS((e^{-\tilde c} g_\eta)^{-1}),\rv)$ weakly.
Concerning the term featuring $U(q(\eta,R))$, and because $\lim_{q \to 0^+} U(q) = 0$, we similarly have the weak convergence
$U(q(\eta,R)) ( \psi_{\eta,R}^{-2}{\psi_{\eta,R}'}^2, \rv)(\log R)^\frac{c_\rM}{2}\nu_{R\D}\to0$.
Altogether, this shows that
\[
(\log R)^{\frac{c_\rM}{2}}
(\cL_\rv F, G)_{L^2(\nu_{R\D})}
\xrightarrow[R \to \infty]{} \Big(F, -\bar \cL_\rv G - \frac{c_\rM}{12} \overline{(\cS (e^{-\tilde c} g_\eta)^{-1}, \rv)} G \Big)_{L^2(\nu)}.
\]
This concludes the proof that
\[
(\cL_\rv F, G)_{L^2(\nu)}
= \Big(F, -\bar \cL_\rv G - \frac{c_\rM}{12} \overline{(\cS (e^{-\tilde c} g)^{-1}, \rv)} G \Big)_{L^2(\nu)}.
\]
Finally, by Proposition \ref{P:construction_cL} and Lemma \ref{L:preserve_cC}, $-\bar{\cL}_\rv-\frac{c_\rM}{12}\overline{\tilde{\vartheta}(\rv)}$ maps $\cC_\mathrm{comp}$ to $L^2(\nu)$. Since $\cC_\mathrm{comp}$ is dense in $L^2(\nu)$ (see Lemma \ref{L:dense}), this implies the closability of $\cL_\rv$ \cite[Theorem~VIII.1.(b)]{ReedSimon1} (we apply this reasoning to the Hilbert space $L^2(\nu)$ equipped with its $L^2$-inner product, not $\cQ$). 
\end{proof}

\section{Spectral resolution}\label{sec:spectral}

This section is devoted to the spectral resolution of the Shapovalov form. We will in particular prove Theorem~\ref{T:spectral}. 
We first record the elementary relation between the electrical thickness and the conformal radius viewed from $\infty$ (recall the definition of $\tilde{c}_0$ from \eqref{E:tilde_c_0}).
\begin{lemma}\label{lem:electrical_thickness}
The coordinate $-\tilde{c}_0$ on $\cE$ coincides with the electrical thickness of \cite{AngSun21_CLE}.
\end{lemma}
\begin{proof}
The electrical thickness is defined on $\cE$ as follows. Let $f\in\cE$ and fix the unique $\delta>0$ such that $e^\delta f(\S^1)\subset\overline{\D^*}$ and $e^\delta f(\S^1)\cap\S^1\neq\emptyset$. Then, the interior of the Jordan curve $e^{-\delta}/f(\S^1)$ is uniformised by $z\mapsto e^{\tilde{c}_0-\delta}/g(1/z)$. In \cite{AngSun21_CLE}, the electrical thickness is defined to be $-\delta-(\tilde{c}_0-\delta)=-\tilde{c}_0$. 
\end{proof}

\begin{proof}[Proof of Theorem \ref{T:spectral}]
Let $F,G\in\cC_c^\infty(\R)$, $\bk, \bk', \tilde \bk, \tilde{\bk}' \in \cT$ with $|\bk| = |\bk'|$ and $|\tilde\bk| = |\tilde\bk'|$ and $\lambda \in \C$ such that $2\Re (\lambda) + |\bk| + |\tilde \bk| +\kappa/8-1<0$.
Let us first show that
\begin{equation}
    \label{eq:fourier2}
\cQ\left(e^{-(2\lambda+|\bk|+|\tilde{\bk}|)c}F \Psi_{\lambda,\bk,\tilde\bk},e^{-(2\bar{\lambda}+|\bk|+|\tilde{\bk}|)c}G\Psi_{\bar{\lambda},\bk',\tilde\bk'}\right)=\frac{1}{\sqrt{\pi}}\int_\R\hat{F}(p)\overline{\hat{G}(-p)}\cQ_{\lambda+ip}\left(\Psi_{\lambda,\bk,\tilde\bk},\Psi_{\lambda,\bk',\tilde\bk'}\right)\d p.
\end{equation}
We have $e^{-(2\lambda+|\bk|+|\tilde{\bk}|)c}F\Psi_{\lambda,\bk,\tilde\bk},e^{-(2\bar{\lambda}+|\bk|+|\tilde{\bk}|)c}G\Psi_{\bar{\lambda},\bk',\tilde\bk'}\in\cC$, so the Hermitian form $\cQ$ evaluated at these two elements is well defined and finite. By definition \eqref{eq:def_shapo} of $\cQ$ and decomposing $\nu = \pi^{-1/2} \d c \otimes \nu^\#$, we have
\begin{align*}
&\cQ\left(e^{-(2\lambda+|\bk|+|\tilde{\bk}|)c}F\Psi_{\lambda,\bk,\tilde\bk},e^{-(2\bar{\lambda}+|\bk|+|\tilde{\bk}|)c}G\Psi_{\bar{\lambda},\bk',\tilde\bk'}\right)\\
&= \int_\R e^{-(2\lambda+|\bk|+|\tilde{\bk}|)c}F(c)\E^\#\left[e^{-(2\lambda+|\bk|+|\tilde{\bk}|)\tilde{c}}\overline{G(\tilde{c})}\Psi_{\lambda,\bk,\tilde\bk}(f)\overline{\Psi_{\bar{\lambda},\bk',\tilde\bk'}(\iota\circ g\circ\iota)}\right]\frac{\d c}{\sqrt{\pi}}.
\end{align*}
Using that $\tilde{c}=\tilde{c}_0-c$ and a change of variables, this is further equal to
\begin{align*}
&\E^\#\left[\Psi_{\lambda,\bk,\tilde\bk}(f)\overline{\Psi_{\bar{\lambda},\bk',\tilde\bk'}(\iota\circ g\circ\iota)}\int_\R F(c)\overline{G(\tilde{c}_0-c)}e^{-(2\lambda+|\bk|+|\tilde{\bk}|)\tilde{c}_0}\frac{\d c}{\sqrt{\pi}}\right]\\
&=\E^\#\left[\Psi_{\lambda,\bk,\tilde\bk}(f)\overline{\Psi_{\bar{\lambda},\bk',\tilde\bk'}(\iota\circ g\circ\iota)}\int_\R\hat{F}(p)\overline{\hat{G}(-p)}e^{-(2ip+2\lambda+|\bk|+|\tilde{\bk}|)\tilde{c}_0}\frac{\d p}{\sqrt{\pi}}\right].
\end{align*}
Now, since $2\Re(\lambda)+|\bk|+|\tilde \bk|+\kappa/8-1<0$, \eqref{eq:def_R} guarantees that the last line in the above display is integrable. Hence, we can apply Fubini to get \eqref{eq:fourier2}.
We will now prove each statement of Theorem~\ref{T:spectral}, starting with \eqref{E:T_Q_lambda}.

\textbf{Proof of \eqref{E:T_Q_lambda}.}
We are going to compute
\begin{equation}\label{E:pf_spectral3}
\cQ\left(\bL_{-\bk}\bar{\bL}_{-\tilde{\bk}}(e^{-2\lambda c}F),\bL_{-\bk'}\bar{\bL}_{-\tilde{\bk}'}(e^{-2\bar{\lambda}c}G)\right)
\end{equation}
in two ways. First, since $\cP_{\lambda+ip} \ind = e^{-2(\lambda+ip) c}$ for all $p \in \R$, we can write in $L^2(\R,\d c)$ by Fourier inversion:
\[
e^{-2\lambda c} F(c) = \int_\R \frac{\d p}{\sqrt{\pi}} ~\hat{F}(p) e^{-2(ip+\lambda)c} = \int_\R \frac{\d p}{\sqrt{\pi}} ~\hat{F}(p) \cP_{\lambda+ip} \ind.
\]
Using that $\bL_{-\bk} \bar \bL_{-\tilde \bk} \cP_{\lambda+ip} = \cP_{\lambda+ip} \bL_{-\bk}^{\lambda+ip} \bar{\bL}_{-\tilde \bk}^{\lambda+ip}$ (consequence of the conjugation relation of Proposition~\ref{lem:commutations_lambda}) and dominated convergence theorem, we infer that
\begin{equation}\label{E:pf_spectral4}
\bL_{-\bk}\bar{\bL}_{-\tilde{\bk}}(e^{-2\lambda c}F)
= \int_\R \frac{\d p}{\sqrt{\pi}} ~\hat{F}(p) \bL_{-\bk}\bar{\bL}_{-\tilde{\bk}} \cP_{\lambda+ip} \ind
= \int_\R \frac{\d p}{\sqrt{\pi}} ~\hat{F}(p) \cP_{\lambda+ip} \bL^{\lambda+ip}_{-\bk}\bar{\bL}^{\lambda+ip}_{-\tilde{\bk}} \ind.
\end{equation}
By Proposition \ref{lem:commutations_lambda}, $\bL^{\lambda+ip}_{-\bk}\bar{\bL}^{\lambda+ip}_{-\tilde{\bk}} \ind = \Psi_{\lambda+ip,\bk,\tilde \bk}$. By definition of $\cP_{\lambda+ip} = e^{-(\bL_0^{\lambda+ip}+\bar\bL_0^{\lambda+ip})}$ and by \eqref{E:L_0^lambda}, $\cP_{\lambda+ip} \Psi_{\lambda+ip,\bk,\tilde \bk} = e^{-(2\lambda+2ip+|\bk|+|\tilde\bk|)c} \Psi_{\lambda+ip,\bk,\tilde \bk}$. Altogether, this shows that
\[
\bL_{-\bk}\bar{\bL}_{-\tilde{\bk}}(e^{-2\lambda c}F)
=  e^{-(2\lambda+|\bk|+|\tilde\bk|)c}\int_\R \frac{\d p}{\sqrt{\pi}} ~\hat{F}(p) e^{-2ipc} \Psi_{\lambda+ip,\bk,\tilde \bk}.
\]
Since the Fourier transform of
\[
c \mapsto \int_\R \frac{\d p}{\sqrt{\pi}} ~\hat{F}(p) e^{-2ipc} \Psi_{\lambda+ip,\bk,\tilde \bk}
\qquad \text{is given by} \qquad
p \mapsto \hat{F}(p) \Psi_{\lambda+ip,\bk,\tilde\bk},
\]
an interpolation of \eqref{eq:fourier2} implies that \eqref{E:pf_spectral3} is equal to
\begin{equation}\label{E:pf_spectral6}
\int_\R\hat{F}(p)\overline{\hat{G}(-p)}\cQ_{\lambda+ip}\left(\Psi_{\lambda+ip,\bk,\tilde\bk},\Psi_{\bar{\lambda}+ip,\bk',\tilde\bk'}\right)\frac{\d p}{\sqrt{\pi}}.
\end{equation}
To get a second expression for \eqref{E:pf_spectral3}, we first use the integration by parts of Theorem \ref{T:intro_shapo}:
\[
\cQ\left(\bL_{-\bk}\bar{\bL}_{-\tilde{\bk}}(e^{-2\lambda c}F),\bL_{-\bk'}\bar{\bL}_{-\tilde{\bk}'}(e^{-2\bar{\lambda}c}G)\right)
=
\cQ\left(\bL_{\bk'}\bL_{-\bk}\bar{\bL}_{\tilde{\bk}'}\bar{\bL}_{-\tilde{\bk}}(e^{-2\lambda c}F),e^{-2\bar{\lambda} c}G\right).
\]
Similarly to \eqref{E:pf_spectral4}, we have
\[
\bL_{\bk'}\bL_{-\bk}\bar{\bL}_{\tilde{\bk}'}\bar{\bL}_{-\tilde{\bk}}(e^{-2\lambda c}F)
= \int_\R \frac{\d p}{\sqrt{\pi}} \, \hat{F}(p) \cP_{\lambda+ip} \bL^{\lambda+ip}_{\bk'}\bL^{\lambda+ip}_{-\bk}\bar{\bL}^{\lambda+ip}_{\tilde{\bk}'}\bar{\bL}^{\lambda+ip}_{-\tilde{\bk}} \ind.
\]
By definition of the Gram matrix, $\bL^{\lambda+ip}_{\bk'}\bL^{\lambda+ip}_{-\bk}\bar{\bL}^{\lambda+ip}_{\tilde{\bk}'}\bar{\bL}^{\lambda+ip}_{-\tilde{\bk}} \ind = \rB_{\lambda+ip}(\bk,\bk') \rB_{\lambda+ip}(\tilde{\bk},\tilde{\bk}')\ind$ and, because $\cP_{\lambda+ip}\ind = e^{-2(\lambda+ip)c}$, we find that
\[
\bL_{\bk'}\bL_{-\bk}\bar{\bL}_{\tilde{\bk}'}\bar{\bL}_{-\tilde{\bk}}(e^{-2\lambda c}F)
= e^{-2\lambda c} \int_\R \frac{\d p}{\sqrt{\pi}} \, \hat{F}(p) e^{-2ipc} \rB_{\lambda+ip}(\bk,\bk') \rB_{\lambda+ip}(\tilde{\bk},\tilde{\bk}').
\]
Noting that the Fourier transform of
\[
c \mapsto \int_\R \frac{\d p}{\sqrt{\pi}} \, \hat{F}(p) e^{-2ipc} \rB_{\lambda+ip}(\bk,\bk') \rB_{\lambda+ip}(\tilde{\bk},\tilde{\bk}')
\quad \text{is given by} \quad
p \mapsto \hat{F}(p) \rB_{\lambda+ip}(\bk,\bk') \rB_{\lambda+ip}(\tilde{\bk},\tilde{\bk}'),
\]
we can conclude as before, using \eqref{eq:fourier2}, that \eqref{E:pf_spectral3} equals
\begin{equation}\label{E:pf_spectral5}
 \int_\R \hat{F}(p)\overline{\hat{G}(-p)}\rB_{\lambda+ip}(\bk,\bk') \rB_{\lambda+ip}(\tilde{\bk},\tilde{\bk}')\cQ_{\lambda+ip}(\ind,\ind)\frac{\d p}{\sqrt\pi}.
\end{equation}
Recalling that $\rR(\lambda + ip) = \cQ_{\lambda+ip}(\ind,\ind)$ and because \eqref{E:pf_spectral6} and \eqref{E:pf_spectral5} agree for all $F,G \in \cC_c^\infty(\R)$, we get for Lebesgue-almost every $p \in \R$,
\[
\cQ_{\lambda+ip}( \Psi_{\lambda+ip,\bk,\tilde \bk}, \Psi_{\bar{\lambda}-ip,\bk',\tilde \bk'} )
= \rB_{\lambda+ip}(\bk,\bk') \rB_{\lambda+ip}(\tilde \bk,\tilde \bk') \rR(\lambda+ip).
\]
Finally, both functions are continuous (in fact, analytic) in $p\in\R$ since we are in the region of absolute convergence. 
Thus, the equality holds for all $p\in\R$. In the case where the $\lambda$-levels $(|\bk|,|\tilde \bk|)$ and $(|\bk'|,|\tilde \bk'|)$ are different, the proof is by all means identical and we find zero. This concludes the proof of \eqref{E:T_Q_lambda}.

\textbf{Proof of \eqref{E:T_spectral1}.}
Combining \eqref{eq:fourier2} and \eqref{E:T_Q_lambda} we see that \eqref{E:T_spectral1} holds for all $\lambda \in \C$ with $2\Re (\lambda) + |\bk| + |\tilde \bk| +\kappa/8-1<0$, i.e. that for such values of $\lambda$,
\[
\cQ\left(F\cP_\lambda\Psi_{\lambda,\bk,\tilde{\bk}},G\cP_{\bar{\lambda}}\Psi_{\bar{\lambda},\bk',\tilde{\bk}'}\right)=\int_\R\hat{F}(p)\overline{\hat{G}(-p)}\rR(\lambda+ip)\rB_{\lambda+ip}(\bk,\bk')\rB_{\lambda+ip}(\tilde{\bk},\tilde{\bk}')\frac{\d p}{\sqrt{\pi}}.
\]
We conclude by noticing that the LHS is analytic in the whole $\lambda$-plane, and the RHS is analytic in the region $\Re(\lambda)<\frac{1}{2}(1-\frac{\kappa}{8})$, so both sides coincide in this region. (In particular, the RHS analytically continues to the whole $\lambda$-plane too, but the analytic continuation of this integral includes some additional contributions when the poles of $\rR(\lambda)$ cross the integration line.)

\textbf{Proof of \eqref{eq:proj_F}.} The Fourier transform of any $F\in\cC_c^\infty(\R)$ extends analytically to the whole $p$-plane. Moreover, $\hat{F}$ decays faster than any polynomials as $|p|\to\infty$. Now, let $F\in\cC_\mathrm{comp}$, and recall that $(\Psi_{ip,\bk,\tilde{\bk}})_{\bk,\tilde{\bk}\in\cT}$ is a basis of $\C[(a_m,\bar{a}_m)_{m\geq1}]$ for all $p\in\R$ (Theorem \ref{thm:module_structure}). Hence, the Fourier transform $\hat{F}(p,\cdot)$ decomposes on this basis for each $p\in\R$, i.e. we can write
\[F=\int_\R\sum_{\bk,\tilde{\bk}\in\cT}\alpha_F(p,\bk,\tilde{\bk})e^{-2ipc}\Psi_{ip,\bk,\tilde{\bk}}\frac{\d p}{\sqrt{\pi}},\]
where $p\mapsto\alpha_F(p,\bk,\tilde{\bk})$ is analytic for all $\bk,\bk'\in\cT$ and decays faster than any polynomial as $|p|\to\infty$. Moreover, $\alpha_F(\,\cdot,\bk,\bk')=0$ for $|\bk|+|\tilde{\bk}|$ large enough ($F(p,\cdot)$ is a polynomial). Let $\cP:\cC_\mathrm{comp}\to\cC_\mathrm{comp}$ be the operator defined by
\[\cP F:=\int_\R\sum_{\bk,\tilde{\bk}\in\cT}\alpha_F(p,\bk,\tilde{\bk})e^{-(2ip+|\bk|+|\tilde{\bk}|)c}\Psi_{ip,\bk,\tilde{\bk}}\frac{\d p}{\sqrt{\pi}}.\]
(Formally, we have $\cP=\int_\R\cP_{ip}\frac{\d p}{\sqrt{\pi}}$.) This operator is clearly invertible, and by standard properties of the Laplace transform, we have $\alpha_{\cP^{-1}F}(p,\bk,\tilde{\bk})=\alpha_F(p-\frac{i}{2}(|\bk|+|\tilde{\bk}|),\bk,\tilde{\bk})=:\tilde{\alpha}_F(p,\bk,\tilde{\bk})$. Hence, we have in $L^2(\nu)$
\[F=\int_\R\sum_{\bk,\tilde{\bk}\in\cT}\tilde{\alpha}_F(p,\bk,\tilde{\bk})e^{-(2ip+|\bk|+|\tilde{\bk}|)c}\Psi_{ip,\bk,\tilde{\bk}}\frac{\d p}{\sqrt{\pi}}.\]

It just remains to express $\tilde{\alpha}(p,\bk,\tilde{\bk})$ in terms of the pairings $\cQ(F,e^{(2ip-|\bk|-|\tilde{\bk}|)c}\Psi_{-ip,\bk,\tilde{\bk}})$. Note that these pairings are well defined since $F\in\cC_\mathrm{comp}$. Moreover, for $\bk,\tilde\bk\in\cT$ fixed and $\Re(\lambda)<\frac{1}{2}(1-|\bk|-\tilde\bk-\frac{\kappa}{8})$, we have by \eqref{E:T_spectral1}
\begin{align*}
\cQ(F,\cP_{\bar{\lambda}}\Psi_{\bar{\lambda},\bk',\tilde{\bk}'})
&=\sum_{\bk,\tilde{\bk}\in\cT}\tilde{\alpha}_F(-i\lambda,\bk,\tilde{\bk})\cQ_\lambda(\Psi_{\lambda,\bk,\tilde{\bk}},\Psi_{\bar{\lambda},\bk',\tilde{\bk}'})\\
&=\sum_{\bk,\tilde{\bk}\in\cT}\tilde{\alpha}_F(-i\lambda,\bk,\tilde{\bk})\rR(\lambda)\rB_\lambda(\bk,\bk')\rB_\lambda(\tilde{\bk},\tilde{\bk}').
\end{align*}
Implicitly, we have first multiplied $\cP_\lambda\Psi_{\lambda,\bk',\tilde{\bk}'}$ by a test function $G$ (in order to apply $\eqref{E:T_spectral1}$), as in the proof of \eqref{E:T_Q_lambda}. Both sides of the last display are analytic for $\Re(\lambda)<\frac{1}{2}(1-\frac{\kappa}{8})$, so the equality extends to this region. 
Inverting this linear system for $\lambda=ip\in i\R$ yields
\[\tilde{\alpha}_F(p,\bk,\tilde{\bk})=\frac{1}{\rR(ip)}\sum_{\bk',\tilde{\bk}'\in\cT}\cQ\left(F,\cP_{-ip}\Psi_{-ip,\bk',\tilde{\bk}'}\right)\rB_{ip}^{-1}(\bk,\bk')\rB_{ip}^{-1}(\tilde{\bk},\tilde{\bk}'),\]
concluding the proof of \eqref{eq:proj_F}.
\end{proof}

\section{Uniqueness of restriction measures}\label{sec:uniqueness}

In this section, we prove the uniqueness of restriction measures as stated in Theorem \ref{thm:uniqueness}. Before moving to the actual proof, we would like to remind the reader that the representations $(\bL_n,\bar{\bL}_n)_{n\in\Z}$ and the map $\Theta$ depend only on the geometric structure of $\cJ_{0,\infty}$, and not on the measure that we put on it. The same holds for the representations $(\cW_\lambda,(\bL_n^\lambda,\bar{\bL}_n^\lambda)_{n\in\Z})$. What may depend on the measure are the properties determined by the associated $L^2$-product (and the integration by parts formula), in particular the adjoint relations $\Theta^*=\Theta$, $\bL_n^*=\Theta\circ\bL_{-n}\circ\Theta$, etc.

\begin{proof}[Proof of Theorem \ref{thm:uniqueness}]
The existence part is already known \cite{Zhan21_SLEloop}, so we only need to show uniqueness. 
A conformally invariant Borel measure on the space $\cJ$ is completely characterised by its restriction to $\cJ_{0,\infty}$. Since the restriction to $\cJ_{0,\infty}$ of a restriction measure on $\cJ$ (in the sense of Definition \ref{D:restriction_measure}) is a weak restriction measure on $\cJ_{0,\infty}$, it is enough to prove uniqueness of weak restriction measures on $\cJ_{0,\infty}$. Let $\tilde{\nu}$ be such a Borel measure on $\cJ_{0,\infty}$.
We first claim that $\tilde{\nu}$ is scale invariant, i.e. $\tilde{\nu}(\lambda \cdot) = \tilde{\nu}$ for $\lambda>0$. It is enough to check that for any $\lambda,R>0$, $\tilde{\nu}(\lambda \cdot) \ind_{\cdot \subset R\D} = \tilde{\nu}(\cdot) \ind_{\cdot \subset R\D}$. This equality is a consequence of the fact that $\tilde{\nu}$ is a weak restriction measure and the scale invariance of $\Lambda^*$, proving the desired claim.
Now,
scale invariance of $\tilde{\nu}$ implies that (up to rescaling by a multiplicative constant) $\tilde{\nu}=\pi^{-1/2} \d c\otimes\tilde{\nu}^\#$ for some probability measure $\tilde{\nu}^\#$ on $\cE$. Let $\tilde{\E}^\#$ be the expectation w.r.t. $\tilde{\nu}^\#$. 

As in Definition \ref{D:restriction_measure}, for any domain $D\subset \C$ containing $0$, let
$\d\tilde{\nu}_D(\eta):=\ind_{\eta \subset D} e^{\frac{c_\rM}{2}\Lambda^*(\eta,\del D)}\d\tilde\nu(\eta)$.
As stated in Proposition~\ref{prop:ibp_disc}, $\tilde{\nu}_\D$ satisfies the integration by parts formula of Proposition \ref{prop:ibp_disc}.
Moreover, by Lemma \ref{L:convergence_nu}, $(\log R)^{c_\rM/2} \tilde\nu_{R\D} \to \tilde\nu$ weakly as $R \to \infty$. 
The same limiting procedure as the one described in the proof of Theorem \ref{T:ibp} enables us to infer that, for any $F, G \in \cC_\mathrm{comp}$,
\[
(\cL_\rv G, F)_{L^2(\tilde\nu)}
= \Big(G, -\bar \cL_\rv F - \frac{c_\rM}{12} \overline{(\cS (e^{-\tilde c} g)^{-1}, \rv)} F \Big)_{L^2(\tilde\nu)},
\]
and a similar formula holds with $\bar\cL_\rv$ instead of $\cL_\rv$. In particular, applying the previous formula to $G$ constant in the support of $F$, one gets that
\[
    \int\bL_nF\d\tilde{\nu}=0\qquad\text{and}\qquad\int\bar\bL_nF\d\tilde{\nu}=0, \qquad \forall n<0, \,\forall F\in\cC_\mathrm{comp}.
\]

It is a simple exercise to check that the Kac table is bounded from below by $-\frac{(\kappa-4)^2}{16\kappa}$, and we fix an arbitrary $\lambda_0<-\frac{(\kappa-4)^2}{16\kappa}<0$. Note that $\tilde\E^\#[e^{-2\lambda_0\tilde{c}_0}]<\infty$ since $\tilde{c}_0\leq0$ with full $\tilde{\nu}^\#$-probability (see e.g. \cite[Equation (21)]{Pommerenke75}). Let $N,\tilde{N}\in\N$, $n\in\Z_{<0}$, and fix $P\in\cW_{\lambda_0-N-\tilde N+n}^{N,\tilde{N}}$ (recall the notation from Theorem \ref{thm:module_structure}). By definition of the operators $\bL_n^\lambda$ (recall \eqref{eq:L_lambda} and \eqref{E:cP_lambda}), the previous equation implies
\[\tilde\E^\#[e^{-2\lambda_0\tilde{c}_0}\bL_n^{\lambda_0-N-\tilde N+n}(P)]=0\qquad\text{and}\qquad\tilde{\E}^\#[e^{-2\lambda_0\tilde{c}_0}\bar\bL_n^{\lambda_0-N-\tilde{N}+n}(P)=0].\]
Since the spaces $\cW_\lambda^{N,\tilde{N}}$ are generated by all the Virasoro descendants of the constant function $\ind$, we infer that for all $P\in\cW_{\lambda_0-N-\tilde{N}}^{N,\tilde{N}}$:
\begin{equation}\label{eq:characterisation}
 \tilde{\E}^\#[e^{-2\lambda_0\tilde{c}_0}P]=\left\lbrace\begin{aligned}
 &0&&\text{if }N+\tilde{N}>0\\
 &\tilde{\E}^\#[e^{-2\lambda_0\tilde{c}_0}]&&\text{if }P=\ind.
 \end{aligned}\right.
 \end{equation}
 Finally, since $\lambda_0-N-\tilde{N}$ does not belong to the Kac table for any $N,\tilde{N}\in\N$, an elementary induction using the third item of Theorem \ref{thm:module_structure} implies that for all $N\in\N$, we have $\oplus_{n+\tilde{n}\leq N}\cW_{\lambda_0-n-\tilde{n}}^{n,\tilde{n}}=\mathrm{span}\{\ba^\bk\bar{\ba}^{\tilde{\bk}}|\,|\bk|+|\tilde{\bk}|\leq N\}$. Hence, $\oplus_{n,\tilde{n}\in\N}\cW_{\lambda_0-n-\tilde{n}}^{n,\tilde{n}}=\C[(a_m,\bar a_m)_{m\geq1}]$.  

We are now in position to conclude the proof. The space $\cE$ is a Radon space, and $\C[(a_m,\bar{a}_m)_{m\geq1}]$ is dense in $\cC^0(\cE)$ by Lemma \ref{L:dense}. Hence, we can conclude by the Riesz--Markov representation theorem (e.g. \cite[Theorem 2.14]{Rudin74}) that \eqref{eq:characterisation} characterises uniquely the measure $e^{-2\lambda_0\tilde{c}_0}\tilde{\nu}^\#$ in the space of Borel measures on $\cE$, up to the global multiplicative constant $\tilde{\E}^\#[e^{-2\lambda_0\tilde c_0}]$. Namely, $e^{-2\lambda_0\tilde{c}_0}\tilde{\nu}^\#=Ce^{-2\lambda_0\tilde{c}_0}\nu^\#$ for some $C>0$. It follows that $\tilde{\nu}^\#=C\nu^\#$ and $C=1$ since both are probability measures.

\end{proof}

\begin{remark}
The operators $\bL_n^\lambda, \bar \bL_n^\lambda$ \eqref{eq:op_lambda} are well defined for all \emph{complex} values of $c_\rM$ and form two commuting representations of the  Virasoro algebra with central charge $c_\rM$. From these operators and as before, we can then define $\Psi_{\lambda,\bk,\tilde{\bk}}=\bL_{-\bk}^\lambda \bar \bL_{-\tilde \bk}^\lambda \ind$. For $\lambda$ not in the Kac table, the Riesz-Markov representation theorem shows that the assignment $\Psi_{\lambda,\bk,\tilde{\bk}}\to\delta_{(\bk,\tilde{\bk})=(\emptyset,\emptyset)}$ extends uniquely to a \emph{complex} Borel measure $\nu^\#_{c_\rM}$ on $\cE$ (recall \eqref{eq:characterisation}). On the other hand, Definition \ref{D:restriction_measure} also makes sense for all complex $c_\rM$, and following the same route as in our proof of uniqueness, such a measure (rather, its shape) is determined by the above measure $\nu^\#_{c_\rM}$. So, if such a measure exists, it is unique. With some additional effort, it should be possible to prove the existence as well, i.e. that $\nu^\#_{c_\rM}$ satisfies the properties of Definition \ref{D:restriction_measure}.
Of course, the existence of the SLE loop measure grants for free that the existence and positivity for $c_\rM\leq1$. For non-real $c_\rM$, it is clear that a restriction measure cannot be positive. Finally, we believe that positivity fails for $c_\rM>1$ (due to positivity properties of the Shapovalov form in this region, it is very plausible that an odd function of $\Theta$ would have negative ``$L^2$-norm").
\end{remark}

\section{Applications}\label{S:applications}

We now discuss two consequences of the characterisation of the SLE loop measure in terms of restriction property.

\subsection{Reversibility}\label{S:reversibility}
The reversibility of the SLE loop measure is its invariance under the map $\iota:z\mapsto1/\bar{z}$, which is present in Zhan's axiomatisation. This property is not obvious \textit{a priori}: in the chordal case for instance, Schramm's definition of SLE only uses scale invariance \cite{Schramm2000}, but not the full M\"obius invariance. In fact, the key property of invariance under inversion $z\mapsto1/z$ was not part of the axioms and remained to be proved. This proof came only several years later, in \cite{zhan2008reversibility} for the simple case $\kappa\in(0,4]$, and in \cite{Miller16_ig3} for the non-simple case $\kappa\in(4,8)$.

Now, we explain how Theorem \ref{thm:uniqueness} provides an alternative and simpler approach to the question of reversibility. We remind the reader that our proof of uniqueness uses only the \emph{scale} invariance of the SLE loop measure (which allows to write $\nu=\d c\otimes\nu^\#$), and not the full M\"obius invariance (hence the argument is not recursive). By invariance of the Brownian loop measure under complex conjugation and M\"obius inversion, it follows immediately from Definition \ref{D:restriction_measure} that $\iota^*(\d c\otimes\nu^\#)$ is also a restriction measure, hence $\iota^*(\d c\otimes\nu^\#)=K\d c\otimes\nu^\#$ for some $K>0$. To see that $K=1$, observe that for all $F\in\cC_c^\infty(\R)$ with $\int_\R F(c)\d c=1$, we have 
\[1=\int_\R F(c)\d c=K\E^\#[\int F(\tilde{c})\d c]=K\E^\#[\int_\R F(c-\tilde{c}_0)\d c]=K\E^\#[\int_\R F(c)\d c]=K\int_\R F(c)\d c=K.\] 

As mentioned to us by W. Werner, and before the proof of Zhan \cite{zhan2008reversibility} (see also \cite{10.1214/14-AOP943} and \cite{lawler2021new}), such an approach for proving reversibility of SLE had been considered, but a characterisation of SLE in terms of restriction property was missing; see for instance \cite[Section~8.1]{10.1214/154957805100000113}. Theorem \ref{thm:uniqueness} is this missing piece, at least in the loop context.

\subsection{Duality}

Our second application concerns the duality $\kappa \leftrightarrow 16/\kappa$ for SLE. The duality of chordal SLE was first conjectured by Duplantier (see also Dubédat \cite{10.1214/009117904000000793}) and later proved by Zhan \cite{zhan2008duality}, Dubédat \cite{dubedat2009duality} and Miller--Sheffield \cite{MR3477777}. Duality for the SLE loop, as stated in Theorem \ref{T:duality} below, appears to be new. 
We mention that the paper in preparation \cite{AngCaiSunWu} gives an independent proof of this result, based on conformal welding of quantum discs and L\'evy excursion theory (private communication with the authors).

In the rest of this section, we make the dependence of the measures on $\kappa$ explicit in the notation, since it involves both $\kappa$ and $16/\kappa$ (e.g. we write $\nu_\kappa$ for $\nu$). We recall from \cite{Zhan21_SLEloop} that the SLE loop measure is well defined for all values of $\kappa \in (0,8)$, and it is an infinite (but $\sigma$-finite) measure. In the regime $\kappa\in(4,8)$, samples $\eta'$ are almost surely self-intersecting, but not self-crossing. The curve $\eta'$ is a random connected subset of $\hat{\C}$, so the connected components of $\hat{\C}\setminus\eta'$ are simply connected. We let $\nu_{\kappa'}$ be the restriction of the SLE$_{\kappa'}$ loop measure to those loops having 0 and $\infty$ in different connected components. Denote by $D_0$ (resp. $D_\infty$) the connected component containing 0 (resp. $\infty$). We define the \emph{outer boundary} of $\eta'$ to be the boundary of $D_\infty$, and we denote it by $\eta$. At this stage, we do not know if $\eta$ is a Jordan curve, but there is a unique conformal map $e^{-\tilde{c}}g:\D^*\to D_\infty$ with $\tilde{c}\in\R$ and $g\in\tilde{\cU}:=\{g:\D^*\to\hat{\C}\text{ conformal},\,g(\infty)=\infty,\,g'(\infty)=1\}$. We denote by $\mu_\kappa$ the law of $e^{-\tilde{c}}g$, which is a $\sigma$-finite Borel measure on $\R\times\tilde{\cU}$. By scale invariance, we have a factorisation $\mu_\kappa=\frac{\d\tilde{c}}{\sqrt{\pi}}\otimes\mu_\kappa^\#$, where $\mu_\kappa^\#$ is a Borel probability measure on $\tilde{\cU}$ and the global multiplicative constant has been fixed for convenience. We will abuse notations by viewing $\mu_\kappa^\#$ as a Borel measure on $\cU=\{f:\D\to\hat{\C}\text{ conformal},\,f(0)=0,\,f'(0)=1\}$ using the isomorphism $\tilde{\cU}\to\cU,\,g\mapsto\iota\circ g\circ\iota$.

\begin{theorem}[Duality]\label{T:duality} Fix $\kappa' \in (4,8)$ and let $\kappa = 16/\kappa'\in(2,4)$. With the notations just above, the outer boundary of $\eta'$ is $\nu_{\kappa'}$-a.s. a Jordan curve, and we have $\mu_\kappa=\nu_\kappa$ as Borel measures on $\cJ_{0,\infty}$ (i.e. the outer boundary of an SLE$_{\kappa'}$-loop is an SLE$_\kappa$-loop).
\end{theorem}

Theorem \ref{T:duality} does not address the self-dual case $\kappa'=\kappa=4$. In this case, the curve is almost surely simple, and the duality for SLE$_4$ can be understood as the reversibility proved in Section~\ref{S:reversibility}.

By conformal restriction, Theorem~\ref{T:duality} implies a similar result in any simply connected domain $D$, where one looks at the boundary of the connected component which contains the boundary~$\partial D$.

\begin{proof}
Before going into the actual proof, we make a preliminary remark. Every $f\in\cU$ has a Taylor expansion $f(z)=z(1+\sum_{m=1}^\infty a_mz^m)$ in the neighbourhood of $0$, and $\C[(a_m,\bar{a}_m)_{m\geq1}]$ is dense in $\cC^0(\cU)$, see Section \ref{subsec:dense}. Hence, a Borel probability measure is characterised by its evaluation on the space of polynomials in these coefficients. On the other hand, for $\kappa\in(0,4]$, we can view the SLE$_\kappa$ shape measure $\nu_\kappa^\#$ as a Borel probability measure on $\cU$ giving full mass to $\cE$ (corresponding to Jordan curves, see Section \ref{SS:notation}). Thus, to prove that $\mu^\#_\kappa=\nu_\kappa^\#$, it suffices to prove that these measures coincide on $\C[(a_m,\bar{a}_m)_{m\geq1}]$, and we will get as a corollary that the outer boundary of SLE$_{\kappa'}$ is $\nu_{\kappa'}$-a.s. a Jordan curve. 

For all simply connected domains $D\subset\C$, the SLE$_{\kappa'}$ loop measure in $D$ is by definition $\d\nu_{\kappa',D}(\eta')=\ind_{\eta'\subset D}e^{\frac{c_\rM}{2}\Lambda^*(\eta',D^c)}\d\nu_{\kappa'}(\eta')$, and it satisfies the restriction property with central charge $c_\rM=c_\rM(\kappa')=1-6(\frac{2}{\sqrt{\kappa'}}-\frac{\sqrt{\kappa'}}{2})^2$ \cite{Zhan21_SLEloop}.
Denote by $\mu_{\kappa,D}$ the distribution of the outer boundary a SLE$_{\kappa'}$ in $D$.
The main (elementary) observations to make are that $c_\rM = c_\rM(\kappa)$ and that for any \emph{simply connected} domain $D\subset \C$ and curve $\eta'$, denoting $\eta$ the outer boundary of $\eta'$,
\[
\ind_{\eta' \subset D} = \ind_{\eta \subset D}
\quad \text{and} \quad
\Lambda^*(\eta',D^c) = \Lambda^*(\eta,D^c).
\]
Thus, the restriction property of SLE$_{\kappa'}$ gives rise to a restriction property for its outer boundary:
$\d\mu_{\kappa,D}(\eta')=\ind_{\eta'\subset D}e^{\frac{c_\rM}{2}\Lambda^*(\eta',D^c)}\d\mu_{\kappa}(\eta')$.
Because $\nu_{\kappa',D}$ is conformally invariant, $\mu_{\kappa,D}$ is also conformally invariant. Overall, this proves that $\mu_\kappa$ is a $c_\rM$-restriction measure in the sense of Definition~\ref{D:restriction_measure}, except that it is defined on the space $\R \times \cU$ instead of $\R \times \cE$.
It then satisfies the integration by parts formula of Section \ref{S:pf_ibp}, which determines $\mu_\kappa^\#$ on all polynomials (following the proof of Theorem \ref{thm:uniqueness}). Hence, $\mu_\kappa^\#$ coincides with $\nu_\kappa^\#$ on $\C[(a_m,\bar{a}_m)_{m\geq1}]$, so $\mu_\kappa^\#=\nu_\kappa^\#$ in the space of Borel probability measures on $\cU\simeq\tilde{\cU}$. From our preliminary remark, we deduce that $\mu_\kappa^\#=\nu_\kappa^\#$ in the space of Borel probability measures on $\cU$, so $\mu_\kappa^\#$ descends to a probability measure on $\cE$ which it equals $\nu_\kappa^\#$. 

We have just proved that the shape measure of the outer boundary of $\eta'$ is the SLE$_\kappa$ shape measure $\nu_\kappa^\#$. This implies that $\mu_\kappa=K\nu_\kappa$ for some global multiplicative constant $K>0$. Our choice of normalisation ensures that the distribution of the $\log$-conformal radius of $\eta$ viewed from $\infty$ is $\frac{\d\tilde{c}}{\sqrt{\pi}}$. This coincides with our choice of normalisation for $\nu_\kappa$, so indeed $K=1$.
\end{proof}

\section{Related and future works, open questions}\label{sec:future}

    \subsection{On Airault--Malliavin measures, Kontsevich--Suhov measures, and SLE}\label{subsec:comparison}
In this section, we review the differences and similarities between Airault--Malliavin's unitarising measures \cite{AiraultMalliavin01}, and Kontsevich--Suhov's restriction measures \cite{KontsevichSuhov07}. The distinction is subtle, and we hope that the following discussion will clarify certain aspects. We will not attempt to review the rich developments that followed these works; instead, we refer to the introduction of \cite{GQW24} and references therein.

First, we need to remind the reader of conformal welding. Let $f\in\cE$, and set $\eta:=f(\S^1)$ and $g:\D^*\to\mathrm{ext}(\eta)$ conformal with $g(\infty)=\infty$, $g'(\infty)>0$. Since $f,g$ extend to homeomorphisms on the closure of their domains, we can set $h:=g^{-1}\circ f|_{\S^1}\in\mathrm{Homeo}(\S^1)$. Postcomposing $h$ with a rotation does not change the welding curve, and this leads to an embedding $\cE\to\S^1\backslash\mathrm{Homeo}(\S^1)$. The converse problem of associating a curve to a homeomorphism is called \emph{conformal welding}, and does not admit a solution (let alone unique) in general (see \cite{Bishop} for a comprehensive introduction to conformal welding). The homeomorphism $h$ is called the \emph{welding homeomorphism} of $\eta$ (or $f$). 

This leads to a left (resp. right) action of the group $\mathrm{Diff}^\omega(\S^1)$ on $\cE$: given $f\in\cE$ and $\tilde{\Phi}\in\mathrm{Diff}^\omega(\S^1)$, define $\tilde{\Phi}\cdot f$ (resp. $f\cdot\tilde{\Phi}$) to be the element of $\cE$ with welding homeomorphism $\tilde{\Phi}\circ h$ (resp. $h\circ\tilde{\Phi}$). This action was considered by Kirillov in the case of smooth Jordan curves, and he gave explicit formulas for the infinitesimal action \cite[(7)]{Kirillov98}. One can then consider the Lie derivatives along the invariant vector fields, acting as differential operators on the space of test functions on $\cE$. This defines a representation of the Witt algebra, which can be turned into a Virasoro representation using Neretin polynomials \cite[Section~6]{Kirillov98}. In Airault--Malliavin's terminology, a measure on $\cE$ is a \emph{unitarising measure (for the Virasoro algebra)} if Kirillov's representation is unitary with respect to the $L^2$-inner-product of the measure (equivalently, the Virasoro generators satisfy suitable integration by parts formulas) \cite[Section 2.3]{AiraultMalliavin01}.

Crucially, Kirillov's representation is \emph{not} the same as the one considered in Section \ref{sec:rep} of this paper. The simplest way to see this is that the infinitesimal action considered in this work does not exponentiate to an action of a group. Instead, the notion of conformal restriction refers to the variation of the measure when we perturb the curve by a local conformal transformation. Visually, Kirillov's action takes place on the left-hand-side of Figure \ref{fig:setup}, while the restriction action is generated on the right-hand-side, and there is a non trivial change of variables between the two (see below). As pointed out in \cite[Section 2.5.2]{KontsevichSuhov07}, the hermiticity relations \cite[(2.3.4)]{AiraultMalliavin01} are strikingly similar to the infinitesimal restriction property of \cite[Definition 2.4]{KontsevichSuhov07}, but it should be stressed again that these statements refer to two different families of operators. Moreover, the hermiticity relations of \cite{KontsevichSuhov07} should be understood with respect to the Hermitian form $\cQ$ introduced in this paper (and in \cite{GQW24}).

Recently, we studied the Kirillov action on the $L^2$-space of the SLE$_\kappa$ loop measure: remarkably, it satisfies \cite[(2.3.4)]{AiraultMalliavin01} with the central charge $c_\rL:=26-c_\rM\geq25$ \cite[Theorem 1.1]{BJ25}. The proof relies on Theorem \ref{T:ibp}, and what amounts to the computation of the ``Jacobian" of the change of variables between the two sets of invariant coordinates induced by the above Virasoro representations. This change of variables is based on the (infinitesimal) solution theory of the Beltrami equation, which is nothing but the resolvent of the Cauchy--Riemann operator acting on $(-1,0)$-differentials. \textit{A posteriori}, the number $26=2\times 13$ can be interpreted in light of Takhtajan--Zograf's local index theorem for this operator \cite[Theorem 2\footnote{Put $n=-1$ in this statement.}]{TakhtajanZograf}. Additionally, the value of the central charge reconciles Airault--Malliavin's assumption (unitarity requires $c_\rL>0$,\footnote{In fact, $c_\rL\geq1$ is more appropriate due to standard results from representation theory \cite[Proposition~8.2]{KacRaina_Bombay}, and $c_\rL\geq25$ is even more natural in light of the coupling with Liouville CFT \cite{sheffield2016,AHS20,BJ25}.} see \cite[(1.2.8)]{AiraultMalliavin01} and below), and Kontsevich--Suhov's assumption $c_\rM\leq 1$ \cite[Conjecture~1]{KontsevichSuhov07} (motivated by the restrition property of SLE). To the best of our knowledge, \cite{BJ25} is the first proof that the SLE loop measure (hence any Kontsevich--Suhov measure) is an Airault--Malliavin measure (hence that such measures exist\footnote{\cite[Theorem 2.2]{AiraultMalliavinThalmaier02} is a non-existence statement for \emph{probability} measures; of course, any unitarising measure is invariant under $\mathrm{PSL}_2(\R)$ and must have infinite mass.}).

    \subsection{Multiple SLE loop measure}
Let $N\in\Z_{>0}$. Let $\cJ_{\hat\C,N}$ be the set of $N$-tuples of Jordan curves $\boldeta=(\eta_1,...,\eta_N)$ such that $\hat\C\setminus\cup_{j=1}^N\eta_j$ consists of $N$ simply connected components, and one $N$-connected component. By convention, the \emph{interior} of $\eta_j$ is the simply connected component bounded by $\eta_j$. One can define a conformally invariant measure on $\cJ_{\hat\C,N}$ satisfying the conformal restriction property, by setting
\[\d\nu_{\hat\C,N}(\boldeta)=\ind_{\{\boldeta\in\cJ_{\hat\C,N}\}}e^{\frac{c_\rM}{2}\Lambda^*(\boldeta)}\d\nu_{\hat\C}^{\otimes N}(\boldeta),\]
where $\Lambda^*(\boldeta)$ is a natural quantity (expressed as a mass of Brownian loops) which generalises the case $N=2$ studied in \cite{LuoMaibach}; see \cite{BJSW25}.

Let $z_1,...,z_N$ be pairwise distinct points in $\hat\C$, and restrict the above measure to the event that $\eta_j$ surrounds $z_j$ and no other marked point. Let $c_j$ be the $\log$-conformal radius of $\eta_j$ viewed from $z_j$, and define $\langle\Psi_{\lambda_1}(z_1)\cdots\Psi_{\lambda_N}(z_N)\rangle_{\hat\C,N}:=\int\prod_{j=1}^Ne^{-2\lambda_jc_j}\d\nu_{\hat\C,N}$, provided the integral converges. 
\begin{open}
    Is there an analytic formula for $C(\lambda_1,\lambda_2,\lambda_3):=\langle\Psi_{\lambda_1}(0)\Psi_{\lambda_2}(1)\Psi_{\lambda_3}(\infty)\rangle_{\hat\C,3}$?
\end{open}
This question is natural from the point of view of conformal field theory, where the three-point correlation function is called the \emph{structure constant}. In Liouville CFT, the structure constant is known as the DOZZ formula \cite{DornOtto94,Zamolodchikov96,Teschner_dozz,KRV_DOZZ}. Another motivation comes from the recent breakthrough \cite{AngSun21_CLE}, showing that the three-point correlation function of the conformal loop ensemble is the imaginary DOZZ formula \cite[Theorem 8.3]{AngSun21_CLE}, confirming predictions from physics \cite{IJS16}.

In connection with CFT, it would be interesting to prove that the correlation functions defined above satisfy the celebrated Belavin--Polyakov--Zamolodchikov (BPZ) equations \cite{BPZ84} when one or more $\lambda_j$ lies on the Kac table. One obstruction is that the integral defining the correlations does not converge for these values.
\begin{problem}
    Prove that the correlation function $\langle\prod_{j=1}^N\Psi_{\lambda_j}(z_j)\rangle$ admits a meromorphic continuation in the $\lambda_j$'s. For $\lambda_j$ in the Kac table, show that its continuation satisfies a BPZ equation.
\end{problem}
We expect these BPZ equations to hold due to the structure of degenerate modules (see Theorem~\ref{thm:module_structure} and Corollary~\ref{C:level2}) and the conformal Ward identities expressing the conformal invariance of the correlation functions (the latter being part of an ongoing work with Xin Sun and Baojun Wu.

The last question related to CFT is whether one can express the four-point correlation in terms of the three-point correlation, i.e. write a conformal bootstrap formula. One idea is to introduce a fifth SLE loop $\eta_5$ disconnecting $\eta_1,\eta_2$ from $\eta_3,\eta_4$, write $\langle\prod_{j=1}^4\Psi_{\lambda_j}(z_j)\rangle_{\hat\C,4}$ as the $\cQ$-inner-product over the law of $\eta_5$, and express this inner-product using Theorem \ref{T:spectral}.
\begin{problem}
    Write a conformal boostrap formula along the lines above.
\end{problem}

 By Koebe's theorem, one can map the $N$-connected $D_{\boldeta}:=\hat\C\setminus\cup_{j=1}^N\overline{\mathrm{int}(\eta_j)}$ to a circle domain (which is unique modulo M\"obius transformations), and this circle domain encodes the modulus of $D_{\boldeta}$. In a work in progress with Xin Sun and Baojun Wu \cite{BJSW25}, we compute the law of this modulus, which turns out to be related to the Belavin--Knizhnik measure from bosonic string theory \cite{Polyakov81,BelavinKnizhnik}. We also prove a coupling statement of the SLE with the Gaussian free field, generalising the conformal welding of random surfaces \cite{sheffield2016,AHS20,BJ25} to multiply connected cases.

\appendix

\section{Representation theory of the Virasoro algebra}\label{app:virasoro}
This appendix reviews some standard facts from the representation theory of the Virasoro algebra. The main reference is \cite{KacRaina_Bombay}.

The \emph{Virasoro algebra} is the infinite dimensional Lie algebra with generators $((L_n)_{n\in\Z},\bc)$ and relations 
\[[L_n,L_m]=(n-m)L_{n+m}+\frac{n^3-n}{12}\delta_{n,-m}\bc,\qquad[\bc,L_n]=0.\]
In all representations, we assume that $\bc$ acts by multiplication by a scalar, and we slightly abuse notation by identifying $\bc$ with the operator $\bc\mathrm{Id}$ where $\bc\in\C$ called the \emph{central charge}.\footnote{In the main text, the central charge is denoted $c_\rM$, but we want to distinguish between the abstract modules and the representation under study. The bold font is used to avoid any confusion with the coordinate $c$ on $\cJ_{0,\infty}$.}

A \emph{highest-weight representation} is a Virasoro module $V$ such that $V=\mathrm{span}\{e_\bk=L_{-\bk}\cdot e_\emptyset|\,\bk\in\cT\}$ for some vector $e=e_\emptyset\in V$ called the \emph{highest-weight state}. If $L_0\cdot e=\lambda e$ for some $\lambda\in\C$, the representation is called of \emph{weight} $\lambda$. Every highest-weight representation has a grading $V=\oplus_{N\in\N}V^N$, with $V^N=\mathrm{span}\{L_{-\bk}\cdot e|\,\bk\in\cT_N\}$. 

A \emph{Verma module} is a highest-weight representation such that all vectors $L_{-\bk}\cdot e$ are linearly independent. For each $\bc\in\C$ and $\lambda\in\C$, there exists a unique Verma module of central charge $c$ and weight $\lambda$, which we denote by $M_{\bc,\lambda}$. Vectors in $M_{\bc,\lambda}^N$ for $N>0$ are called the \emph{descendants} of $e_\emptyset$. Every highest-weight representation of central charge $\bc$ and weight $\lambda$ is a quotient of $M_{\bc,\lambda}$ by a submodule. The Verma module $M_{\bc,\lambda}$ has a unique maximal proper (possibly trivial) submodule $I_{\bc,\lambda}$, and the quotient 
\[V_{\bc,\lambda}:=M_{\bc,\lambda}/I_{\bc,\lambda}\]
 is irreducible. It is called the \emph{irreducible quotient} (of $M_{\bc,\lambda}$ by the maximal proper submodule). A vector $\chi\in M_{\bc,\lambda}$ is \emph{singular} if $L_n\cdot\chi=0$ for all $n\geq1$, and $\chi$ is a descendant of $e_\emptyset$. A singular vector at some level $N$ generates its own highest-weight representation (isomorphic to $M_{\bc,\lambda+N}$); in particular $M_{\bc,\lambda}$ is reducible if it has at least one singular vector. Conversely, every proper submodule is generated by singular vectors (and their descendants). In particular, $I_{\bc,\lambda}$ is the space of all singular vectors in $M_{\bc,\lambda}$ and their descendants. 

The \emph{Shapovalov form} is the unique Hermitian form $\langle\cdot,\cdot\rangle_{c,\lambda}$ on $M_{\bc,\lambda}$ such that $\langle e,e\rangle_{\bc,\lambda}=1$, and
\[\langle L_n\cdot v,w\rangle_{\bc,\lambda}=\langle v,L_{-n}\cdot w\rangle_{\bc,\lambda},\qquad\forall v,w\in M_{\bc,\lambda},\,\forall n\in\Z.\]
The coefficients of the Gram matrix are denoted $\rB_\lambda(\bk,\bk')=\langle L_{-\bk}\cdot e_\emptyset,L_{-\bk'}\cdot e_\emptyset\rangle_{\bc,\lambda}$, for all $\bk,\bk'\in\cT$. Equivalently, one has $L_{\bk'}L_{-\bk}\cdot e_\emptyset=\rB_\lambda(\bk,\bk')e_\emptyset$ (recall notation \eqref{eq:string}). These coefficients are universal: they depend only on the central charge and the conformal weight. The maximal proper submodule $I_{\bc,\lambda}$ is the kernel of the Shapovalov form, i.e.
\[I_{\bc,\lambda}=\{v\in M_{\bc,\lambda}|\,\forall w\in M_{\bc,\lambda},\,\langle v,w\rangle_{\bc,\lambda}=0\}.\]

The \emph{dual Verma module} $M^\vee_{\bc,\lambda}$ is the space of linear forms on $M_{\bc,\lambda}$ equipped with the dual representation $(L_n^\vee)_{n\in\Z}$, i.e.\footnote{In principle, the dual representation should rather be $L_n^\vee\cdot\alpha(v):=-\alpha(L_n\cdot v)$ but this is equivalent due to the automorphism $L_n\mapsto-L_{-n}$.}
\[L_n^\vee\cdot\alpha(v):=\alpha(L_{-n}\cdot v),\qquad\forall\alpha\in M_{\bc,\lambda}^\vee,\,\forall v\in M_{\bc,\lambda},\,\forall n\in\Z.\]
The space $M^\vee_{\bc,\lambda}$ has a distinguished vector $e^\vee$: the dual to the highest-weight vector in $M_{\bc,\lambda}$.
The Verma module and its dual are not isomorphic as Virasoro modules if $M_{\bc,\lambda}$ is reducible ($M_{\bc,\lambda}^\vee$ is not even a highest-weight representation in this case), but they are always isomorphic as graded vector spaces. $V_{\bc,\lambda}^\vee:=\mathrm{span}\{L_{-\bk}^\vee e^\vee|\,\bk\in\cT\}$. By definition, $(V_{\bc,\lambda}^\vee,(L_n^\vee)_{n\in\Z})$ is a highest-weight representation of weight $\lambda$.

\begin{lemma}\label{lem:irreducible_quotient}
We have $V_{\bc,\lambda}\simeq V_{\bc,\lambda}^\vee$ as highest-weight representations.
\end{lemma}

\begin{proof}
Since $V_{\bc,\lambda}=M_{\bc,\lambda}/I_{\bc,\lambda}$, we have a linear isomorphism
\[V_{\bc,\lambda}\simeq\{\alpha\in M_{\bc,\lambda}^\vee|\,\forall v\in I_{\bc,\lambda},\,\alpha(v)=0.\}\]
In the rest of the proof, we abuse notations by writing $V_{\bc,\lambda}$ for this subspace of $M_{\bc,\lambda}^\vee$.

For all non-empty partitions $\bk\in\cT$ and all $\chi\in I_{\bc,\lambda}$, we have $L_{-\bk}^\vee e^\vee(\chi)=e^\vee(L_{\bk}\chi)=0$, where the last equality is an easy consequence of the fact that the kernel of the Shapovalov form is generated by singular vectors and their descendants. This proves that $V_{\bc,\lambda}^\vee\subset V_{\bc,\lambda}$. 

On the other hand, $V_{\bc,\lambda}^\vee$ is a highest-weight representation of weight $\lambda$, we have $\dim V_{\bc,\lambda}^N\leq\dim V_{\bc,\lambda}^{\vee N}$ for all $N\in\N$. It follows that $\dim V_{\bc,\lambda}^{\vee N}=\dim V_{\bc,\lambda}^N$ for all $N\in\N$, so $V_{\bc,\lambda}$ is isomorphic to the irreducible quotient by uniqueness of the maximal proper submodule.
\end{proof}

\section{The SLE loop measure is Borel}\label{app:SLE}

Before stating the result, we recall some facts concerning topologies and their associated Borel $\sigma$-algebras on the space $\cJ_{0,\infty}$.

\begin{itemize}
    \item The Hausdorff distance on the set of non-empty, compact subsets of $\hat\C$ is given by
\[d_\cH(X,Y)=\max\big\lbrace\sup_{x\in X}\,d(x,Y),\,\sup_{y\in Y}\,d(y,X)\big\rbrace,\]
where $d$ is the distance on $\hat\C$ induced by (say) the round metric on $\hat \C$. The Hausdorff metric restricts to a metric on $\cJ_{0,\infty}$. We will denote the associated topology $\tau_\cH$ and its Borel $\sigma$-algebra $\cB(\tau_\cH)$.
\item We endow the space $\R \times \cE$ with the topology of local uniform convergence on compact subsets of $\D$. We denote by $\tau_{\cC}$ the pullback of this topology to the space $\cJ_{0,\infty}$ by the map $\eta \in \cJ_{0,\infty} \mapsto (c_\eta,f_\eta) \in\R \times\cE$, and by $\cB(\tau_\cC)$ the associated Borel $\sigma$-algebra.
\end{itemize}
A theorem of Carathéodory gives a geometric description of convergence in the space $(\cJ_{0,\infty},\tau_{\cC})$, sometimes referred to as kernel convergence. See \cite[Theorem 1.8]{Pommerenke75}. In particular, $\tau_\cC \subset \tau_\cH$. The topologies are however different as one can see by considering the outer boundary of $D(0,1+1/n) \cup D(2,1+1/n)$, $n\ge 1$: they converge in $(\cJ_{0,\infty},\tau_{\cC})$ to $\S^1$, but do not converge in $(\cJ_{0,\infty},\tau_{\cH})$.\footnote{We thank an anonymous referee for this counterexample.}
\begin{lemma}
    Although $\tau_\cC \subsetneq \tau_\cH$, the $\sigma$-algebras $\cB(\tau_\cC)$ and $\cB(\tau_\cH)$ coincide.
\end{lemma}
\begin{proof}
    Let $F: \eta \in (\cJ_{0,\infty},\tau_\cH) \to (c_\eta,f_\eta) \in \R \times \cE$, where $\cE$ is equipped with the topology of uniform convergence on compact subsets of $\D$.
    Since $\tau_\cC\subset\tau_\cH$, $F$ is continuous and in particular measurable. Moreover, $F$ is a bijection. Since the space $\R \times \cE$ and the space of compact subsets of $\hat\C$ equipped with the Hausdorff distance are Polish, and since $\cJ_{0,\infty}$ is a measurable subset of the latter, the Lusin--Souslin theorem (see \cite[Theorem 15.1]{zbMATH00722611}) implies that $F$ is a Borel isomorphism. The $\sigma$ algebras $\cB(\tau_\cC)$ and $\cB(\tau_\cH)$ thus coincide.
\end{proof}

The main result of this appendix reads as follows.

\begin{proposition}\label{prop:sle_is_borel}
The SLE loop measure is a Borel measure on the space of simple Jordan curves equipped with the Hausdorff metric.
\end{proposition}

To prove this proposition, we will first start by proving a similar statement for chordal SLE. For the reader's convenience, we start by a short reminder on this theory. We refer to \cite{RohdeSchramm05} for a more complete introduction, and to \cite[Chapter 6]{Pommerenke75} for some background on the original analytic (deterministic) theory due to Loewner.

A \emph{Jordan slit} is an (unparametrised) simple Jordan arc in $\H$ with an endpoint at $0\in\del\H$. We denote by $\mathcal{K}$ the space of Jordan slits, equipped with the Hausdorff topology. Let also $\overline{\mathcal{K}}$ be the space of all simple Jordan arcs with an endpoint at 0 and an endpoint in $\H\cup\{\infty\}$. This space is also equipped with the Hausdorff metric, and $\mathcal{K}$ is an open subset of $\overline{\mathcal{K}}$.

Each Jordan slit has a unique parametrisation $\eta$ on some interval $[0,T]$ such that the following holds. For all $t\in[0,T]$, let $f_t:\H\setminus\eta[0,t]\to\H$ be the conformal map with the hydrodynamical normalisation (i.e. $f_t(z)=z+O(1/z)$ as $z\to\infty$). Then, $f_t(z)=z+\frac{2t}{z}+o(1/z)$ as $z\to\infty$. One says that $\eta$ is parametrised by half-plane capacity. In this case, $f_t$ satisfies the differential equation $\del_tf_t=\frac{2}{f_t-W_t}$, where $W\in\cC^0([0,T];\R)$ is called the \emph{driving function}. Conversely, every driving function $W\in\cC^0([0,T];\R)$ generates a family of compact hulls in the upper-half plane (but not necessarily a curve). We denote by $\mathfrak{L}:\cC^0(\R_+)\to\overline{\mathcal{K}}$ the Loewner map sending driving functions to Jordan slits.

Chordal SLE$_\kappa$ (run until capacity $T>0$) is the probability measure on $\mathcal{K}$ where the driving function is a Brownian motion in $[0,T]$ with speed $\kappa>0$. We denote by $\mathrm{SLE}^T_\kappa$ this measure, so $
\mathrm{SLE}_\kappa^1=\mathfrak{L}_*\P$
 where $\P$ is the law of $\sqrt{\kappa}$ times standard real Brownian motion on $[0,1]$, a Borel probability measure on $\cC^0([0,1])$. It is well-known that for $\kappa\in(0,4]$, the solution to the Loewner equation is almost surely a simple Jordan slit.

 \begin{proposition}\label{prop:cSLE_is_borel}
    For all $T>0$, $\mathrm{SLE}_\kappa^T$ is a Borel measure on $\mathcal{K}$. Moreover, $\mathrm{SLE}_\kappa^T$ converges weakly to a (Borel) probability measure on $\overline{\mathcal{K}}$ giving full mass to Jordan curves from 0 to $\infty$ in $\H$. 
 \end{proposition}

\begin{proof}
We start by showing that $\mathrm{SLE}^1_\kappa$ is Borel.
  For each $z\in\H$ and $r>0$, define the open set $E_{z,r}:=\{\eta\text{ visits }D(z,r)\}\subset\mathcal{K}$. We can rewrite this event as
  \begin{equation}\label{E:event_SLE}
  E_{z,r}=\bigcup_{\alpha \in (0,1)\cap\Q} \bigcap_{n \ge 1} \{\exists q\in\Q^2\cap D(z,\alpha r) :\,\Im(f_1(q))<1/n,\, \Im(q)\ge 1-\alpha\}.
  \end{equation}
    Indeed, if $\eta$ visits some point $w \in D(z,r) \cap \mathbb{H}$, then a sequence of rational points $(q_n)_{n \ge 1}$ converging to $w$ satisfies $\Im (f_1(q_n)) \to 0$ and $\Im(q_n) \to \Im(q)>0$ as $n \to \infty$. Conversely, let $\alpha \in (0,1)$ and assume that for all $n \ge 1$, there exists $q_n \in \Q^2 \cap D(z,\alpha r)$ such that $\Im(f_1(q_n))<1/n$ and $\Im(q_n)\ge 1-\alpha$. By extracting a subsequence if necessary, $(q_n)$ converges to some $q \in \overline{D(z,\alpha r)} \cap \mathbb{H} \subset D(z,r)\cap \mathbb{H}$ which satisfies $\Im(f_1(q))=0$. Since $\eta$ is a Jordan slit, this implies that $\eta$ hits $q$ (non simple curves could surround $q$ without touching it), showing that the right hand side of \eqref{E:event_SLE} is included in the left hand side of \eqref{E:event_SLE}.
  Because for all $q \in \mathbb{H}$, $f_1(q)$ is measurable with respect to the Brownian motion driving the Loewner equation, the identity \eqref{E:event_SLE} implies that $E_{z,r}$ is $\mathrm{SLE}_\kappa^1$-measurable.

  Now, let $\eta_0,\eta\in\mathcal{K}$ and $r>0$. From the previous paragraph, for all $z\in\eta$, the event $\{d(z,\eta_0)<r\}$ is measurable (with $d$ the usual distance). Hence, the event $\{\sup_{z\in\eta}\,d(z,\eta_0)<r\}$ is also measurable. The same is true when exchanging the roles of $\eta_0$ and $\eta$, so by definition of the Hausdorff distance $d_\cH$, the event 
  $\{d_\cH(\eta_0,\eta)<r\}$ is measurable. This shows that metric balls in $\mathcal{K}$ are measurable, so $\mathrm{SLE}_\kappa^1$ is Borel. 

  By scaling, we easily deduce that $\mathrm{SLE}_\kappa^T$ is a Borel measure on $\mathcal{K}$ (or $\overline{\mathcal{K}}$) for all $T>0$. Chordal SLE$_\kappa$ from 0 to $\infty$ is then Borel on $\overline{\mathcal{K}}$ as a weak limit of Borel measures.
 \end{proof}

\begin{proof}[Proof of Proposition \ref{prop:sle_is_borel}]
The proof of Proposition \ref{prop:cSLE_is_borel} also applies to the setup of radial SLE$_\kappa(\rho)$ \cite[Section 2.1.2]{Miller17_ig4}. The Loewner equation in this setting takes a different form and the driving function is no longer a Brownian motion but a continuous semimartingale, locally absolutely continuous with respect to Brownian motion.
Whole-plane SLE$_\kappa(\rho)$ is a measure on Jordan arcs joining two distinct points in the Riemann sphere, which can be constructed as a two-sided version of radial SLE \cite[Section 2.1.3]{Miller17_ig4}. Hence, it is a Borel measure on the space of Jordan arcs as a weak limit of Borel measures.

Now, we recall the construction of the SLE loop measure \cite{Zhan21_SLEloop} (see also \cite[Section~2.3]{AngSun21_CLE}). Let $p,q\in\hat{\C}$ distinct. Sample a whole-plane SLE$_\kappa(2)$ from $p$ to $q$. Conditionally on this curve, sample an independent chordal SLE$_\kappa$ from $q$ to $p$ in the complement of the curve. Integrate over $p,q\in \hat{\C}$ with respect to the kernel appearing in \cite[(4.5)]{Zhan21_SLEloop} (see also \cite[(2.9)]{AngSun21_CLE}). All the measures appearing in this process are Borel (with respect to the Hausdorff topology), and all the operations are continuous with respect to the weak topology. Hence, the SLE$_\kappa$ loop measure is Borel.
\end{proof}

\section{A formal computation}\label{Appendix:formal}

In this section, we make a formal computation explaining how one could derive the integration by parts formula \eqref{E:intro_IBP} from the path integral
\[ 
\text{``}~\d \nu_\kappa(\eta) = \exp \Big( \frac{c_\rM}{24} I^L(\eta) \Big) \d \nu_{8/3}(\eta),~\text{''}
\]
where $I^L(\eta)$ is the Loewner energy, or universal Liouville action, of $\eta$. 
As already alluded to, this path integral is inspired by the large deviation result of \cite{carfagnini2023onsager}, but
the fact that this non-rigorous computation yields the correct integration by parts formula also gives an \textit{a posteriori} justification to the above path integral.
Let $F, G \in \cC_\mathrm{comp}$ be test functions, $\rv \in \C(z)\del_z$ be a vector field and $(\phi_t)_t$ be the associated flow, defined for small complex values of $t$. Recall that for all $\eta \in \cJ_{0,\infty}$,
\[ 
F(\phi_t(\eta)) = F(\eta) + t\cL_\rv F(\eta) + \bar t \bar\cL_\rv F(\eta) + o(t).
\]
We have
\begin{align*}
    \text{``}~\int F(\phi_t(\eta)) \overline{G(\eta)} \d \nu_\kappa(\eta)
    & = \int F(\phi_t(\eta)) \overline{G(\eta)} \exp \Big( \frac{c_\rM}{24} I^L(\eta) \Big) \d \nu_{8/3}(\eta) \\
    & = \int F(\eta) \overline{G(\phi_t^{-1}(\eta))} \exp \Big( \frac{c_\rM}{24} I^L(\phi_t^{-1}(\eta)) \Big) \d \nu_{8/3}(\eta). ~\text{''}
\end{align*}
In the last equality, we made a change of variable and used the invariance of the SLE$_{8/3}$-loop measure under quasiconformal maps. To first order, $\phi_t^{-1}$ is the flow associated to $-\rv$. Identifying the $t$-coefficients in the above display, we thus get that
\begin{align*}
    \text{``}~\int \cL_\rv F(\eta) \overline{G(\eta)} \d\nu_\kappa(\eta)
    & = - \int F(\eta) \Big( \overline{\bar\cL_\rv G(\eta)} + \overline{G(\eta)} \frac{c_\rM}{24} \cL_\rv I^L(\eta) \Big) \exp \Big( \frac{c_\rM}{24} I^L(\eta) \Big) \d \nu_{8/3}(\eta) \\
    & = - \int F(\eta) \Big( \overline{\bar\cL_\rv G(\eta)} + \overline{G(\eta)} \frac{c_\rM}{24} \cL_\rv I^L(\eta) \Big) \d \nu_\kappa(\eta). ~\text{''}
\end{align*}
By \cite[Chapter 2, Theorem 3.8]{TakhtajanTeo06},
$\cL_\rv I^L = 2\tilde\vartheta(\rv) - 2\vartheta(\rv)$. Altogether,
\begin{align*}
    \text{``}~\int \cL_\rv F(\eta) \overline{G(\eta)} \d\nu_\kappa(\eta)
    = - \int F(\eta) \Big( \overline{\bar\cL_\rv G(\eta)} + \frac{c_\rM}{12} \overline{G(\eta)} (\tilde\vartheta_\eta(\rv) - \vartheta_\eta(\rv)) \Big) \d \nu_\kappa(\eta) ~\text{''}
\end{align*}
which coincides with the first identity in \eqref{E:intro_IBP}.

\bibliographystyle{alpha}
\bibliography{bpz}
\end{document}